\documentclass[11pt, reqno]{amsart}

\usepackage{fullpage}
\usepackage{amsmath}
\usepackage{amsthm}
\usepackage{amsfonts}
\usepackage{amssymb}

\newtheorem{lemma}{Lemma}[section]
\newtheorem{proposition}[lemma]{Proposition}
\newtheorem{theorem}[lemma]{Theorem}
\newtheorem{corollary}[lemma]{Corollary}
\theoremstyle{remark}
\newtheorem{remark}[lemma]{Remark}

\theoremstyle{definition}
\newtheorem{definition}[lemma]{Definition}

\numberwithin{equation}{section}

\def\C{{\mathbb C}}

\def\sign{{\rm sign}}

\def\Re{{\,\rm Re}}

\def\Ai{{\rm Ai}}
\def\Bi{{\rm Bi}}

\title{Semiclassical low energy scattering for one--dimensional Schr\"odinger operators with exponentially decaying potentials}

\author{Ovidiu Costin}
\address{The Ohio State University, Department of Mathematics, 100 Math Tower, 231 West 18th Avenue,
Columbus, OH 43210-1174, U.S.A.}
\email{costin@math.ohio-state.edu}
\thanks{The first author was partly supported by the National Science Foundation
DMS-0807266 and a Guggenheim fellowship.}

\author{Roland Donninger}
\address{\'Ecole Polytechnique F\'ed\'erale de Lausanne, 
MA B1 487, Station 8, CH-1015 Lausanne, Switzerland }
\email{roland.donninger@epfl.ch}
\thanks{The second author 
was partly supported by an Erwin Schr\"odinger fellowship of the
FWF (Austrian Science Fund), Project No.~J2843.}

\author{Wilhelm Schlag}
\address{University of Chicago, Department of Mathematics,
5734 South University Avenue, Chicago, IL 60637, U.S.A.}
\email{schlag@math.uchicago.edu}
\thanks{The third author was partly supported by the National
Science Foundation DMS-0617854 and a Guggenheim fellowship.}

\author{Saleh Tanveer}
\address{The Ohio State University, Department of Mathematics, 100 Math Tower, 231 West 18th Avenue,
Columbus, OH 43210-1174, U.S.A.}
\email{tanveer@math.ohio-state.edu}
\thanks{The fourth author was partly supported by the National Science Foundation
DMS-0807266}

\begin{document}
\begin{abstract}
 We consider semiclassical Schr\"odinger operators on the real line of the form 
$$H(\hbar)=-\hbar^2 \frac{d^2}{dx^2}+V(\cdot;\hbar)$$ 
with $\hbar>0$ small. The potential $V$ is assumed to be smooth, positive and exponentially decaying towards infinity. We establish semiclassical global representations of Jost solutions $f_\pm(\cdot,E;\hbar)$ with error terms that are uniformly controlled for small $E$ and $\hbar$, and construct the scattering matrix as well as the semiclassical spectral measure associated to $H(\hbar)$. 
 This is crucial in order to obtain decay bounds for the corresponding wave and Schr\"odinger flows.
 As an application we consider the wave equation on a Schwarzschild background for large angular momenta $\ell$ where the role of the small parameter $\hbar$ is played by $\ell^{-1}$.
 It follows from the  results in this paper and \cite{DSS2}, that the decay bounds obtained in \cite{DSS1}, \cite{DS} for individual 
 angular momenta $\ell$ can be summed to yield the sharp $t^{-3}$ decay for data without symmetry assumptions. 
\end{abstract}

\maketitle

\section{Introduction}
In this paper we study Schr\"odinger operators $H(\hbar)$ of the form
\begin{equation}
\label{eq:SO}
H(\hbar)f(x):=-\hbar^2 f''(x)+V(x;\hbar)f(x) 
\end{equation}
for $x \in \mathbb{R}$.
The potential $V$ is assumed to be positive, smooth \footnote{One can, of course, obtain similar results under the weaker assumption that $V$ has only a finite number of derivatives. However, for the sake of simplicity we do not pursue this issue here.} and exponentially decaying as $|x| \to \infty$.
As indicated by our notation, we also allow $V$ to depend on the semiclassical parameter $\hbar$, 
but only in a very mild way, i.e., all assumptions on $V$ are assumed to hold 
uniformly in small $\hbar$.
This will be important for our main application, the wave equation on a Schwarzschild background.
More precisely, we require $V(\cdot;\hbar)\in C^\infty(\mathbb{R})$ 
to be of the form $V(x;\hbar)=V_\pm(e^{-|x|};\hbar)$
for $\pm x \geq a$
where $a>0$ is some (large) constant and
$V_\pm(\cdot;\hbar)$ are smooth functions with $V_\pm(0;\hbar)=0$.
It follows from the standard theory (see e.g., \cite{teschl}, \cite{weidmann2}) that under the above assumptions the operator $H(\hbar)$ has a 
self--adjoint realization on $L^2(\mathbb{R})$ with domain $\mathcal{D}(H(\hbar))=H^2(\mathbb{R})$
and its spectrum is purely absolutely continuous with $\sigma(H(\hbar))=[0,\infty)$.
The present paper is devoted to the study of low energy ($E \to 0+$) scattering in the semiclassical limit $\hbar \to 0+$, i.e., 
we consider the low energy and semiclassical limits \emph{simultaneously}.
To be more precise, we are primarily interested in global semiclassical  representations of the 
Jost solutions $f_\pm(x,E;\hbar)$ and their derivatives with respect to $E$ that hold \emph{uniformly} in small 
$E$ and $\hbar$.
As a consequence, we obtain uniform bounds on relevant quantities that are constructed based on the Jost solutions
such as the reflection and transmission amplitudes as well as the spectral measure.

Our main motivation for investigating this problem comes from the large angular momentum behavior
of solutions to the Regge--Wheeler equation which describes perturbations of a Schwarzschild black
hole in general relativity \cite{RW}.
The Regge--Wheeler potential decays exponentially towards the event horizon of the black hole and thus,
one is naturally led to a Schr\"odinger operator of the form \eqref{eq:SO} where $\hbar^{-1}$ 
corresponds to the angular momentum $\ell$.
Thus, in the limit $\ell \to \infty$ the problem becomes semiclassical.
Consequently, the present paper is part of the series \cite{DSS1}, \cite{DSS2}, \cite{DS} devoted
to the study of wave evolution on a Schwarzschild background.
However, we believe that our results and techniques are of independent interest and therefore we
decided to embed our work in the broader context of rigorous quantum mechanical scattering theory in
the semiclassical limit.
Scattering theory is by now a classical subject and we cannot possibly do justice to the  vast  literature in that area.
We confine our discussion of the 
existing literature
to selected works devoted to the one--dimensional case which are related to our problem.

A standard reference for scattering on the real line is 
\cite{DS} which provides many of the by now classical results on the subject.
Low energy scattering for one--dimensional Schr\"odinger operators is studied in, e.g., \cite{Klaus88},
\cite{Newton86}, \cite{BGW85}, \cite{Yafaev82} whereas the semiclassical limit is considered in
\cite{Ramond96}. However, we find that the available results are not suitable for our purposes.
The only paper we are aware of that studies the double asymptotics $E,\hbar \to 0+$ in a way similar to us 
is \cite{schlag4} which deals with inverse square potentials.
Since the Regge--Wheeler potential exhibits inverse square decay towards spatial
infinity, \cite{schlag4} is used to deal with the far field region in our main application which we now
describe in more detail.

\subsection{The problem of wave evolution on a Schwarzschild background}

One of the major open problems in mathematical general relativity is a proof of the stability of
black holes as described by the Kerr solution.
As a first step in this direction one considers the behavior of linearized perturbations of
Schwarzschild black holes.
It is well-known that such perturbations are described by the Regge--Wheeler equation
\begin{equation} 
\label{eq:RW}
\partial_t^2\psi_\ell(t,x)-\partial_x^2\psi_\ell(t,x)+V_\ell(x)\psi_\ell(t,x)=0,\quad (t,x) \in 
[0,\infty) \times \mathbb{R}
\end{equation}
which is effectively a one--dimensional wave equation with a potential.
The number $\ell \in \mathbb{N}_0$ denotes the angular momentum and the coordinate system $(t,x)$ is chosen 
in such a way that $x \to \infty$ corresponds to spatial
infinity whereas $x \to -\infty$ is the location of the event horizon of the black hole.
Furthermore, the Regge--Wheeler potential $V_\ell$ reads
$$ V_\ell(x)=\left (1-\frac{1}{r(x)}\right)\left (\frac{\ell(\ell+1)}{r(x)^2}+
\frac{\sigma}{r(x)^3} \right )$$
where $\sigma \in \{-3,0,1\}$ is a fixed parameter which corresponds to different physical
situations.
The function $r(x)$ is implicitly defined by the equation $x=r(x)+\log(r(x)-1)$.
Consequently, $V_\ell(x)$ behaves like $\frac{\ell(\ell+1)}{x^2}$ as 
$x \to \infty$ and decays exponentially as $x \to -\infty$.
More precisely, for $x<0$, $V_\ell(x)$ can be written as a convergent series in powers of $e^x$, see 
\cite{DSS2}.
We refer the reader not familiar with this problem to the standard literature, e.g., 
\cite{Chandra} or the introduction of \cite{DSS1} as well as the references therein.
Recently  \cite{DSS1}, \cite{DS} established   decay bounds for solutions of Eq.~\eqref{eq:RW} of the form
\begin{equation} 
\label{eq:Price}
\|\langle \cdot \rangle^{-4}\psi_\ell(t,\cdot)\|_{L^\infty}
\leq C_\ell \langle t\rangle^{-3}\left \{
\|\langle \cdot \rangle^4 \psi_\ell(0,\cdot)\|_{L^1}+
\|\langle \cdot \rangle^4 \partial_x \psi_\ell(0,\cdot)\|_{L^1}+
\|\langle \cdot \rangle^4 \partial_t\psi_\ell(0,\cdot)\|_{L^1}
\right \}
\end{equation}
for $t \geq 0$ and any $\ell \in \mathbb{N}_0$.
Since the Regge--Wheeler equation \eqref{eq:RW} results from an angular momentum decomposition,
the ultimate goal is to sum the decay bounds \eqref{eq:Price} over all $\ell$.
However, in order to conclude the overall decay $t^{-3}$ one needs to control the behavior of the constants
$C_\ell$ in \eqref{eq:Price} as $\ell \to \infty$.
At this point the semiclassical analysis of the present paper (and also \cite{schlag4}) comes into play.
Indeed, by dividing Eq.~\eqref{eq:RW} by $\hbar^{-2}:=\ell(\ell+1)$ one is led to a semiclassical
problem of the type \eqref{eq:SO}.
As a consequence, the Regge--Wheeler potential is a prominent example from 
mathematical physics where our results apply, at least for $x \leq 0$.
In the case $x\geq 0$, where the potential decays like an inverse square, one has to rely on \cite{schlag4}.
In fact, in \cite{DSS2} it is shown how to synthesize \cite{schlag4}
and the present results to obtain the sharp $t^{-3}$ decay for linear perturbations of 
Schwarzschild without symmetry assumptions on the data.
At this point it has to be remarked that the long--sought $t^{-3}$ decay also follows from an 
independent result by Tataru \cite{tataru09} which appeared simultaneously to \cite{DSS2}.
An important consequence of our results (in conjunction with \cite{schlag4}) is that the small energy contributions decay rapidly as $\ell \to \infty$
and do not present an obstruction to the summation.
Indeed, as expected, the important contributions come from energies close to the maximum 
of the potential \cite{DSS2}.

\subsection{Notations and conventions}
Throughout this work we write $\langle \cdot \rangle: \mathbb{R} \to [\frac12,\infty)$ to denote a 
smooth, symmetric function that satisfies $\langle x \rangle=|x|$ for $|x|\geq 1$.
Furthermore, we assume $\langle \cdot \rangle$ to be strictly monotonically increasing on $(0,\infty)$.
For two real numbers $a$, $b$ we write $a \lesssim b$ if there exists a constant $c>0$ such that
$a \leq cb$. Unless stated otherwise, it is assumed that $c$ is absolute, i.e., it does not 
depend on any of the quantities
involved in the inequality. Similarly, we use $a \gtrsim b$ and $a \simeq b$ means $a \lesssim b$
and $b \lesssim a$.
Asymptotic equality is denoted by the symbol $\sim$.
For the Wronskian $W(f,g)$ of two functions $f,g$ we use the convention $W(f,g):=fg'-f'g$.
The expression $O(f(x))$ for some function $f$ is used to denote a generic \emph{real--valued}
function that satisfies $|O(f(x))|\lesssim |f(x)|$ in a domain of $x$ that follows from the context.
We write $O_\mathbb{C}(f(x))$ if the functions attains complex values as well.
Finally, the letter $C$ (possibly with indices) stands for a positive constant that may change its
value at each occurrence.

\subsection{The scattering matrix}
The (semiclassical) \emph{Jost solutions} $f_\pm(x,E;\hbar)$ of the operator $H(\hbar)$, where $E\geq 0$, 
are defined by the relation \footnote{Note carefully our somewhat unusual convention. We denote
by $E$ the \emph{square root} of the energy rather than the energy itself. This turns out to
be more convenient.}
$$ H(\hbar)f_\pm(\cdot,E;\hbar)=E^2 f_\pm(\cdot,E;\hbar) $$
and the asymptotic conditions $f_\pm(x,E;\hbar)\sim e^{\pm i\frac{E}{\hbar}x}$ as $x \to \pm \infty$.
Note, however, that $f_\pm(\cdot,E;\hbar)$ are not eigenfunctions of $H(\hbar)$ since they do not belong
to $L^2(\mathbb{R})$.
It is a classical result (see e.g., \cite{deift}) that the Jost solutions exist if
$V(\cdot;\hbar) \in L^1(\mathbb{R})$ (which is clearly the case here) and that they are unique.
The physical significance of the Jost solutions emerges from the fact that they correspond to outgoing waves. 
This terminology is explained by the observation that the functions $\psi_\pm(t,x,E;\hbar)=e^{-i\frac{E^2}{\hbar}t}f_\pm(x,E;\hbar)$ 
satisfy the Schr\"odinger equation
$$ i\hbar \partial_t \psi_\pm(t,\cdot,E;\hbar)
=H(\hbar)\psi_\pm(t,\cdot,E;\hbar) $$
and they behave asymptotically as $\psi_\pm(t,x,E;\hbar)\sim e^{-i\frac{E}{\hbar}(Et\mp x)}$ for $x \to \pm \infty$. 
In other words, they describe waves that travel towards $\pm \infty$.
It should again be remarked here that traditionally one writes $E$ instead of $E^2$ but in order to avoid
square roots we prefer to use $E^2$.
From the construction of the Jost solutions via Volterra iterations \cite{deift} 
it also follows that the asymptotics can be formally differentiated, i.e., 
$f_\pm'(x,E;\hbar)\sim \pm i\frac{E}{\hbar}e^{\pm i\frac{E}{\hbar}x}$ as $x \to \pm \infty$ and 
we immediately obtain the ($x$--independent) Wronskians
$$ W(f_\pm(\cdot,E;\hbar), \overline{f_\pm(\cdot,E;\hbar)})=\mp 2i\tfrac{E}{\hbar} $$
by evaluation as $x \to \pm \infty$.
Consequently, if $E \not=0$, $\{f_\pm(\cdot,E;\hbar),\overline{f_\pm(\cdot,E;\hbar)}\}$ are two fundamental systems for the ordinary differential equation $H(\hbar)f=E^2 f$ and there must exist connection coefficients $a_\pm(E;\hbar)$ and $b_\pm(E;\hbar)$ such that
$$ f_\pm(\cdot,E;\hbar)=a_\mp(E;\hbar)f_\mp(\cdot,E;\hbar)+b_\mp(E;\hbar)\overline{f_\mp(\cdot,E;\hbar)}. $$
This yields
$$ 2i\tfrac{E}{\hbar}=W(f_-,\overline{f_-})=W(a_+f_+ + b_+\overline{f_+},\overline{a_+ f_+}+\overline{b_+}f_+)
 =-2i\tfrac{E}{\hbar}|a_+|^2+2i\tfrac{E}{\hbar}|b_+|^2$$ 
where we omit the arguments $E$ and $\hbar$ occasionally in order to improve readability. 
We obtain $|b_+(E;\hbar)|\geq 1$ and consequently,
\begin{equation}
\label{eq:West}
|W(f_-,f_+)|=|W(a_+f_++b_+\overline{f_+}, f_+)|=2|b_+|\tfrac{|E|}{\hbar}\geq 2\tfrac{|E|}{\hbar} 
\end{equation}
which shows that $\{f_-(\cdot,E;\hbar),f_+(\cdot,E;\hbar)\}$ and $\{\overline{f_-(\cdot,E;\hbar)},
\overline{f_+(\cdot,E;\hbar)}\}$ are also two fundamental systems for the equation $H(\hbar)f=E^2 f$ provided that $E \not= 0$.
This is already a nontrivial statement since it contains global information.
For obvious reasons these two systems are called \emph{outgoing} and \emph{incoming}, respectively.
If $E=0$, $\{f_-(\cdot,0;\hbar),f_+(\cdot,0;\hbar)\}$ may or may not be a fundamental system for $H(\hbar)f=0$, depending on the 
special form of the potential. 
In the latter case one speaks of the existence of a \emph{zero energy resonance}.
However, we are not faced with this complication since it is not hard to see that the positivity 
assumption on the potential already excludes the existence of a zero energy resonance, see \cite{schlag4}.
A general solution $f(\cdot,E;\hbar)$ of $H(\hbar)f(\cdot,E;\hbar)=E^2 f(\cdot,E;\hbar)$ can be expanded as
$$ f(x,E;\hbar)=c_I^-(E;\hbar)\overline{f_-(x,E;\hbar)}+c_I^+(E;\hbar)\overline{f_+(x,E;\hbar)} $$
or, alternatively,
$$ f(x,E;\hbar)=c_O^+(E;\hbar)f_+(x,E;\hbar)+c_O^-(E;\hbar)f_-(x,E;\hbar) $$
and there exists a linear transformation that relates the representation with respect to
the incoming basis $(\overline{f_-},\overline{f_+})$ to the one with respect to the outgoing basis $(f_+,f_-)$.
The $2 \times 2$ matrix that represents this transformation is denoted by $\mathbb{S}(E;\hbar)$ 
and called the \emph{scattering matrix}.
We write 
$$ \mathbb{S}=\left (\begin{array}{cc}s_{11} & s_{12} \\
s_{21} & s_{22} \end{array} \right ) $$
and in the special case $c_I^-(E;\hbar)=1$, $c_I^+(E;\hbar)=0$ we obtain
$$ \overline{f_-(x,E;\hbar)}=s_{11}(E;\hbar)f_+(x,E;\hbar)+s_{21}(E;\hbar)f_-(x,E;\hbar). $$
The physical interpretation of this relation is that an incoming wave from the left scatters
at the potential and gets decomposed into a transmitted and a reflected part.
Consequently, $s_{11}$ and $s_{21}$ are called \emph{transmission} and \emph{reflection amplitudes}
and we write $s_{11}(E;\hbar)=t(E;\hbar)$, $s_{21}(E;\hbar)=r(E;\hbar)$.
A similar statement is true for an incoming wave from the right.
It is not hard to see that $s_{22}=t$ and $|t|^2+|r|^2=1$ which complies with the probabilistic interpretation of the 
quantum mechanical wave function.
Furthermore, it is well-known that $\mathbb{S}(E;\hbar)$ is unitary and completely determined by $t(E;\hbar)$
and $r(E;\hbar)$.
In semiclassical low energy scattering one is interested in the behavior of $\mathbb{S}(E;\hbar)$ as \emph{both} $E$ and $\hbar$ 
tend to $0$, \emph{simultaneously} and \emph{independently of each other}.

Let us remark that in the pure quantum mechanical setting, where $\hbar>0$ is fixed, the behavior of 
$\mathbb{S}(E;\hbar)$ for $E \to 0+$ is trivial. This follows from the fact that for exponentially decaying potentials the 
Jost solutions $f_\pm(x,E;\hbar)$ are smooth in $E$ around $0$ \cite{deift}. 
On the other hand, if $E>0$ is fixed and only $\hbar \to 0+$ is considered then the standard WKB
method works and the corresponding result is well-known and classical.
However, the problem becomes highly nontrivial if one allows both limits $E, \hbar \to 0+$ 
and moreover, if one is interested in uniform bounds for all derivatives with
respect to $E$. 
In order to approach this problem, we rely on a   coordinate transformation 
which maps $H(\hbar)f=E^2 f$ to a perturbed Bessel equation.
The resulting equation is solved by a perturbative construction around Bessel functions which is based
on suitable Volterra iterations.
At this point the asymptotic theory of Bessel functions becomes crucial.
For the convenience of the reader we have compiled the necessary results on Airy and Bessel functions
in the two Appendices \ref{sec:airy} and \ref{sec:bessel}.
However, since we need to control all derivatives of the perturbative solutions, the standard results on
Volterra equations are insufficient and we have to extend them.
This is done in Appendices \ref{sec:volterra} and \ref{sec:airy} where the required theory is developed.
As a consequence, we obtain the following representation of the Jost function of the operator $H(\hbar)$.

\begin{theorem}
\label{thm:Jostat0}
Fix small constants $E_0, \hbar_0>0$ and let $H(\hbar)$ be the semiclassical Schr\"odinger operator on $L^2(\mathbb{R})$
given by
$$ H(\hbar)f(x):=-\hbar^2 f''(x)+V(x;\hbar)f(x),\quad x\in \mathbb{R},\: \hbar \in (0,\hbar_0) $$
where the potential $V$ satisfies
\begin{enumerate}
\item $V^{(k)} \in C(\mathbb{R}\times [0,\hbar_0])$ for all $k\in\mathbb{N}_0$, \footnote{By $V^{(k)}$
we mean the $k$--th derivative with respect to the first variable, i.e., $V^{(k)}(x;\hbar)=\partial_x^k V(x;\hbar)$.}
\item $V(x;\hbar)>0$ for all $x \in \mathbb{R}$, $\hbar \in [0,\hbar_0)$,
\item $V(x;\hbar)=V_\pm(e^{-|x|};\hbar)$ for $\pm x \geq a$ where $a>0$ is some fixed constant and
the functions $V_\pm$ have the property that
$V_\pm^{(k)} \in C([0,1]\times [0,\hbar_0])$ for all $k \in \mathbb{N}_0$ and
$V_\pm(0;\hbar)=0$ for all $\hbar \in (0,\hbar_0)$.
\end{enumerate}
Then, with $\alpha:=\sqrt{\frac{\hbar^2}{4}+4E^2}$, 
the outgoing semiclassical Jost function $f_+(0,E;\hbar)$ of the operator $H(\hbar)$ 
has the representation
\begin{align*} 
f_+(0,E;\hbar)&=\alpha^\frac12 \gamma_+(E;\hbar)
e^{\frac{1}{\hbar}(S_+(E;\hbar)+iT_+(E;\hbar))}
\Big [ 1+\hbar \sigma_1(E;\hbar)+
\sigma_2(E;\hbar)e^{-\frac{2}{\hbar}S_+(E;\hbar)}\Big ] \\
f_+'(0,E;\hbar)&=\alpha^\frac12 \hbar^{-1}c_+(E;\hbar)\gamma_+(E;\hbar)
e^{\frac{1}{\hbar}(S_+(E;\hbar)+iT_+(E;\hbar))} 
\Big [ 1+\hbar \sigma_3(E;\hbar)+
\sigma_4(E;\hbar)e^{-\frac{2}{\hbar}S_+(E;\hbar)}\Big ]
\end{align*} 
where $c_+(E;\hbar), S_+(E;\hbar)\gtrsim 1$, $T_+(E;\hbar)\gtrsim E|\log(\frac{\hbar^2}{4}+4E^2)|$, $|\gamma_+(E;\hbar)|\simeq 1$,
and we have the bounds
$$ |\partial_E^\ell A(E;\hbar)|\leq C_\ell \hbar^{-\ell},\quad A\in \{c_+, \gamma_+, S_+, T_+, \sigma_1,
\sigma_2,\sigma_3,\sigma_4 \} $$
for all $\ell \in \mathbb{N}_0$, $E \in(0,E_0)$ and $\hbar\in (0,\hbar_0)$.
Furthermore, $\sigma_1(E;\hbar)$ and $\sigma_3(E;\hbar)$ are real-valued.
\end{theorem}

Based on the representation in Theorem \ref{thm:Jostat0} we obtain our main result on
the scattering matrix.

\begin{theorem}
\label{thm:main1}
Under the assumptions of Theorem \ref{thm:Jostat0}, 
the reflection and transmission amplitudes $r(E;\hbar)$ and $t(E;\hbar)$ associated to
$H(\hbar)$ are of the form
\begin{align*} 
t(E;\hbar)&=\tau(E;\hbar)\tfrac{E}{E+\hbar}e^{-\frac{1}{\hbar}(S(E;\hbar)+iT(E;\hbar))}[1+\hbar \sigma_1(E;\hbar)] \\
r(E;\hbar)&=\rho(E;\hbar)e^{-\frac{2i}{\hbar}T_-(E;\hbar)}[1+\hbar \sigma_2(E;\hbar)]
\end{align*}
for functions $S$, $T$, $T_-$, $\tau$ and $\rho$ 
satisfying $S(E;\hbar)\gtrsim 1$, $T(E;\hbar), T_-(E;\hbar) \gtrsim E|\log(\frac{\hbar^2}{4}+4E^2)|$, 
$|\tau(E;\hbar)|\simeq 
|\rho(E;\hbar)|\simeq 1$
for $E \in (0,E_0)$ and $\hbar \in (0,\hbar_0)$.
Furthermore, we have the bounds
$$ |\partial_E^\ell A(E;\hbar)|\leq C_\ell \hbar^{-\ell},\quad A \in \{S, T, T_-, \tau, \rho, \sigma_1,
\sigma_2 \} $$
for all $\ell \in \mathbb{N}_0$, $E \in (0,E_0)$ and $\hbar \in (0,\hbar_0)$ where $C_\ell>0$ 
are constants that only depend on $\ell$.
\end{theorem}

Some remarks are in order:
\begin{itemize}
\item Theorem \ref{thm:main1} shows that the probability for a quantum particle to tunnel through the 
potential barrier decays exponentially 
as $\hbar \to 0+$, i.e., as one approaches the classical regime.
As already mentioned, for fixed $E>0$ this exponential behavior is well-known, 
see e.g., \cite{Ramond96}. 
\item The factor $\tfrac{E}{E+\hbar}$ in $t(E;\hbar)$ reflects that zero energy is not a resonance, which is equivalent to the vanishing
of the transmission coefficient at zero energy. 
\item Theorem \ref{thm:main1} yields control \emph{over all derivatives} with respect to $E$ of 
the involved quantities.
This is the most salient feature of the result and it unveils the main difference to Schr\"odinger
operators with potentials that exhibit power law decay where
one loses powers of $E$ upon differentiating 
with respect to $E$, see e.g., \cite{schlag4}.
Such  behavior is in stark contrast to the situation here where one only loses powers of $\hbar$ and such a loss 
is in fact negligible compared
to the size of~$e^{-\frac{1}{\hbar}}$.
In particular this shows that the reflection and transmission amplitudes are smooth in $E$
at $E=0$ which is not the case for potentials that decay according to a power law.
\item The functions $S(E;\hbar)$, $T(E;\hbar)$ and $T_-(E;\hbar)$ are not uniquely determined 
and
we do not state explicit formulas here.
However, comparison with the classical WKB result for $E>0$ small but fixed shows that
in this case 
one
may take 
$$ S(E;\hbar)=\int_{x_{t,-}(\hbar)}^{x_{t,+}(\hbar)}\sqrt{|V(x;\hbar)-E^2|}dx $$
where $\pm x_{t,\pm}(\hbar)>0$ are the two solutions of $V(x;\hbar)-E^2=0$ 
which are bounded if $E>0$
is fixed.
Similar statements apply to $T(E;\hbar)$ and $T_-(E;\hbar)$, cf., e.g., \cite{Ramond96}.
\item Note further that the exponential decay of the transmission amplitude $t(E;\hbar)$ for $\hbar \to 0+$
should not be confused with the exponential decay found in \cite{Yafaev82}.
The result in \cite{Yafaev82} is valid for fixed $\hbar>0$ in the limit $E \to 0+$ and has a
completely different origin, namely the slow decay of the potential considered there.
\item The positivity assumption on the potential can be dropped and replaced by some 
weaker condition. However, one certainly has to exclude the existence of a zero energy resonance 
for the result to hold.
\end{itemize}

\subsection{The spectral measure}
The solution of the Schr\"odinger equation
\begin{equation}
\label{eq:SG}
i\hbar \partial_t \psi(t,\cdot)=H(\hbar)\psi(t,\cdot)  
\end{equation}
is of course given by $\psi(t,\cdot)=e^{-\frac{i}{\hbar}t H(\hbar)}\psi(0,\cdot)$ where the exponential has to
be interpreted according to the functional calculus for self--adjoint operators.
It is a consequence of Stone's formula (see \cite{teschl}) that the kernel $e^{-\frac{i}{\hbar}tH(\hbar)}(x,x')$ 
of the solution operator to the Schr\"odinger equation \eqref{eq:SG} is given
by the oscillatory integral
$$ e^{-\frac{i}{\hbar}tH(\hbar)}(x,x')=-\frac{2}{\pi \hbar^2}
\int_0^\infty e^{-\frac{i}{\hbar}tE^2}e(x,x',E;\hbar) E dE, \quad x'\leq x $$
with the \emph{semiclassical spectral measure}
\begin{equation}
\label{eq:spectralmeasure}
e(x,x',E;\hbar):=\mathrm{Im}\left [\frac{f_-(x',E;\hbar)f_+(x,E;\hbar)}{W(f_-(\cdot,E;\hbar), 
f_+(\cdot,E;\hbar))} \right ].
\end{equation}
Evidently, it is important to obtain bounds for the derivatives $\partial_E^\ell e(x,x',E;\hbar)$,
$\ell \in \mathbb{N}_0$,
since these immediately translate into decay estimates for the time evolution.
In this paper we prove the following result.

\begin{theorem}
\label{thm:main2}
Under the assumptions of Theorem \ref{thm:Jostat0}, the semiclassical spectral measure 
associated to the Schr\"odinger operator $H(\hbar)$ is of the form
$$ e(0,0,E;\hbar)=\gamma(E;\hbar)\hbar e^{-\frac{2}{\hbar}\tilde{S}(E;\hbar)} $$
where $\gamma$, $\tilde{S}$ are real--valued functions and $\tilde{S}(E;\hbar) \gtrsim 1$ 
for all $E \in (0,E_0)$, $\hbar \in (0,\hbar_0)$.
Furthermore, we have the bounds
$$ |\partial_E^\ell \gamma(E;\hbar)|\leq C_\ell \hbar^{-\ell},\quad |\partial_E^\ell \tilde{S}(E;\hbar)|\leq C_\ell \hbar^{-\ell} $$
for all $E \in (0,E_0)$, $\hbar \in (0,\hbar_0)$ and $\ell \in \mathbb{N}_0$.
\end{theorem}

Again, we make several remarks.

\begin{itemize}
\item
In some sense the behavior of $e(x,x',E;\hbar)$ at $x=x'=0$ is the most important one which 
is why we only consider $e(0,0,E;\hbar)$ here.
Based on our proof it is straightforward to obtain corresponding results for 
the general case $x,x' \in \mathbb{R}$. 

\item We emphasize again that derivatives with respect to $E$ cost powers of $\hbar$ instead of
powers of $E$ as is the case for potentials that decay according to a power law.
This behavior translates via stationary phase arguments (cf.~\cite{schlag1}, \cite{schlag2})
into a decay of the corresponding time evolution
that is faster than $t^{-N}$ for any $N \in \mathbb{N}$.

\item Theorem \ref{thm:main2} shows that the spectral measure decays exponentially in the 
semiclassical limit $\hbar \to 0+$.
Furthermore, the same is true for all $E$--derivatives of $e(0,0,E;\hbar)$.
Consequently, the small energy contributions to bounds for
the associated Schr\"odinger flow
decay rapidly as $\hbar \to 0+$.
\end{itemize}

\section{Preliminary transformations}

As outlined in the introduction, we intend to study the Jost function $f_+(\cdot,E;\hbar)$ which is uniquely defined by $H(\hbar)f_+(\cdot,E;\hbar)=E^2 f_+(\cdot,E;\hbar)$ and the asymptotic condition $f_+(x,E;\hbar)\sim e^{i\frac{E}{\hbar}x}$ as $x \to \infty$.
By applying the variation of constants formula, it immediately follows that $f_+(\cdot,E;\hbar)$ satisfies the Volterra equation
$$ f_+(x,E;\hbar)=e^{i\frac{E}{\hbar}x}+\tfrac{1}{E \hbar}\int_x^\infty \sin\left (\tfrac{E}{\hbar}(y-x)\right )V(y;\hbar)f_+(y,E;\hbar)dy. $$
Evidently, this representation does not give us any useful information in the limit $\hbar \to 0+$ which we are interested in.
This is just an instance of the fact that the potential $V(x;\hbar)$ for large $|x|$ 
cannot be treated as a perturbation in the semiclassical limit 
$\hbar \to 0+$.

\subsection{The Liouville--Green transform}
\label{sec:LG}
The main tool we are going to use is the Liouville--Green transform, see e.g., \cite{Olver}.
For the sake of completeness, we briefly review the main ideas.
Consider a semiclassical problem of the form
\begin{equation} 
\label{eq:modelLG}
-\hbar^2 f''(x)+Q(x)f(x)=0 \end{equation}
for $x \in I \subset \mathbb{R}$, $I$ an open interval.
Here, the function $Q$ is assumed to be sufficiently smooth.
Now suppose $\varphi: I \to J \subset \mathbb{R}$ is a diffeomorphism (not necessarily orientation--preserving) onto some open interval $J$ and set $g(\varphi(x)):=|\varphi'(x)|^\frac12 f(x)$.
A straightforward computation shows that $g$ satisfies the equation
\begin{equation}
 \label{eq:LG}
-\hbar^2 g''(\varphi(x))+\frac{Q(x)}{\varphi'(x)^2}g(\varphi(x))=\hbar^2 \left (\frac34 \frac{\varphi''(x)^2}{\varphi'(x)^4}-\frac12 \frac{\varphi'''(x)}{\varphi'(x)^3}\right )g(\varphi(x)) 
\end{equation}
if and only if $f$ satisfies Eq.~\eqref{eq:modelLG}.
Observe that, if $\varphi$ behaves reasonably well on $I$, this transformation introduces a right--hand side that is of size $\hbar^2$ \emph{globally} and may therefore be treated as a perturbation.
The point where this procedure turns into an art form is when it comes to choose the diffeomorphism $\varphi$. This has to be done on a case by case basis.
For instance, if $Q$ does not vanish on $I$, one may choose $\varphi$ in such a way that $\frac{Q(x)}{\varphi'(x)^2}=\sign\;Q(x)$ which leads to the classical WKB ansatz.
It is a historical curiosity that the WKB method is in fact a special case of the much older Liouville--Green transform.
Another situation which we are going to encounter is when $Q$ has a single nondegenerate zero, i.e., $Q(x_t)=0$ and $Q'(x_t)\not=0$ for some $x_t \in I$.
In this case, $x_t$ is called a \emph{turning point} of the equation since the behavior of solutions changes from oscillatory to exponentially increasing / decreasing.
In order to capture this transition, one chooses $\varphi$ such that $\frac{Q(x)}{\varphi'(x)^2}=\varphi(x)$.
This leads to the so--called \emph{Langer transform} \cite{Miller} where the problem at hand is mapped to a perturbed Airy equation 
(see Appendix~\ref{sec:airy}).
However, we remark that the ``right'' choice of $\varphi$ is not always obvious since it may be advantageous to move some terms of $Q$ to the right--hand side of the equation \emph{prior} to applying the Liouville--Green transform.

\subsection{Reduction to a perturbed Bessel equation}
The semiclassical problem we study is given by
\begin{equation} 
\label{eq:main2}
-\hbar^2 f''(x)+[V(x;\hbar)-E^2]f(x)=0  
\end{equation}
for $x > x_0$ where $x_0 \in\mathbb{R}$ is fixed throughout.
According to the requirements in Theorem \ref{thm:main1} we may assume that 
$V(x;\hbar)=e^{-x}[1+\varepsilon(e^{-x};\hbar)]$ for $x > x_0$ where
$|\varepsilon(e^{-x};\hbar)|\lesssim e^{-x}$ and $\varepsilon(\cdot;\hbar)$ is smooth around $0$ 
and satisfies $\varepsilon(\cdot;\hbar)>-1$
for all $0\leq \hbar \ll 1$ (the last condition follows from the positivity assumption on the potential).
Since we only consider energies $E$ close to $0$, Eq.~\eqref{eq:main2} has a unique
nondegenerate turning point $x_t(E;\hbar)>0$ defined by the relation $V(x_t(E;\hbar);\hbar)=E^2$.
Assume for the moment that $V(x;\hbar)=e^{-x}$ exactly for $x>x_0$.
In this case, the turning point is given by $x_t(E;\hbar)=-2\log E$ and thus, it is not smooth in $E$ at $E=0$.
Since we want to control derivatives with respect to $E$, 
it is desirable to have the dependence of the turning point on $E$ be as regular as possible.
Consequently, it is advantageous to introduce new variables.
Based on formula \eqref{eq:LG}, it is reasonable to look for a diffeomorphism $\varphi$ satisfying $\varphi'(x)^2=e^{-x}$ and thus, $\varphi(x)=2e^{-\frac{x}{2}}$ is a possible choice.
Obviously, $\varphi$ maps $(x_0,\infty)$ to $(0,y_0)$ diffeomorphically where $y_0:=2e^{-\frac{x_0}{2}}$ and $\varphi'(x)<0$ for all $x \in (x_0,\infty)$, i.e., $\varphi$ is orientation--reversing.
Setting $\tilde{f}(\varphi(x)):=|\varphi(x)|^\frac12 f(x)$, Eq.~\eqref{eq:main2} transforms into
\begin{equation} 
\label{eq:modBessel}
-\hbar^2 \tilde{f}''(y)+\left [1-\frac{\frac{\hbar^2}{4}+4E^2}{y^2} \right ]\tilde{f}(y)=0 
\end{equation}
on $y \in (0,y_0)$ which turns out to be a modified Bessel equation.
Indeed, setting $\hat{f}(z):=\tilde{f}(\hbar z)$ yields
$$ \hat{f}''(z)+\left [\frac{\nu^2+\frac14}{z^2}-1 \right ]\hat{f}(z)=0,\quad \nu:=2\tfrac{E}{\hbar} $$
which is solved by $z \mapsto \sqrt{z}I_{\pm i\nu}(z)$, the modified Bessel functions.
This motivates the following.

\begin{lemma}
\label{lem:prtobessel}
Let $\varphi(x):=2e^{-\frac{x}{2}}$ and $y_0:=2e^{-\frac{x_0}{2}}$ be as above. 
Then the function $\tilde{f}$, defined by $\tilde{f}(\varphi(x))=|\varphi'(x)|^\frac12 f(x)=e^{-\frac{x}{4}}f(x)$, satisfies the equation
\begin{equation} 
\label{eq:besselform}
-\hbar^2 \tilde{f}''(y)+\left [1+\varepsilon(y^2; \hbar)-\frac{\frac{\hbar^2}{4}+4E^2}{y^2} \right ]\tilde{f}(y)=0 
\end{equation}
for $y \in (0,y_0)$ if and only if $f$ satisfies Eq.~\eqref{eq:main2} for $x \in (x_0,\infty)$.
Here, the function $\varepsilon(\cdot;\hbar)$ is smooth around $0$ and we have the bounds
$$ |\partial_y^k \varepsilon(y^2;\hbar)|\leq C_k y^{\max\{2-k,0\}} $$
for all $y \in (0,y_0)$, $0<\hbar\ll 1$ and $k \in \mathbb{N}_0$.
Finally, there exists a unique solution $\tilde{f}_+(\cdot,E;\hbar)$ to Eq.~\eqref{eq:besselform} with the asymptotic behavior
$$ \tilde{f}_+(y,E;\hbar)\sim \left(\frac{y}{2}\right)^{\frac12-2i\frac{E}{\hbar}}\quad (y \to 0+) $$
and the outgoing Jost function of the operator $H(\hbar)$ is given by
$$ f_+(x,E;\hbar)=|\varphi'(x)|^{-\frac12}\tilde{f}_+(\varphi(x),E;\hbar)
=e^{\frac{x}{4}}\tilde{f}_+(2e^{-\frac{x}{2}},E;\hbar). $$
\end{lemma}

\begin{proof}
By applying formula \eqref{eq:LG}, Eq.~\eqref{eq:main2} transforms into Eq.~\eqref{eq:besselform}.
Furthermore, 
 as already remarked, Eq.~\eqref{eq:main2} has a unique solution $f_+(\cdot,E;\hbar)$ 
 (the Jost function) that satisfies $f_+(x,E;\hbar) \sim e^{i\frac{E}{\hbar}x}$ as $x \to \infty$,
see \cite{deift}.
 Consequently, by setting $f_+(x,E;\hbar)=|\varphi'(x)|^{-\frac12}\tilde{f}_+(\varphi(x),E;\hbar)$
 we see that $\tilde{f}_+(\cdot,E;\hbar)$ is a solution to Eq.~\eqref{eq:besselform} and the stated
 asymptotics of $\tilde{f}_+(\cdot,E;\hbar)$ 
 follow immediately from the asymptotics of $f_+(\cdot,E;\hbar)$ and the explicit form
 of $\varphi$.
 The bounds on $\varepsilon$ are a consequence of the assumed properties of $V$.
\end{proof}

From now on we use the shorthand notation $\alpha^2:=\frac{\hbar^2}{4}+4E^2$.
We are interested in the behavior of Eq.~\eqref{eq:besselform} for small $\alpha$.
However, this is a singular perturbation problem and it is thus natural 
 to rescale the independent variable  by setting $\tilde{g}(z):=\tilde{f}(\alpha z)$ which transforms Eq.~\eqref{eq:besselform} into
\begin{equation}
\label{eq:besselformrescale}
 -\hbar_1^2 \tilde{g}''(z)+\left [1-\tfrac{1}{z^2}+\varepsilon(\alpha^2 z^2;\hbar) \right ]\tilde{g}(z)=0
\end{equation}
for $z \in (0,\frac{y_0}{\alpha})$ where $\hbar_1:=\frac{\hbar}{\alpha}$.
Note that $0<\alpha \ll 1$ and $0<\hbar\ll 1$ implies $\hbar_1 \in (0,2)$ and thus, $\hbar_1$ is not necessarily small. 
The point of this rescaling is that $|\varepsilon(\alpha^2 z^2;\hbar)|\lesssim \alpha^2 z^2$ and therefore, Eq.~\eqref{eq:besselformrescale} is a regular perturbation of a modified Bessel equation, at least for small $z$.
Clearly, $z$ can become as large as $\alpha^{-1}$ and in this case, the perturbation $\varepsilon$ is of order $1$.
However, as we will demonstrate now, the perturbation can be made globally small by a suitable Liouville--Green transform.

\section{The normal form reduction}
First, we study a useful diffeomorphism that will appear frequently in the following.

\begin{lemma}
 \label{lem:zeta}
 Let $\zeta$ be defined by
 $$ \zeta(x):=\sign (x-1) \left |\tfrac32 \int_1^x \sqrt{|1-\tfrac{1}{u^2}|}du 
\right |^\frac23,\quad x>0. $$
 Then $\zeta: (0,\infty) \to \mathbb{R}$ is an orientation--preserving diffeomorphism that satisfies $\zeta(1)=0$ and $\zeta(x)\simeq x^{\frac23}$, $\zeta'(x) \simeq x^{-\frac13}$ for $x \geq 2$.
 Furthermore, we have the derivative bounds
$$|\zeta^{(k)}(x)|\leq C_k \langle x \rangle^{\frac23-k} $$ 
for all $x \geq \frac12$, $k\in \mathbb{N}_0$ and the asymptotic behavior
$$ \tfrac23 [-\zeta(x)]^\frac32=-\log x+\log 2 -1 + x^2\varepsilon_0(x), \quad |\varepsilon_0^{(k)}(x)|\leq C_k $$
for all $x\in (0,\frac23)$.
\end{lemma}

\begin{proof}
 It is obvious that $\zeta$ is smooth on $(0,1)\cup (1,\infty)$ and smoothness around $1$ is proved by a Taylor expansion, cf.~the proof of Lemma~\ref{lem:tauTP} below.
 This also yields $\zeta'(x)>0$ for all $x>0$ and thus, $\zeta: (0,\infty) \to \mathbb{R}$ is injective and orientation--preserving.
 Obviously, $\zeta$ is also surjective and thus, a diffeomorphism.
 The stated asymptotic behavior of $\zeta(x)$ for $x \to \infty$ is also easily seen directly from the definition and the derivative bounds for $x \geq \frac12$ follow from the theory set forth in Appendix~\ref{sec:symbol}.
 For the asymptotic behavior as $x \to 0+$ note that, for $x\in (0,1)$, we have
 \begin{align*}
  \tfrac23 [-\zeta(x)]^\frac32&=\int_x^1 \sqrt{\tfrac{1}{u^2}-1}du=\log u+\sqrt{1-u^2}-\left. \log\left (1+\sqrt{1-u^2}\right )\right |_x^1 \\
  &=-\log x-\sqrt{1-x^2}+\log\left (1+\sqrt{1-x^2} \right ) \\
  &=-\log x+\log 2-1+\tilde{\varepsilon}(x)
 \end{align*}
where $\tilde{\varepsilon}$ is smooth on $(-1,1)$ and $\tilde{\varepsilon}(x)=O(x^2)$.
\end{proof}

Our goal is to ``remove'' the perturbation $\varepsilon$ in Eq.~\eqref{eq:besselformrescale} by a Liouville--Green transform.
Note that Eq.~\eqref{eq:besselformrescale} has a unique nondegenerate positive turning point $z_t(\alpha^2;\hbar)$ which is 
implicitly defined by 
$$ 1-\frac{1}{z_t(\alpha^2;\hbar)^2}+\varepsilon(\alpha^2 z_t(\alpha^2;\hbar)^2;\hbar)=0. $$
Since $\varepsilon(0;\hbar)=0$, it follows from the smoothness of $\varepsilon(\cdot;\hbar)$ and the implicit function theorem that $z_t(\cdot;\hbar)$ is smooth around $0$ and of the form $z_t(\alpha^2;\hbar)=1+O(\alpha^2)$, uniformly in $0<\hbar\ll 1$.
According to Section \ref{sec:LG}, we are looking for a diffeomorphism $\varphi(\cdot,\alpha^2;\hbar)$ that satisfies
\begin{equation} 
\label{eq:LGbessel}
\frac{1-\frac{1}{z^2}+\varepsilon(\alpha^2 z^2;\hbar)}{\varphi'(z,\alpha^2;\hbar)^2}=1-\frac{1}{\varphi(z,\alpha^2;\hbar)^2} 
\end{equation}
for $z \in (0,\frac{y_0}{\alpha})$ where $'$ always means differentiation with respect to the first variable.
Clearly, if $\alpha=0$, we may choose $\varphi(z,0;\hbar)=z$, and thus, since $\varepsilon(\alpha^2 z^2;\hbar)$ depends smoothly on $\alpha^2$, we expect $\varphi(\cdot,\alpha^2;\hbar)$ to be a smooth perturbation of the identity.
Fortunately, Eq.~\eqref{eq:LGbessel} is separable and denoting the new independent variable by $w$ (i.e., $w=\varphi(z,\alpha^2;\hbar)$), integration yields
\begin{equation} 
\label{eq:LGbesselint1}
\int_1^w \sqrt{|1-\tfrac{1}{u^2}|}du=\int_{z_t(\alpha^2;\hbar)}^z \sqrt{|1-\tfrac{1}{u^2}+\varepsilon(\alpha^2 u^2;\hbar)|}du 
\end{equation}
which, by montonicity, defines a one--to--one correspondence between $w$ and $z$ and this implicitly yields the desired mapping $\varphi(\cdot,\alpha^2;\hbar)$.
We can immediately read off that $\varphi(z_t(\alpha^2;\hbar),\alpha^2;\hbar)=1$, i.e., the turning point is mapped to $1$, and $\lim_{z\to 0+}\varphi(z,\alpha^2;\hbar)=0$. 
We also set
\begin{equation}
\label{eq:defw0} 
w_0(\alpha^2;\hbar):=\lim_{z \to \alpha^{-1}y_0-}\varphi(z,\alpha^2;\hbar) 
\end{equation}
which is finite for any $0<\alpha \ll 1$ and $w_0(\alpha^2;\hbar) \to \infty$ as $\alpha \to 0+$.
Thus, Eq.~\eqref{eq:LGbesselint1} defines a bijective mapping $\varphi(\cdot,\alpha^2;\hbar): (0,\alpha^{-1}y_0)\to (0,w_0(\alpha^2;\hbar))$.
In order to study the analytical properties of $\varphi$, it is important to realize that $\varphi(\cdot,\alpha^2;\hbar)$ 
is smooth on $(0,z_t(\alpha^2;\hbar))\cup(z_t(\alpha^2;\hbar),\alpha^{-1}y_0)$ but at $z=z_t(\alpha^2;\hbar)$ 
(or, equivalently, $w=1$) we run into 
some difficulties since both sides of Eq.~\eqref{eq:LGbesselint1} are nonsmooth there.
However, it may still be the case that the two singularities cancel out and in fact, this is exactly what happens.
In order to see this, compare the function $\zeta$ from Lemma~\ref{lem:zeta} with the left--hand side of Eq.~\eqref{eq:LGbesselint1} which shows that it is natural to raise both sides of Eq.~\eqref{eq:LGbesselint1} to the power $\frac23$.
More precisely, taking the sign issues into account, we obtain
$$ \sign(w-1)\left |\tfrac32 \int_1^w \sqrt{|1-\tfrac{1}{u^2}|}du\right |^\frac23=\tau(z,\alpha^2;\hbar)$$ 
where
\begin{equation}
\label{eq:deftau}
\tau(z,\alpha^2;\hbar):=\sign(z-z_t(\alpha^2;\hbar))\left |\tfrac32 \int_{z_t(\alpha^2;\hbar)}^z \sqrt{|1-\tfrac{1}{u^2}+\varepsilon(\alpha^2 u^2;\hbar)|}du \right |^\frac23
\end{equation}
and it follows that $\varphi$ is in fact given by 
$\varphi(z,\alpha^2;\hbar)=\zeta^{-1}(\tau(z,\alpha^2;\hbar))$.

\subsection{Analytical properties of $\tau$}
The function $\tau(z,\alpha^2;\hbar)$ diverges for $z\to 0+$ and has an apparent singularity at the turning point $z=z_t(\alpha^2;\hbar)$. Furthermore, it also diverges for $z \to \alpha^{-1}y_0$ and $\alpha \to 0+$.
Consequently, we study the analytical properties of $\tau$ near those points separately and we start with the behavior near the turning point.

\begin{lemma}
 \label{lem:tauTP}
 Let $\delta, \alpha_0>0$ be sufficiently small. Then the function
 $\tau$, defined in \eqref{eq:deftau},
 satisfies the bounds
$$ |\partial_{\alpha^2}^\ell \partial_z^k \tau(z,\alpha^2;\hbar)|\leq C_{k,\ell}|z-z_t(\alpha^2;\hbar)|^{\max\{1-k-\ell,0\}} $$
for all $z \in (1-\delta,1+\delta)$, $\alpha \in (0,\alpha_0)$, $0<\hbar\ll 1$ and $k,\ell \in \mathbb{N}_0$.
\end{lemma}

\begin{proof}
First of all, we have already noted that the turning point $z_t(\alpha^2;\hbar)$ is a smooth function of $\alpha^2$ near $0$ and satisfies $z_t(\alpha^2;\hbar)=1+O(\alpha^2)$, uniformly in $0<\hbar\ll 1$.
Consequently, by choosing $\alpha_0>0$ sufficiently small, we can always accomplish that $z_t(\alpha^2;\hbar)\in (\frac12,\frac32)$ for all $0<\alpha<\alpha_0$ and $0<\hbar\ll 1$ which we shall assume in the following and we consider $(u,\alpha,\hbar) \in (\frac12,\frac32)\times (0,\alpha_0)\times (0,\hbar_0)$ where $\hbar_0>0$ is sufficiently small.
Now we study the integrand in Eq.~\eqref{eq:deftau}.
Since the turning point is nondegenerate, we can write
  $$ 1-\tfrac{1}{u^2}+\varepsilon(\alpha^2 u^2;\hbar)=[u-z_t(\alpha^2;\hbar)]g(u,\alpha^2;\hbar),\quad
  g(z_t(\alpha^2;\hbar),\alpha^2;\hbar)\gtrsim 1 $$
  with a suitable function $g(u,\alpha^2;\hbar)$ that is smooth with respect to $u$ and $\alpha^2$ in the regime of parameters being considered here.   Choose an arbitrary $N \in \mathbb{N}$.
  Since $g$ does not change sign, the function $\tilde{g}:=\sqrt{g}$ is smooth as well and a Taylor expansion yields
  $$ \tilde{g}(u,\alpha^2;\hbar)=\sqrt{g(u,\alpha^2;\hbar)}=\sum_{j=0}^N a_j(\alpha^2; \hbar)[u-z_t(\alpha^2;\hbar)]^j+R_N(u,\alpha^2;\hbar) $$
  with a remainder of the form
  $$ R_N(u,\alpha^2;\hbar)=\frac{1}{N!}\int_{z_t(\alpha^2;\hbar)}^u \partial_s^{N+1}\tilde{g}(s,\alpha^2;\hbar)(u-s)^N ds. $$
The coefficients $a_j(\alpha^2;\hbar)$ are given by
$$ a_j(\alpha^2;\hbar)=\tfrac{1}{j!}\partial_u^{j}\tilde{g}(u,\alpha^2;\hbar)|_{u=z_t(\alpha^2;\hbar)} $$
and thus, they are smooth functions of $\alpha^2$.
We emphasize that due to the nondegeneracy of the turning point we have $a_0(\alpha^2;\hbar)\gtrsim 1$.
Furthermore, the remainder satisfies the derivative bounds
$$ |\partial_{\alpha^2}^\ell \partial_u^k R_N(u,\alpha^2;\hbar)|\leq C_{k,\ell}|u-z_t(\alpha^2;\hbar)|^{N+1-k-\ell} $$
for $k,\ell \in \mathbb{N}_0$ provided that $k+\ell\leq N$.
Consequently, for $u\geq z_t(\alpha^2;\hbar)$, we obtain an expansion of the form
\begin{align*} \sqrt{1-\tfrac{1}{u^2}+\varepsilon(\alpha^2 u^2;\hbar)}=&\sum_{j=0}^N a_j(\alpha^2;\hbar)[u-z_t(\alpha^2;\hbar)]^{j+\frac12} \\
&+[u-z_t(\alpha^2;\hbar)]^{\frac12}R_N(u,\alpha^2;\hbar) 
\end{align*}
and integration yields
\begin{align*}
 \int_{z_t(\alpha^2;\hbar)}^z \sqrt{1-\tfrac{1}{u^2}+\varepsilon(\alpha^2 u^2;\hbar)}du&=\sum_{j=0}^N \frac{a_j(\alpha^2;\hbar)}{j+\tfrac32}[z-z_t(\alpha^2;\hbar)]^{j+\frac32} \\
 &\quad +[z-z_t(\alpha^2;\hbar)]^{\frac32}\tilde{R}_N(z,\alpha^2;\hbar)
\end{align*}
provided that $z \in (\frac12,\frac32)$ and $z \geq z_t(\alpha^2;\hbar)$ where $\tilde{R}_N$ 
satisfies the same bounds as $R_N$.
We obtain
\begin{align*}
 \tau(z,\alpha^2;\hbar)&=\left \{ \tfrac32 [z-z_t(\alpha^2;\hbar)]^{\frac32}\left [\sum_{j=0}^N \frac{a_j(\alpha^2;\hbar)}{j+\tfrac32}[z-z_t(\alpha^2;\hbar)]^j +\tilde{R}_N(z,\alpha^2;\hbar) \right ] \right \}^\frac23 \\
 &=[z-z_t(\alpha^2;\hbar)]\hat{R}_N(z,\alpha^2;\hbar)
\end{align*}
where $|\partial_{\alpha^2}^\ell \partial_z^k \hat{R}_N(z,\alpha^2;\hbar)|\leq C_{k,\ell}$ provided that $k+\ell\leq N$ and $z$ is sufficiently close to $z_t(\alpha^2;\hbar)$ (recall that $a_0(\alpha^2;\hbar)\gtrsim 1$).
An analogous calculation for $z\leq z_t(\alpha^2;\hbar)$ shows that the result is in fact valid for all $z$ with $|z-z_t(\alpha^2;\hbar)|$ sufficiently small.
However, since $z_t(\alpha^2;\hbar)\to 1$ uniformly in $\hbar$ as $\alpha \to 0+$, we can always find a $\delta>0$ such that the above holds for all $z \in (1-\delta,1+\delta)$, $\alpha \in (0,\alpha_0)$ and $0<\hbar\ll1$ provided that $\alpha_0>0$ is chosen small enough.  
Since $N$ was arbitrary, the claim follows.
\end{proof}

Next, we consider the behavior as $z \to 0+$.

\begin{lemma}
\label{lem:tau0}
Let $\delta,\alpha_0>0$ be sufficiently small.
 Then the function $\tau$ given by \eqref{eq:deftau} has the form
 $$ \tfrac23[-\tau(z,\alpha^2;\hbar)]^\frac32=-\log z+\varepsilon_1(z,\alpha^2;\hbar) $$
 where $\varepsilon_1$ satisfies the bounds
 $$ |\partial_{\alpha^2}^\ell \partial_z^k \varepsilon_1(z,\alpha^2;\hbar)|\leq C_{k,\ell} $$
 for all $z \in (0,1-\frac{\delta}{2})$, $\alpha\in (0,\alpha_0)$ and $0<\hbar\ll 1$. 
\end{lemma}

\begin{proof}
 Let $\delta, \alpha_0>0$ be so small that Lemma~\ref{lem:tauTP} is applicable. Moreover, choose $\alpha_0>0$ so small that $z_t(\alpha^2;\hbar)\in (1-\frac{\delta}{3},1+\frac{\delta}{3})$ for all $\alpha \in (0,\alpha_0)$ and $0<\hbar\ll 1$.
 If $u \in (0,1-\frac{\delta}{2})$, the integrand in the definition of $\tau$, Eq.~\eqref{eq:deftau}, can be written as
 $$ \sqrt{\tfrac{1}{u^2}-1-\varepsilon(\alpha^2 u^2;\hbar)}=\tfrac{1}{u}\sqrt{1-u^2-u^2\varepsilon(\alpha^2 u^2;\hbar)}=\tfrac{1}{u}+g(u,\alpha^2;\hbar) $$
 and, since the expression under the square root stays away from zero, $g$ satisfies
 $$ |\partial_{\alpha^2}^\ell \partial_u^k g(u,\alpha^2;\hbar)|\leq C_{k,\ell} $$
 for all $u \in (0,1-\frac{\delta}{2})$, $\alpha \in (0,\alpha_0)$, $0<\hbar\ll 1$ and $k,\ell \in \mathbb{N}_0$.
 We split the integral in the definition of $\tau$ according to $$\int_0^{z_t(\alpha^2;\hbar)}=\int_z^{1-\frac{\delta}{2}}
+\int_{1-\frac{\delta}{2}}^{z_t(\alpha^2;\hbar)}$$
and obtain
$$ \int_z^{1-\frac{\delta}{2}}\sqrt{\tfrac{1}{u^2}-1-\varepsilon(\alpha^2 u^2;\hbar)}du=-\log z+\tilde{g}(z,\alpha^2;\hbar) $$
for $z \in (0,1-\frac{\delta}{2})$ where $\tilde{g}$ satisfies the same bounds as $g$.
For the second integral note that by definition of $\tau$ we have
$$ \int_{1-\frac{\delta}{2}}^{z_t(\alpha^2;\hbar)}\sqrt{\tfrac{1}{u^2}-1-\varepsilon(\alpha^2 u^2;\hbar)}du=\tfrac23[-\tau(1-\tfrac{\delta}{2},\alpha^2;\hbar)]^\frac32 $$
and according to Lemma~\ref{lem:tauTP} we have the bounds
$$ |\partial_{\alpha^2}^\ell [-\tau(1-\tfrac{\delta}{2},\alpha^2;\hbar)]^\frac32|\leq C_\ell |1-\tfrac{\delta}{2}-z_t(\alpha^2;\hbar)|^{\frac32-\ell}\leq C_\ell $$
since $|1-\frac{\delta}{2}-z_t(\alpha^2;\hbar)|\gtrsim \delta$ for all $\alpha \in (0,\alpha_0)$ and $0<\hbar\ll 1$ which finishes the proof.
\end{proof}

To conclude the study of $\tau$, we finally consider the case for large $z$.

\begin{lemma}
 \label{lem:taularge}
 Let $\delta,\alpha_0>0$ be sufficiently small.
 Then the function $\tilde{\tau}:=\frac23 |\tau|^\frac32$, where $\tau$ is defined in \eqref{eq:deftau}, satisfies the bounds
 \begin{align*}
 |\partial_{\alpha^2}^\ell \partial_z^k \tilde{\tau}(z,\alpha^2;\hbar)|&\leq 
C_{k,\ell}z^{1-k+2\ell}
 \end{align*}
 for all $z \in (1+\frac{\delta}{2}, \alpha^{-1}y_0)$, $\alpha \in (0,\alpha_0)$, $0<\hbar\ll 1$ and $k,\ell \in \mathbb{N}_0$.
\end{lemma}

\begin{proof}
As in the proof of Lemma~\ref{lem:tau0} we assume that $\delta,\alpha_0>0$ are so small that Lemma~\ref{lem:tauTP} holds and $z_t(\alpha^2;\hbar) \in (1-\frac{\delta}{3},1+\frac{\delta}{3})$
for all $\alpha \in (0,\alpha_0)$, $0<\hbar \ll 1$.
 According to Eq.~\eqref{eq:deftau}, we have
 $$ \tilde{\tau}(z,\alpha^2;\hbar)=\tfrac23 \tau(z,\alpha^2;\hbar)^\frac32=\int_{z_t(\alpha^2;\hbar)}^z \sqrt{1-\tfrac{1}{u^2}+\varepsilon(\alpha^2 u^2;\hbar)}du $$
 for $z \in (1+\frac{\delta}{2},\alpha^{-1}y_0)$.
 Again, we split the integral as
 \begin{equation*} 
\int_{z_t(\alpha^2;\hbar)}^z =
\int_{z_t(\alpha^2;\hbar)}^{1+\frac{\delta}{2}}+\int_{1+\frac{\delta}{2}}^z =:A(\alpha^2;\hbar)
+B(z,\alpha^2;\hbar)
\end{equation*}
and, noting that 
$$ A(\alpha^2;\hbar)=\tfrac23 \tau(1+\tfrac{\delta}{2},\alpha^2;\hbar)^\frac32, $$
we infer from Lemma~\ref{lem:tauTP} the bounds
$$ |\partial_{\alpha^2}^\ell A(\alpha^2;\hbar)|\leq C_\ell |1+\tfrac{\delta}{2}-z_t(\alpha^2;\hbar)|^{\frac32-\ell}\leq C_\ell,\quad \ell \in \mathbb{N}_0 $$
since by the assumptions on $\alpha_0$ and $\delta$, we have $|1+\tfrac{\delta}{2}-z_t(\alpha^2;\hbar)|\gtrsim \delta$ for all $\alpha \in (0,\alpha_0)$ and $0<\hbar \ll 1$.
In order to obtain suitable bounds on $B(z,\alpha^2;\hbar)$, note that, by Lemma~\ref{lem:prtobessel}, 
$\varepsilon$ satisfies 
$$ |\partial_{u^2}^k \varepsilon(\alpha^2 u^2; \hbar)|=|(\alpha^2)^k \varepsilon^{(k)}(\alpha^2 u^2; \hbar)|
\leq C_k (\alpha^2)^k \leq C_k (u^2)^{-k} $$
since $|u|\lesssim \alpha^{-1}$. Consequently, 
$\partial_{\alpha^2}^\ell \varepsilon(\alpha^2 u^2;\hbar)=(u^2)^\ell \varepsilon^{(\ell)}(\alpha^2 u^2; \hbar)$
behaves like a symbol and we infer the bounds 
$$|\partial_{\alpha^2}^\ell \partial_{u^2}^k \varepsilon(\alpha^2 u^2;\hbar)|
\leq C_{k,\ell}(u^2)^{\ell-k}, \quad k,\ell \in \mathbb{N}_0 $$
for $u \in (1+\tfrac{\delta}{2}, \alpha^{-1}y_0)$. 
This implies 
$$|\partial_{\alpha^2}^\ell \partial_u^k [1-\tfrac{1}{u^2}+\varepsilon(\alpha^2 u^2;\hbar)]|\leq C_{k,\ell} u^{-k+2\ell} $$
and, since $1-\frac{1}{u^2}+\varepsilon(\alpha^2 u^2;\hbar)\gtrsim 1$ in the parameter regime we are considering, we obtain $$|\partial_{\alpha^2}^\ell \partial_z^k B(z,\alpha^2;\hbar)|\leq C_{k,\ell} z^{1-k+2\ell}$$
for all $z \in (1+\frac{\delta}{2}, \alpha^{-1}y_0)$, $\alpha \in (0,\alpha_0)$, $0<\hbar \ll 1$ and $k,\ell\in \mathbb{N}_0$, cf.~Lemma~\ref{lem:symbc}. 
Trivially, the same bounds hold for $A(\alpha^2;\hbar)$ and the claim follows.
\end{proof}

\subsection{Application of the Liouville--Green transform}

We have collected enough information on the function $\tau$ to study the diffeomorphism $\varphi$.

\begin{proposition}
\label{prop:phi}
 Let $\zeta$, $\tau$ be as in Lemma~\ref{lem:zeta} and Eq.~\eqref{eq:deftau}, respectively, and suppose
$\alpha_0>0$ is sufficiently small. Then the function 
$\varphi:=\zeta^{-1}\circ \tau$ has the following properties:
\begin{enumerate}
 \item For any $\alpha \in (0,\alpha_0)$ and $0<\hbar\ll 1$, $\varphi(\cdot,\alpha^2;\hbar): (0,\alpha^{-1}y_0) \to (0,w_0(\alpha;\hbar))$ is an orientation--preserving diffeomorphism where $w_0$ is defined in \eqref{eq:defw0},
 \item $\varphi(z_t(\alpha^2;\hbar),\alpha^2;\hbar)=1$, i.e., $\varphi$ maps the turning point to $1$,
 \item $\varphi$ is of the form
 $$ \varphi(z,\alpha^2;\hbar)=z[1+\alpha^2 z^2 \varepsilon_2(z,\alpha^2;\hbar)] $$
 where the function $\varepsilon_2$ satisfies the bounds
 $$ |\partial_{\alpha^2}^\ell \partial_z^k \varepsilon_2(z,\alpha^2;\hbar)|\leq C_{k,\ell} 
\langle z \rangle^{-k+2\ell} $$
 for all $z \in (0,\alpha^{-1}y_0)$, $\alpha \in (0,\alpha_0)$, $0 <\hbar\ll 1$ and $k,\ell \in \mathbb{N}_0$,
 \item $\varphi(z,\alpha^2;\hbar)\simeq z$ and $\varphi'(z,\alpha^2;\hbar)\simeq 1$ for all $z \in (0,w_0(\alpha^2;\hbar))$, $\alpha \in (0,\alpha_0)$ and $0<\hbar\ll 1$.
\end{enumerate}

\end{proposition}

\begin{proof}
 Throughout the proof we assume that $\alpha_0$ is so small that $z_t(\alpha^2;\hbar)\in (\frac23,\frac43)$ for all $\alpha \in (0,\alpha_0)$ and $0<\hbar\ll 1$.
 \begin{enumerate}
  \item As already noted, the function $\varphi(\cdot,\alpha^2;\hbar)$ is an orientation--preserving bijection between the intervals $(0,\alpha^{-1}y_0)$ and $(0,w_0(\alpha^2;\hbar))$.
  It is also clear from Eq.~\eqref{eq:LGbesselint1} that $\varphi(\cdot,\alpha^2;\hbar)$ is smooth on $(0,z_t(\alpha^2;\hbar))\cup (z_t(\alpha^2;\hbar),\alpha^{-1}y_0)$ for any $\alpha \in (0,\alpha_0)$, $0<\hbar\ll 1$ and smoothness around the turning point $z_t(\alpha^2;\hbar)$ follows from Lemmas \ref{lem:zeta} and \ref{lem:tauTP}.
  \item This is an immediate consequence of Eq.~\eqref{eq:LGbesselint1}.
  \item We distinguish different cases and start with $z$ close to zero, say, $z \in (0,\frac12)$. 
  Since we stay away from the turning point, we may drop the fractional powers and, writing $\tilde{\tau}:=-\frac23|\tau|^\frac32$, $\tilde{\zeta}:=-\frac23|\zeta|^\frac32$, we obtain $\varphi=\tilde{\zeta}^{-1}\circ \tilde{\tau}$.
  Now note that $\tilde{\zeta}$ maps $(0,\frac23)$ to $(-\infty,\tilde{\zeta}(\frac23))$ diffeomorphically. 
  Thus, $\hat{\zeta}:=\exp \circ \tilde{\zeta}: (0,\frac23) \to (0,e^{\tilde{\zeta}(\frac23)})$ diffeomorphically and,
  according to Lemma~\ref{lem:zeta}, $\hat{\zeta}$
is of the form
  $$ \hat{\zeta}(x)=e^{\tilde{\zeta}(x)}=e^{\log x-\log 2+1-x^2\varepsilon_0(x)}=\tfrac12 x e^{1-x^2\varepsilon_0(x)} $$
  which implies the bounds $|\hat{\zeta}^{(k)}(x)|\leq C_k x^{\max\{1-k,0\}}$ for all $x\in (0,\frac23)$ and $k\in \mathbb{N}_0$.
  Since $\hat{\zeta}'\simeq 1$, the inverse $\hat{\zeta}^{-1}$ satisfies the bounds
  $|(\hat{\zeta}^{-1})^{(k)}(x)|\leq C_k x^{\max\{1-k,0\}}$ for $x \in (0,e^{\tilde{\zeta}(\frac23)})$
and we note that $\hat{\zeta}^{-1}=\tilde{\zeta}^{-1} \circ \log$.
  Now consider $\hat{\tau}:=\exp \circ \tilde{\tau}: (0,\frac12) \to (0,e^{\tilde{\tau}(\frac12,\alpha^2;\hbar)})$ which, according to Lemma~\ref{lem:tau0}, looks like
  $$ \hat{\tau}(z,\alpha^2;\hbar)=e^{\tilde{\tau}(z,\alpha^2;\hbar)}=e^{\log z-\varepsilon_1(z,\alpha^2;\hbar)}=ze^{-\varepsilon_1(z,\alpha^2;\hbar)} $$
  and we obtain the bounds $|\partial_{\alpha^2}^\ell \partial_z^k \hat{\tau}(z,\alpha^2;\hbar)|\leq C_{k,\ell}z^{\max\{1-k,0\}}$ for all $z \in (0,\frac12)$, $\alpha \in (0,\alpha_0)$, $0<\hbar\ll 1$ and $k,\ell \in \mathbb{N}_0$.
  Note that 
$$ \varphi=\tilde{\zeta}^{-1}\circ \tilde{\tau}=\tilde{\zeta}^{-1}\circ \log \circ \exp \circ \tilde{\tau}=\hat{\zeta}^{-1} \circ \hat{\tau} $$
and this yields the bounds
$$ |\partial_{\alpha^2}^\ell \partial_z^k \varphi(z,\alpha^2;\hbar)|\leq C_{k,\ell}z^{\max\{1-k,0\}} $$
for all $z\in (0,\frac12)$, $\alpha \in (0,\alpha_0)$, $0<\hbar \ll 1$ and $k,\ell \in \mathbb{N}_0$.
As a consequence, since $\varphi(z,0;\hbar)=z$, the fundamental theorem of calculus yields
\begin{equation} 
\label{eq:phialpha}
\varphi(z,\alpha^2;\hbar)=z+\int_0^{\alpha^2}\partial_\beta \varphi(z,\beta;\hbar)d\beta  
\end{equation}
and we obtain
$\varphi(z,\alpha^2;\hbar)=z[1+\alpha^2 \tilde{\varepsilon}(z,\alpha^2;\hbar)]$
where
$$ |\partial_{\alpha^2}^\ell \partial_z^k \tilde{\varepsilon}(z,\alpha^2;\hbar)| \leq C_{k,\ell} $$
for all $z \in (0,\frac12)$, $\alpha \in (0,\alpha_0)$, $0<\hbar\ll 1$ and $k,\ell \in\mathbb{N}_0$.
 In order to prove that in fact $\tilde{\varepsilon}(z,\alpha^2;\hbar)=O(z^2)$, 
 we insert $\varphi(z,\alpha^2;\hbar)=z[1+\alpha^2 \tilde{\varepsilon}(z,\alpha^2;\hbar)]$ into Eq.~\eqref{eq:LGbessel} which yields
 $$ \tilde{\varepsilon}'(0,\alpha^2;\hbar)[1+\alpha^2 \tilde{\varepsilon}(0,\alpha^2;\hbar)]=0 $$
 and we infer that $\tilde{\varepsilon}'(0,\alpha^2;\hbar)=0$ since $\tilde{\varepsilon}(0,\alpha^2;\hbar)=-\frac{1}{\alpha^2}$ contradicts the above bounds on $\tilde{\varepsilon}$.
 However, this and Eq.~\eqref{eq:LGbessel} necessarily imply $\tilde{\varepsilon}(0,\alpha^2;\hbar)=0$.
 Consequently, the above bounds on $\tilde{\varepsilon}$ and the fundamental theorem of calculus show that in fact 
$\varphi(z,\alpha^2;\hbar)=z[1+\alpha^2 z^2 \varepsilon_2(z,\alpha^2;\hbar)]$ where $\varepsilon_2$ satisfies the same bounds as $\tilde{\varepsilon}$. This settles the claim for $z \in (0,\frac12)$.

Around the turning point, i.e., for $z \in (\frac13,3)$, say, it follows from Lemmas \ref{lem:zeta}, \ref{lem:tauTP}, \ref{lem:tau0} and \ref{lem:taularge} that 
$$ |\partial_{\alpha^2}^\ell \partial_z^k \varphi(z,\alpha^2;\hbar)|\leq C_{k,\ell} $$
for all $\alpha \in (0,\alpha_0)$, $0<\hbar\ll 1$, $k,\ell \in \mathbb{N}_0$ and thus, the stated bounds are a consequence of the formula \eqref{eq:phialpha} which is valid for all $z \in (0,\alpha^{-1}y_0)$.

Finally, for the large $z$ behavior, say, $z \in (2,\alpha^{-1}y_0)$, we again remove the fractional powers and use the representation $\varphi=\tilde{\zeta}^{-1} \circ \tilde{\tau}$ where this time $\tilde{\zeta}:=\tfrac23 |\zeta|^\frac32$ and $\tilde{\tau}:=\frac23 |\tau|^\frac32$.
From Lemma~\ref{lem:zeta} and Appendix~\ref{sec:symbol} it follows that $\tilde{\zeta}^{-1}$ satisfies the bounds $|(\tilde{\zeta}^{-1})^{(k)}(x)|\leq C_k x^{1-k}$ for all $x \geq c_0>0$ and $k \in \mathbb{N}_0$ where $c_0$ is arbitrary but fixed.
Furthermore, from Lemma~\ref{lem:taularge} we infer that
$|\partial_{\alpha^2}\varphi(z,\alpha^2;\hbar)|\lesssim z^3$ and thus, Eq.~\eqref{eq:phialpha} implies
$$ \varphi(z,\alpha^2;\hbar)=z[1+O(z^2 \alpha^2)]. $$
The stated derivative bounds follow from the bounds in Lemma~\ref{lem:taularge} and Eq.~\eqref{eq:phialpha}.

\item For small $z$, the estimate $\varphi'(z,\alpha^2;\hbar) \simeq 1$ follows immediately from the above and
for large $z$ it is a consequence of $\varphi=\tilde{\zeta}^{-1} \circ \tilde{\tau}$, Lemma~\ref{lem:zeta}, 
and the fact that $\tilde{\tau}'(z,\alpha^2;\hbar)\simeq 1$, uniformly in small $\alpha$, $\hbar$, 
which can be read off from the definition of $\tau$. 
The estimate $\varphi(z,\alpha^2;\hbar)\simeq z$ now follows from 
$\varphi(z,\alpha^2;\hbar)=\int_0^z \varphi'(u,\alpha^2;\hbar)du$.
 \end{enumerate}
\end{proof}

\begin{remark}
 \label{rem:w0}
 From Proposition~\ref{prop:phi} we can also read off the behavior of $w_0(\alpha^2;\hbar)$ as $\alpha \to 0+$. Indeed, we have
 $$ w_0(\alpha^2;\hbar)=
\lim_{z \to \alpha^{-1}y_0-}\varphi(z,\alpha^2;\hbar)\simeq \lim_{z \to \alpha^{-1}y_0}z \simeq \alpha^{-1}, $$
uniformly in small $\hbar$.
\end{remark}

Now we can apply the Liouville--Green transform induced by the diffeomorphism $\varphi(\cdot,\alpha^2;\hbar)$ to Eq.~\eqref{eq:besselformrescale}.

\begin{proposition}[Normal form reduction]
\label{prop:normalform}
 Assume $0<\alpha\ll 1$, $0<\hbar \ll 1$ and let $\varphi(\cdot,\alpha^2;\hbar)$ be the diffeomorphism from Proposition~\ref{prop:phi}. Furthermore, as before, set $\alpha^2=\frac{\hbar^2}{4}+4E^2$.
 Then the function $g$, defined by
$$g \left (\varphi(\tfrac{2}{\alpha}e^{-\frac{x}{2}},\alpha^2;\hbar) \right )= e^{-\frac{x}{4}}\varphi'(\tfrac{2}{\alpha}e^{-\frac{x}{2}},\alpha^2;\hbar)^{\frac12} f(x), $$ satisfies the equation
 \begin{equation}
  \label{eq:normalform}
  -\hbar_1^2 g''(w)+\left (1-\tfrac{1}{w^2} \right )g(w)=\hbar^2 V_1(w,\alpha^2;\hbar)g(w),\quad \hbar_1=\tfrac{\hbar}{\alpha}
 \end{equation}
for $w \in (0,w_0(\alpha^2;\hbar))$ if and only if $f$ is a solution to Eq.~\eqref{eq:main2} for $x>x_0$.
Here, $V_1$ is a suitable function that satisfies the bounds
$$ |\partial_{\alpha^2}^\ell \partial_w^k V_1(w,\alpha^2;\hbar)|\leq C_{k,\ell}\langle w \rangle^{-k+2\ell},\quad k,\ell \in \mathbb{N}_0 $$
in the above domain of $w,\alpha,\hbar$.
Furthermore, there exists a unique solution $g_+(\cdot,E;\hbar)$ to Eq.~\eqref{eq:normalform} with the asymptotic behavior
$$ g_+(w,E;\hbar)\sim \left (\frac{\alpha w}{2} \right )^{\frac12-2i\frac{E}{\hbar}}\quad (w \to 0+) $$
and the outgoing Jost function $f_+(\cdot,E;\hbar)$ of the operator $H(\hbar)$ is given by
$$ f_+(x,E;\hbar)=e^\frac{x}{4}\varphi'(\tfrac{2}{\alpha}e^{-\frac{x}{2}},\alpha^2;\hbar)^{-\frac12}
g_+\left ( \varphi(\tfrac{2}{\alpha}e^{-\frac{x}{2}},\alpha^2;\hbar), E;\hbar \right ). $$
\end{proposition}

\begin{remark}
The crucial point of Proposition~\ref{prop:normalform} is  that the right--hand side of Eq.~\eqref{eq:normalform} is \emph{globally} small, even after dividing by $\hbar_1^2$.
\end{remark}

\begin{proof}[Proof of Proposition~\ref{prop:normalform}]
 By construction of $\varphi(\cdot,\alpha^2;\hbar)$ (see Eqs.~\eqref{eq:LGbessel} and \eqref{eq:LG}), the function $g(\varphi(z,\alpha^2;\hbar))=\varphi'(z,\alpha^2;\hbar)^\frac12 \tilde{g}(z)$ solves Eq.~\eqref{eq:normalform} on $w \in (0,w_0(\alpha^2;\hbar))$ with 
\begin{equation} 
\label{eq:V1}
V_1(\varphi(z,\alpha^2;\hbar),\alpha^2;\hbar)=\frac{\hbar_1^2}{\hbar^2} \left (\frac34 \frac{\varphi''(z,\alpha^2;\hbar)^2}{\varphi'(z,\alpha^2;\hbar)^4}-\frac12 \frac{\varphi'''(z,\alpha^2;\hbar)}{\varphi'(z,\alpha^2;\hbar)^3} \right )
\end{equation}
if and only if $\tilde{g}$ solves Eq.~\eqref{eq:besselformrescale} on $z \in (0,\alpha^{-1}y_0)$ where $y_0=2e^{-\frac{x_0}{2}}$.
 However, this is the case if and only if $\tilde{f}(y)=\tilde{g}(\frac{y}{\alpha})$ satisfies Eq.~\eqref{eq:besselform} on $y\in (0,y_0)$ which, according to Lemma~\ref{lem:prtobessel}, is true if and only if $f(x)=e^{\frac{x}{4}}\tilde{f}(2e^{-\frac{x}{2}})$ solves Eq.~\eqref{eq:main2} on $x>x_0$.
 The asymptotic behavior of $g_+$ follows from $\varphi(z,\alpha^2;\hbar)\sim z$ and $\varphi'(z,\alpha^2;\hbar)\sim 1$ as $z \to 0+$, see Proposition~\ref{prop:phi}, which implies the asymptotics $\varphi^{-1}(w,\alpha^2;\hbar)\sim w$ and $(\varphi^{-1})'(w,\alpha^2;\hbar)\sim 1$ as $w \to 0+$ for the inverse $\varphi^{-1}(\cdot,\alpha^2;\hbar)$ of $\varphi(\cdot,\alpha^2;\hbar)$.
 
 It remains to prove the bounds on $V_1$.
 To this end, note first that, according to Proposition~\ref{prop:phi}, we have
 $$ \frac{\varphi''(z,\alpha^2;\hbar)^2}{\varphi'(z,\alpha^2;\hbar)^4}=\frac{O(\alpha^2 z)^2}{\varphi'(z,\alpha^2;\hbar)^4}=O(\alpha^4 z^2)=O(\alpha^2) $$
 and, similarly,
 $$ \frac{\varphi'''(z,\alpha^2;\hbar)}{\varphi'(z,\alpha^2;\hbar)^3}=O(\alpha^2) $$
 on $z \in (0,\alpha^{-1}y_0)$ since $\varphi'(z,\alpha^2;\hbar)\simeq 1$ and $w_0(\alpha^2;\hbar)\simeq \alpha^{-1}$, see Remark \ref{rem:w0}.
 This proves $|V_1(w,\alpha^2;\hbar)|\lesssim 1$ in the present parameter regime.
 Now note that Proposition~\ref{prop:phi} implies the bounds
 $$ |\partial_{\alpha^2}^\ell \partial_z^k \varphi(z,\alpha^2;\hbar)|\leq C_{k,\ell}\langle z \rangle^{1-k+2\ell},\quad k,\ell \in \mathbb{N}_0 $$
 for all $z \in (0,\alpha^{-1}y_0)$ and we infer inductively from
 $$ \varphi(\varphi^{-1}(w,\alpha^2;\hbar),\alpha^2;\hbar)=w $$
 the bounds
 $$ |\partial_{\alpha^2}^\ell \partial_w^k \varphi^{-1}(w,\alpha^2;\hbar)|\leq C_{k,\ell}\langle w \rangle^{1-k+2\ell},\quad k,\ell \in \mathbb{N}_0 $$
 for all $w \in (0,w_0(\alpha^2;\hbar))$, cf.~Lemma~\ref{lem:symbinv} below.
 Consequently, by setting $z=\varphi^{-1}(w,\alpha^2;\hbar)$ in Eq.~\eqref{eq:V1}, the claimed bounds on $V_1$ follow via the chain rule and the bounds on $\varphi$ from Proposition~\ref{prop:phi}, cf.~Lemma~\ref{lem:symbc}.
\end{proof}

\section{Construction of a fundamental system for the normal form equation}
\label{sec:FS}
 
In this section we construct a fundamental system to Eq.~\eqref{eq:normalform}. 
This is done by perturbing modified Bessel functions 
which are solutions to the ``homogeneous'' version (i.e., right--hand side set to zero) of Eq.~\eqref{eq:normalform}.
The fundamental system is constructed by solving suitable Volterra equations the kernels of which are composed of 
modified Bessel functions. 
This requires a good understanding of Bessel functions and we refer the reader to Appendix~\ref{sec:bessel} for the necessary results.
The main technical difficulty we are faced with comes from the fact that we need good estimates for the errors and \emph{all their derivatives}.
In order to keep things as readable as possible, we distinguish different regimes, namely exponential / oscillatory as well as  
$\frac{E}{\hbar}$ small / large.
In fact, we proceed analogously to the construction of the Bessel functions in Appendix~\ref{sec:bessel}.

\subsection{Preliminaries}
We start with an elementary result that proves to be very useful in the sequel.

\begin{lemma}
 \label{lem:perturb}
 Let $I \subset \mathbb{R}$ be an open interval and consider the inhomogeneous differential equation
 \begin{equation} 
 \label{eq:perturb}
u''(x)+V(x)u(x)=f(x)   
 \end{equation}
 for $x \in I$ with $u,V,f$ complex--valued and $V,f$ continuous on $I$.
 Suppose further there exists a function $u_0$ that does not vanish on $I$ and solves the homogeneous equation
 $$ u_0''(x)+V(x)u_0(x)=0. $$
 Define 
 $$ a(x):=\int_{x_0}^x \left [\int_{x'}^x u_0^{-2}(x'')dx'' \right ] u_0(x')f(x')dx' $$
 where $x_0 \in I$ is an arbitrary constant.
 Then $u(x):=u_0(x)[1+a(x)]$ is a solution of Eq.~\eqref{eq:perturb}.
\end{lemma}

\begin{proof}
 This can be verified by straightforward differentiation. 
\end{proof}

In the following, the ratio $2\frac{E}{\hbar}$ will appear frequently as a parameter and therefore
we use the abbreviation $\nu:=2\frac{E}{\hbar}$.
The various parameters and their dependencies are summarized in Table \ref{tbl:params}.
\begin{table}[htb]
\caption{Parameters}
\label{tbl:params}
 \begin{tabular}{c c c}
  \hline
  Parameter & Relevant range & Useful relations \\
  \hline
  $E$ & $0<E\ll 1$ & \\ 
  $\hbar$ & $0<\hbar \ll 1$ & \\
  $\alpha:=\sqrt{\frac{\hbar^2}{4}+4 E^2}$ & $0<\alpha \ll 1$ & \\
  $\hbar_1:=\frac{\hbar}{\alpha}$ & $0<\hbar_1 < 2$ & $\hbar_1=(4\frac{E^2}{\hbar^2}+\frac14)^{-\frac12}$ \\
  $\nu:=2\frac{E}{\hbar}$ & $0<\nu<\infty$ & $\nu=\frac12 \frac{\sqrt{4-\hbar_1^2}}{\hbar_1}$, $\hbar_1=(\nu^2+\frac{1}{4})^{-\frac12}$ \\
  \hline
 \end{tabular}
\end{table}

\subsection{A fundamental system of the normal form equation for small $\nu$}
We start the construction with the case $0 < \nu\lesssim 1$ which means $E \lesssim \hbar$ (and also $\hbar_1\simeq 1$).
Thus, the classical Bessel asymptotics Eqs.~\eqref{eq:besselIasym0} and \eqref{eq:besselIasyminf} become relevant. 
The first result deals with the exponential regime, i.e., to the right of the turning point and towards infinity.

\subsubsection{The exponential regime}
After rescaling by $\hbar_1=\frac{\hbar}{\alpha}$, i.e., setting $G(v):=g(\hbar_1 v)$,
the ``homogeneous'' version of Eq.~\eqref{eq:normalform} takes
the form
\begin{equation} 
\label{eq:besselG}
G''(v)+\left (\frac{\nu^2+\frac14}{w^2}-1 \right )G(v)=0 
\end{equation}
for $\nu=2\frac{E}{\hbar}$ which is a modified Bessel equation.
Eq.~\eqref{eq:besselG} has a fundamental system $\{B_j(\cdot,\nu):j=1,2\}$ of the form
\begin{align*} 
B_1(v,\nu)&=e^{-v}[1+b_1(v,\nu)] \\
B_2(v,\nu)&=e^v[1+b_2(v,\nu)]
\end{align*}
where the functions $b_j(\cdot,\nu)$ are real--valued and satisfy the estimates
$$ |\partial_\nu^\ell \partial_v^k b_j(v,\nu)|\leq C_{k,\ell}\langle v \rangle^{-1-k},\quad k,\ell\in \mathbb{N}_0 $$
for all $v \geq 1$, say, and $0\leq \nu \leq \nu_0$ (where $\nu_0>0$ is some fixed constant).
This follows by appropriate Volterra iterations and is explicitly proved in Appendix~\ref{sec:bessel}, see
in particular Lemma~\ref{lem:bessel+}.

\begin{lemma}
 \label{lem:FSsmexp}
 Let $0<E\lesssim \hbar\ll 1$ and $w_1>0$ be sufficiently large.
 Then there exists a fundamental system $\{\check{g}_j(\cdot,E;\hbar): j=1,2\}$ for Eq.~\eqref{eq:normalform} 
 on $[w_1,w_0(\alpha^2;\hbar)]$ of the form
 $$ \check{g}_j(w,E;\hbar)=B_j(\tfrac{\alpha w}{\hbar},\nu)[1+\hbar \check{c}_j(w,E;\hbar)] $$
 where $\alpha$, $\nu$ are given in Table \ref{tbl:params}.
 The error terms $\check{c}_j(\cdot,E;\hbar)$ are real--valued and satisfy the bounds
 $$ |\partial_E^\ell \partial_w^k \check{c}_j(w,E;\hbar)|\leq C_{k,\ell}\hbar^{-\ell}, \quad
 k,\ell \in \mathbb{N}_0 $$
 for all $w \in [w_1,w_0(\alpha^2;\hbar)]$, and in the above range of $E$, $\hbar$.
\end{lemma}

\begin{proof}
 We start with the growing solution $\check{g}_2(\cdot,E;\hbar)$ and according to 
 Lemma~\ref{lem:perturb} we consider the equation
 $$ \hbar \check{c}_2(w,E;\hbar)=-\int_{w_1}^w\int_v^w B_2(\tfrac{u}{\hbar_1},\nu)^{-2}du\; 
 B_2(\tfrac{v}{\hbar_1},\nu)^2 \alpha^2V_1(v,\alpha^2;\hbar)[1+\hbar \check{c}_2(v,E;\hbar)]dv $$
 for $w \in [w_1,w_0(\alpha^2;\hbar)]$ which is well--defined since $B_2(\frac{w}{\hbar_1},\nu)>0$
 for all $w\geq w_1$ provided that $w_1>0$ is sufficiently large (recall that $\hbar_1 \simeq 1$
 in the domain of $E$ and $\hbar$ which is being considered here).
Rescaling by $\hbar_1$ yields
 $$ \hbar \tilde{c}_2(x,E;\hbar)=-\int_{\hbar_1^{-1}w_1}^x \int_y^x B_2(u,\nu)^{-2}du\;
 B_2(y,\nu)^2 \hbar^2 V_1(\hbar_1 y,\alpha^2;\hbar)[1+\hbar \tilde{c}_2(y,E;\hbar)]dy $$
 for $x \in [\hbar_1^{-1}w_1,\hbar_1^{-1}w_0(\alpha^2;\hbar)]$ where $\check{c}_2(w,E;\hbar)=
 \tilde{c}_2(\tfrac{w}{\hbar_1},E;\hbar)$.
 This equation is of the form
 $$ \hbar \tilde{c}_2(x,E;\hbar)=\int_{\hbar_1^{-1}w_1}^x \int_y^x e^{-2u}a(u,E;\hbar)du\;
 e^{2y}b(y,E;\hbar)[1+\hbar \tilde{c}_2(y,E;\hbar)]dy $$
 for suitable functions $a$, $b$ satisfying
 $$
 |\partial_E^\ell \partial_x^k a(x,E;\hbar)|\leq C_{k,\ell}\langle x \rangle^{-k}\hbar^{-\ell},\quad
 |\partial_E^\ell \partial_x^k b(x,E;\hbar)|\leq C_{k,\ell}\langle x \rangle^{-k}\hbar^{2-\ell} $$
 for all $x \in [\hbar_1^{-1}w_1,\hbar^{-1}w_0(\alpha^2;\hbar)]$, $0<E\lesssim \hbar\ll 1$ and $k,\ell \in \mathbb{N}_0$.
 Furthermore, we recall the bounds $|\partial_E^\ell \hbar_1^{-1}w_1|\leq C_\ell \hbar^{-\ell}$ and
 $|\partial_E^\ell \hbar_1^{-1}w_0(\alpha^2;\hbar)|\leq C_\ell \hbar^{-1-\ell}$ in
 the relevant domain.
 Consequently, Proposition~\ref{prop:volterranox} yields the claim concerning $\check{g}_2(\cdot,E;\hbar)$.

 The second solution $\check{g}_1(\cdot,E;\hbar)$ can now be obtained by the standard reduction ansatz.
 More precisely, we set
 $$ \check{g}_1(w,E;\hbar)=-2\check{g}_2(w,E;\hbar)\int_w^{w_0(\alpha^2;\hbar)}
 \check{g}_2(v,E;\hbar)^{-2}dv $$
 for $w \in [w_1,w_0(\alpha^2;\hbar)]$.
 This is certainly well-defined provided that $w_1>0$ is sufficiently large
 and $\hbar>0$ is sufficiently small.
 It is straightforward to check that $\check{g}_1(\cdot,E;\hbar)$ is indeed of the stated form
 (cf.~Lemma~\ref{lem:Airy+}).
\end{proof}

\subsubsection{The oscillatory regime}
Next, we construct a fundamental system in the oscillatory regime.
To this end note that it follows by Frobenius' method that Eq.~\eqref{eq:besselG} 
has solutions $B_\pm(\cdot,\nu)$
of the form
$$ B_\pm(v,\nu)=v^{\frac12 \pm i\nu}[1+b_\pm(v,\nu)] $$
where the error terms satisfy $|\partial_\nu^\ell \partial_v^k b_\pm(v,\nu)|\leq C_{k,\ell}
v^{\max\{2-k,0\}}$ for all $v \in [0,v_0]$, $\nu \in (0,\nu_0]$, $k,\ell \in \mathbb{N}_0$
and some fixed $v_0,\nu_0>0$, cf.~Corollary \ref{cor:bessel0sm}.
Furthermore, it is known that $|B_\pm(v,\nu)|$ does not vanish for any $v>0$, see Corollary
\ref{cor:bessel0sm}.
Consequently, we construct a fundamental system of Eq.~\eqref{eq:normalform}
by perturbing $B_\pm(\frac{w}{\hbar_1},\nu)$.

\begin{lemma}
 \label{lem:FSsmosc}
 Let $0<E\lesssim \hbar \ll 1$ and $w_1>0$.
 Then there exists a fundamental system $\{\check{g}_{\pm}(\cdot,E;\hbar)\}$ for Eq.~\eqref{eq:normalform} 
 on $(0,w_1]$ which is of the form
 $$ \check{g}_{\pm}(w,E;\hbar)=B_\pm(\tfrac{\alpha w}{\hbar},\nu)[1+\hbar^2 \check{c}_{\pm}(w,E;\hbar)] $$
 where, as before, $\alpha^2=\frac{\hbar^2}{4}+4E^2$ and $\nu=2\frac{E}{\hbar}$.
 The error terms $\check{c}_{\pm}(\cdot,E;\hbar)$ satisfy the bounds
 $$ |\partial_E^\ell \partial_w^k \check{c}_{\pm}(w,E;\hbar)|\leq C_{k,\ell}w^{-k}\hbar^{-\ell}
 ,\quad k,\ell \in \mathbb{N}_0 $$
 for $w \in (0,w_1]$ and in the above range of $E$, $\hbar$.
 Finally, we have $\check{c}_\pm(0,E;\hbar)=\partial_w \check{c}_\pm(0,E;\hbar)=0$.
\end{lemma}

\begin{proof}
According to Lemma~\ref{lem:perturb}, we have to construct a solution to
\begin{equation}
\label{eq:volterracpm}
\hbar^2 \check{c}_\pm(w,E;\hbar)=-\int_0^w \int_v^w B_\pm(\tfrac{u}{\hbar_1},\nu)^{-2} du\;
B_\pm(\tfrac{v}{\hbar_1},\nu)^2 \alpha^2 V_1(v,\alpha^2;\hbar)[1+\hbar^2 \check{c}_\pm(v,E;\hbar)]dv
\end{equation}
with $\hbar_1=\frac{\hbar}{\alpha}$.
By introducing the new variable $w=\hbar_1 e^{-x}$ we obtain
\begin{align*} \hbar^2 \tilde{c}_\pm(x,E;\hbar)=&-\int_x^\infty \int_x^y e^{-z}B_\pm(e^{-z},\nu)^{-2}dz\; 
e^{-y}B_\pm(e^{-y},\nu)^2 \\
&\times \hbar^2 V_1(\hbar_1 e^{-y},\alpha^2;\hbar)[1+\hbar^2 \tilde{c}_\pm(y,E;\hbar)]dy 
\end{align*}
for $x \in [-\log \frac{w_1}{\hbar_1},\infty)$ where $\check{c}_\pm(w,E;\hbar)=\tilde{c}_\pm(-\log \tfrac{w}{\hbar_1},E;\hbar)$.
Now recall that $B_\pm(e^{-x},\nu)=e^{-(\frac12 \pm i\nu)x}[1+b_\pm(e^{-x},\nu)]$ and trivially, we 
have the bounds
$$ |\partial_E^\ell \partial_x^k e^{-x}[1+b_\pm(e^{-x},\nu)]|\leq C_{k,\ell}\langle x \rangle^{-k}\hbar^{-\ell}, \quad k,\ell \in \mathbb{N}_0 $$
by noting that $\partial_E \nu=2\hbar^{-1}$.
Furthermore, according to Proposition~\ref{prop:normalform} we have the bounds
$$ |\partial_E^\ell \partial_x^k \hbar^2 e^{-x}V_1(\hbar_1 e^{-x},\alpha^2;\hbar)|\leq C_{k,\ell}\langle x \rangle^{-2-k}\hbar^{2-\ell},\quad k,\ell \in \mathbb{N}_0 $$
in the relevant domain by noting that $\alpha \gtrsim \hbar$. 
This shows that the Volterra equation for $\tilde{c}_\pm$ is of the form
$$ \hbar^2 \tilde{c}_\pm(x,E;\hbar)=\int_x^\infty \int_x^y e^{(1\pm 2i\nu)z}a(z,E;\hbar)dz\; e^{-(1\pm 2i\nu)y}
b(y,E;\hbar)[1+\hbar^2 \tilde{c}_\pm(y,E;\hbar)]dy $$
where the functions $a,b$ satisfy the estimates
$$ |\partial_E^\ell \partial_x^k a(x,E;\hbar)|\leq C_{k,\ell}\langle x \rangle^{-k}\hbar^{-\ell},
\quad |\partial_E^\ell \partial_x^k b(x,E;\hbar)|\leq C_{k,\ell}\langle x \rangle^{-2-k}
\hbar^{2-\ell} $$
for all $k,\ell \in \mathbb{N}_0$ and in the relevant domain of $x$, $E$, $\hbar$.
As a consequence, we obtain from Proposition~\ref{prop:volterra} the existence of $\tilde{c}_\pm$ 
with the bounds
$$ |\partial_E^\ell \partial_x^k \tilde{c}_\pm(x,E;\hbar)|\leq C_{k,\ell}
\langle x \rangle^{-1-k}\hbar^{-\ell} $$
for all $x \in [-\log \frac{w_1}{\hbar_1},\infty)$, $0<E\lesssim \hbar\ll 1$, $k,\ell \in \mathbb{N}_0$ and
via the chain rule, these bounds translate into the claimed ones for $\check{c}_\pm$.
Finally, the fact that $\check{c}_\pm(0,E;\hbar)=\partial_w \check{c}_\pm(0,E;\hbar)=0$ follows directly
from Eq.~\eqref{eq:volterracpm}.
\end{proof}

\begin{remark}
The bounds on $\check{c}_\pm$ in Lemma~\ref{lem:FSsmosc} are not optimal. In fact, it is not hard
to see that one has
$$|\partial_E^\ell \partial_w^k \check{c}_\pm(w,E;\hbar)|\leq C_{k,\ell}w^{\max\{2-k,0\}}\hbar^{-\ell}$$
but we leave it to the interested reader to prove these stronger bounds.
\end{remark}

\subsubsection{Matching of the fundamental systems}
As a next step, we glue together the two fundamental systems obtained in Lemmas \ref{lem:FSsmosc} and \ref{lem:FSsmexp}. This amounts to calculating the Wronskians $W(\check{g}_\pm(\cdot,E;\hbar), 
\check{g}_j(\cdot,E;\hbar))$ or the asymptotics of $\check{g}_\pm(w,E;\hbar)$ for large $w$.

\begin{lemma}
 \label{lem:Wsm}
 Let $w_1>0$ be sufficiently large.
Then the functions $\check{g}_\pm(\cdot,E;\hbar)$ from Lemma~\ref{lem:FSsmosc} have the representation
\begin{align*} \check{g}_\pm(w,E;\hbar)=&\pi^{-\frac12}2^{-\frac12\pm i\nu}\Gamma(1\pm i\nu)
\Big [e^{\frac{\alpha w}{\hbar}}
[1+\hbar \check{\gamma}_2(E;\hbar)][1+\hbar \check{d}_2(w,E;\hbar)] \\
&-ie^{\pm \nu \pi}e^{-\frac{\alpha w}{\hbar}}[1+\hbar\check{\gamma}_1(E;\hbar)][1+\hbar \check{d}_1(w,E;\hbar)]\Big ] 
\end{align*}
 where the error terms $\check{d}_j(\cdot,E;\hbar)$, $j=1,2$, are real--valued and we have the bounds
 $$ |\partial_E^\ell \partial_w^k \check{d}_j(w,E;\hbar)|\leq C_{k,\ell}\hbar^{-\ell},
 \quad |\partial_E^\ell \check{\gamma}_j(E;\hbar)|\leq C_\ell \hbar^{-\ell} $$
 for all $w \in [w_1,w_0(\alpha^2;\hbar)]$, $0<E\lesssim \hbar\ll 1$ and $k,\ell \in \mathbb{N}_0$.
\end{lemma}

\begin{proof}
Since $\{\check{g}_\pm(\cdot,E;\hbar)\}$ and $\{\check{g}_j(\cdot,E;\hbar): j=1,2\}$ are fundamental systems for the same equation, there exist connection coefficients $\gamma_{\pm,j}(E;\hbar)$ such that
$$ \check{g}_\pm(w,E;\hbar)=\gamma_{\pm,1}(E;\hbar)\check{g}_1(w,E;\hbar)+\gamma_{\pm,2}(E;\hbar)\check{g}_2(w,E;\hbar) $$
and this implies
$$ \gamma_{\pm,1}(E;\hbar)=\frac{W(\check{g}_\pm(\cdot,E;\hbar),\check{g}_2(\cdot,E;\hbar))}{W(\check{g}_1(\cdot,E;\hbar),\check{g}_2(\cdot,E;\hbar))},\quad
 \gamma_{\pm,2}(E;\hbar)=-\frac{W(\check{g}_\pm(\cdot,E;\hbar),\check{g}_1(\cdot,E;\hbar))}{W(\check{g}_1(\cdot,E;\hbar),\check{g}_2(\cdot,E;\hbar))}.$$
We have 
$$ W(\check{g}_1(\cdot,E;\hbar),\check{g}_2(\cdot,E;\hbar))=\tfrac{1}{\hbar_1} 
W(B_1(\cdot,\nu),B_2(\cdot,\nu))[1+O(\hbar)]=\tfrac{2}{\hbar_1}[1+O(\hbar)] $$
where the $O$--term has the property that $\partial_E^\ell O(\hbar)
=O(\hbar^{1-\ell})$ in the domain under consideration.
Analogously, we have
$$  W(\check{g}_\pm(\cdot,E;\hbar),\check{g}_j(\cdot,E;\hbar))=\tfrac{1}{\hbar_1} W(B_\pm(\cdot,\nu),B_j(\cdot,\nu))[1+O_\mathbb{C}(\hbar)]$$
and in particular,
$$  W(\check{g}_\pm(\cdot,E;\hbar),\check{g}_1(\cdot,E;\hbar))=-\hbar_1^{-1} 
\pi^{-\frac12}2^{\frac12\pm i\nu}\Gamma(1\pm i\nu)[1+O_\mathbb{C}(\hbar)], $$
see Lemma~\ref{lem:WBB}.
This shows $\gamma_{\pm,2}(E;\hbar)=\pi^{-\frac12}2^{-\frac12\pm i\nu}\Gamma(1\pm i\nu)[1+O_\mathbb{C}(\hbar)]$ and the claim follows.
\end{proof}

Based on Lemmas \ref{lem:FSsmexp}, \ref{lem:FSsmosc} and \ref{lem:Wsm}, 
we obtain the desired representation of the semiclassical Jost function of the operator $H(\hbar)$.

\begin{corollary}
 \label{cor:Jostsm}
 Let $0<E\lesssim \hbar\ll 1$.
 Then the semiclassical Jost function $f_+(\cdot,E;\hbar)$ of the operator $H(\hbar)$ has the representation
 \begin{align*} f_+(x,E;\hbar)&=2^{-\frac12+i\nu}\hbar^{\frac12-i\nu}e^\frac{x}{4}\varphi'(\tfrac{2}{\alpha}e^{-\frac{x}{2}},\alpha^2;\hbar)^{-\frac12}\check{g}_-(\varphi(\tfrac{2}{\alpha}e^{-\frac{x}{2}}, \alpha^2;\hbar),E;\hbar) \end{align*}
 for all $x>x_0$ where, as always, $\alpha^2=\frac{\hbar^2}{4}+4E^2$, $\nu=2\frac{E}{\hbar}$ and
 the function $\check{g}_-(\cdot,E;\hbar)$ is given in Lemma~\ref{lem:FSsmosc}.

\end{corollary}

\begin{proof}
 According to Proposition~\ref{prop:normalform}, the solution $g_+(\cdot,E;\hbar)$ of Eq.~\eqref{eq:normalform}, which corresponds to the semiclassical Jost function, has the asymptotics
 $$ g_+(w,E;\hbar)\sim 2^{-\frac12+ i\nu}(\alpha w)^{\frac12-i\nu}\quad (w \to 0+). $$
 On the other hand, the solutions $\check{g}_\pm(w,E;\hbar)$ from Lemma~\ref{lem:FSsmosc} satisfy
 $$ \check{g}_\pm(w,E;\hbar)\sim B_\pm(\tfrac{\alpha w}{\hbar})\sim \hbar^{-\frac12\mp i\nu}(\alpha w)^{\frac12\pm i\nu}\quad (w \to 0+), $$
 see Corollary \ref{cor:bessel0sm}, and this implies $g_+(\cdot,E;\hbar)=2^{-\frac12+i\nu}\hbar^{\frac12-i\nu}\check{g}_-(\cdot,E;\hbar)$.
\end{proof}

\subsection{A fundamental system of the normal form equation for large $\nu$}

In this section we study the large $\nu$ case, i.e., $E \gg \hbar$.
It is important to note that in this regime the parameter $\hbar_1=\frac{\hbar}{\alpha}=(\nu^2+\frac14)^{-\frac12}$ can 
become arbitrarily small.
The construction relies on the large parameter asymptotics of the Bessel functions based on Airy functions, 
see Appendix~\ref{sec:bessel}.
Consequently, it is convenient to transform Eq.~\eqref{eq:normalform} according to Appendix~\ref{sec:bessel}.
First, we rescale and introduce the new independent variable $v:=\tfrac{w}{\hbar_1 \nu}$.
Setting $G(v)=g(\hbar_1 \nu v)$, Eq.~\eqref{eq:normalform} transforms into
\begin{equation}
\label{eq:G}
G''(v)-\nu^2(1-\tfrac{1}{v^2})G(v)+\tfrac{1}{4v^2}G(v)=-\alpha^2 \tilde{V}_1(v,E;\hbar)G(v)
\end{equation}
where $v \in (0,(\hbar_1\nu)^{-1} w_0(\alpha^2;\hbar))$ and 
$\tilde{V}_1(v,E;\hbar):=(\hbar_1 \nu)^2 V_1(\hbar_1 \nu v,\alpha^2;\hbar)$.
At this point is useful to note that $\hbar_1 \nu \simeq 1$ if $\nu \gtrsim 1$ and thus, we have rescaled
by a harmless factor.
From Proposition~\ref{prop:normalform} we have the bounds
$$ |\partial_E^\ell \partial_v^k \tilde{V}_1(v,E;\hbar)|\leq C_{k,\ell}\langle v \rangle^{-k}\hbar^{-\ell} $$
in the respective domain of $v$, $E$ and $\hbar$.
Following Appendix~\ref{sec:bessel} we now apply the Liouville--Green transform
$$ \Phi(\zeta(v))=\zeta'(v)^{\frac12}G(v) $$
with $\zeta: (0,\infty)\to \mathbb{R}$ from Lemma~\ref{lem:zeta}.
From Eq.~\eqref{eq:G} we obtain \footnote{Here we use the symbol $\zeta$ to denote both, the diffeomorphism from Lemma~\ref{lem:zeta} as well as the new independent variable.}
\begin{equation}
\label{eq:Phi}
\Phi''(\zeta)-\nu^2 \zeta \Phi(\zeta)-V_2(\zeta)\Phi(\zeta)=-\alpha^2 V_3(\zeta,E;\hbar)\Phi(\zeta) 
\end{equation}
for $\zeta \in (-\infty,\zeta_0(E;\hbar))$ where 
$\zeta_0(E;\hbar):=\zeta((\hbar_1 \nu)^{-1} w_0(\alpha^2;\hbar))\simeq \alpha^{-\frac23}$ and
\begin{equation} 
\label{eq:V3}
V_3(\zeta(v),E;\hbar):=\zeta'(v)^{-2}\tilde{V}_1(v,E;\hbar)
=\zeta(v)\frac{v^2}{v^2-1}\tilde{V}_1(v,E;\hbar).
\end{equation}
The function $V_2$ is given explicitly in Appendix~\ref{sec:bessel} and satisfies
$|V_2^{(k)}(\zeta)|\leq C_k\langle \zeta \rangle^{-2-k}$ for all $\zeta\in \mathbb{R}$ and $k\in \mathbb{N}_0$.
Note further that we have the bounds \footnote{In fact, $V_3$ decays much faster as $\zeta \to -\infty$, 
namely exponentially. However, we do not exploit this fact.}
\begin{align*}
|\partial_E^\ell \partial_\zeta^k V_3(\zeta,E;\hbar)|&\leq C_{k,\ell}
\langle \zeta \rangle^{-2-k}\hbar^{-\ell},\quad \zeta \leq 0 \\
|\partial_E^\ell \partial_\zeta^k V_3(\zeta,E;\hbar)|&\leq C_{k,\ell}
\langle \zeta \rangle^{1-k}\hbar^{-\ell},\quad \zeta \geq 0
\end{align*}
from Proposition~\ref{prop:normalform} and Eq.~\eqref{eq:V3}.

\subsubsection{The exponential regime}
It is a standard fact that the homogeneous version of Eq.~\eqref{eq:Phi} is solved by Bessel functions, see 
Appendix~\ref{sec:bessel}.
In particular, for $\zeta \geq 0$, there exists a fundamental system $\{\phi_j(\cdot,\nu): j=1,2\}$ of 
\begin{equation} 
\label{eq:Phi0}
\Phi''(\zeta)-\nu^2 \zeta \Phi(\zeta)-V_2(\zeta)\Phi(\zeta)=0 
\end{equation}
which is of the form
\begin{align*} 
\phi_1(\zeta,\nu)&=\Ai(\nu^\frac23 \zeta)[1+\nu^{-1}a_1(\zeta,\nu)] \\
\phi_2(\zeta,\nu)&=\Bi(\nu^\frac23 \zeta)[1+\nu^{-1}a_2(\zeta,\nu)]
\end{align*}
where $|a_j(\zeta,\nu)|\lesssim 1$ and 
$$ |\partial_\nu^\ell \partial_\zeta^k a_j(\zeta,\nu)|\leq C_{k,\ell}
\langle \zeta \rangle^{-\frac32-k}\nu^{-\ell},\quad \zeta \geq 1 $$
as well as
$$ |\partial_\nu^\ell a_j(0,\nu)|\leq C_\ell \nu^{-\ell},\quad |\partial_\nu^\ell \partial_\zeta a_j(0,\nu)|\leq C_\ell \nu^{\frac23-\ell} $$
for all $\nu \gg 1$ and $k,\ell \in \mathbb{N}_0$, see Lemma~\ref{lem:bessel+}.
Based on the Bessel functions $\phi_j(\cdot,\nu)$ we now construct a fundamental system of Eq.~\eqref{eq:Phi} in the exponential
regime.

\begin{lemma}
\label{lem:FSlgexp}
For $\zeta \in [0,\zeta_0(E;\hbar)]$ there exists a fundamental system $\{\Phi_1(\cdot,E;\hbar),\Phi_2(\cdot,E;\hbar)\}$ for Eq.~\eqref{eq:Phi} of the form
\begin{align*}
\Phi_1(\zeta,E;\hbar)&=\phi_1(\zeta,\nu)[1+\hbar \sigma_1(\zeta,E;\hbar)] \\
\Phi_2(\zeta,E;\hbar)&=\phi_2(\zeta,\nu)[1+\hbar \sigma_2(\zeta,E;\hbar)]
\end{align*}
with $\nu=2\frac{E}{\hbar}$.
The functions $\sigma_j$, $j=1,2$, are real--valued and $|\sigma_j(\zeta,E;\hbar)|\lesssim 1$ 
for $\zeta \in [0,\zeta_0(E;\hbar)]$, $0<\hbar\ll E\ll 1$. 
Furthermore, they satisfy the bounds
$$ |\partial_E^\ell \partial_\zeta^k \sigma_j(\zeta,E;\hbar)|\leq C_{k,\ell}
\langle \zeta \rangle^{-k}\hbar^{-\ell},\quad \zeta \in [1,\zeta_0(E;\hbar)] $$
as well as
$$ |\partial_E^\ell \sigma_j(0,E;\hbar)|\leq C_\ell \hbar^{-\ell},\quad 
|\partial_E^\ell \partial_\zeta \sigma_j(0,E;\hbar)|\leq C_\ell \nu^\frac23 \hbar^{-\ell} $$
for all $0<\hbar \ll E \ll 1$ and $k,\ell \in \mathbb{N}_0$.
\end{lemma}

\begin{proof}
We start with the growing solution and consider the Volterra equation
$$ \hbar \sigma_2(\zeta,E;\hbar)=-\alpha^2 \int_0^\zeta \int_\eta^\zeta \phi_2(\eta',\nu)^{-2}d\eta'\;
\phi_2(\eta,\nu)^2 V_3(\eta,E;\hbar)[1+\hbar \sigma_2(\eta,E;\hbar)]d\eta $$
for $\zeta \in [0,\zeta_0(E;\hbar)]$.
Rescaling by $\nu^{-\frac23}$ yields
$$ \tilde{\sigma}_2(x,E;\hbar)=-\alpha^2 \nu^{-\frac43}\int_0^x \int_y^x \phi_2(\nu^{-\frac23}u, \nu)^{-2}du\;
\phi_2(\nu^{-\frac23}y, \nu)^2 V_3(\nu^{-\frac23}y,E;\hbar)[1+\tilde{\sigma}_2(y,E;\hbar)]dy $$
for $x \in [0,\nu^\frac23 \zeta_0(E;\hbar)]$ where $\tilde{\sigma}_2(x,E;\hbar)=\hbar \sigma_2(\nu^{-\frac23}x,E;\hbar)$.
We have the bounds $|\partial_E^\ell \nu^\frac23 \zeta_0(E;\hbar)|\leq C_\ell \hbar^{-\frac23-\ell}$.
From Eq.~\eqref{eq:V3} and Proposition~\ref{prop:normalform} we infer the estimates
$$ |\partial_E^\ell \partial_y^k \alpha^2 \nu^{-\frac43} V_3(\nu^{-\frac23}y,E;\hbar)|\leq C_{k,\ell}\langle y \rangle^{-k}\hbar^{\frac43-\ell} $$
in the relevant domain of $y, E, \hbar$ and for all $k,\ell \in \mathbb{N}_0$.
Thus, 
we obtain from Proposition~\ref{prop:volterraAinox} the existence of $\tilde{\sigma}_2(\cdot,E;\hbar)$ with 
$$ |\partial_E^\ell \partial_x^k \tilde{\sigma}_2(x,E;\hbar)|\leq 
C_{k,\ell}\langle x \rangle^{\frac12-k}\hbar^{\frac43-\ell}\leq C_{k,\ell}\langle x \rangle^{-k}\hbar^{1-\ell} $$
and the stated bounds for $\sigma_1$ follow.

The solution $\Phi_1(\cdot,E;\hbar)$ 
can now be constructed by the standard reduction ansatz, i.e., by setting
$$ \Phi_1(\zeta,E;\hbar):=-\tfrac{1}{\pi}\nu^\frac23 \Phi_2(\zeta,E;\hbar)\int_\zeta^{\zeta_0(E;\hbar)}\Phi_2(\eta,E;\hbar)^{-2}d\eta. $$
The so--defined $\Phi_1(\cdot,E;\hbar)$ is clearly a decaying solution to Eq.~\eqref{eq:Phi} and by using
$W(\Ai,\Bi)=\frac{1}{\pi}$, see e.g., \cite{Miller}, it is not hard to see that $\Phi_1(\cdot,E;\hbar)$ is indeed of the stated form.
\end{proof}

\subsubsection{The oscillatory regime}
Similarly, in the oscillatory regime $\zeta\leq 0$ there exists a fundamental system $\{\phi_\pm(\cdot,\nu)\}$ of Bessel functions for Eq.~\eqref{eq:Phi0} of the form
$$ \phi_\pm(\zeta,\nu)=[\Ai(\nu^\frac23 \zeta)\pm i\Bi(\nu^\frac23 \zeta)][1+\nu^{-1}a_\pm(\zeta,\nu)] $$
where the error terms $a_\pm$ satisfy analogous bounds as $a_j$ above in the respective domain of $\zeta$ and $\nu$, see Lemma~\ref{lem:bessel-}.
By perturbing the basis $\{\phi_\pm(\cdot,\nu)\}$ we obtain a fundamental system for Eq.~\eqref{eq:Phi}.

\begin{lemma}
\label{lem:FSlgosc}
For $\zeta\leq 0$ there exists a fundamental system $\{\Phi_\pm(\cdot,E;\hbar)\}$ of Eq.~\eqref{eq:Phi} of 
the form
$$ \Phi_\pm(\zeta,E;\hbar)=\phi_\pm(\zeta,\nu)[1+\hbar \sigma_\pm(\zeta,E;\hbar)] $$
where $\nu=2\frac{E}{\hbar}$ and $|\sigma_\pm(\zeta,E;\hbar)|\lesssim 1$
for $\zeta \leq 0$ and $0<\hbar \ll E \ll 1$.
Furthermore, the error terms $\sigma_\pm$ satisfy the bounds
$$ |\partial_E^\ell \partial_\zeta^k \sigma_\pm(\zeta,E;\hbar)|\leq C_{k,\ell}
\langle \zeta \rangle^{-\frac32-k}\hbar^{-\ell},\quad \zeta \leq -1 $$
as well as
$$ |\partial_E^\ell \sigma_\pm(0,E;\hbar)|\leq C_\ell \hbar^{-\ell},\quad
|\partial_E^\ell \partial_\zeta \sigma_\pm(0,E;\hbar)|\leq C_\ell \nu^\frac23 \hbar^{-\ell} $$
for all $0<\hbar \ll E \ll 1$ and $k,\ell \in \mathbb{N}_0$.
\end{lemma}

\begin{proof}
By choosing $\nu$ sufficiently large, we obtain from Lemma~\ref{lem:bessel-} that 
$|\phi_\pm(\zeta,\nu)|>0$ for all $\zeta \leq 0$.
Thus, the function $\sigma_-$ can be constructed by solving the Volterra equation
$$ \hbar\sigma_-(\zeta,E;\hbar)=-\alpha^2 \int_{-\infty}^\zeta \int_{\eta}^\zeta \phi_-(\eta',\nu)^{-2}d\eta'\;
\phi_-(\eta,\nu)^2 V_3(\eta,E;\hbar)[1+\hbar \sigma_-(\eta,E;\hbar)]d\eta. $$
This is done via Propositions \ref{prop:volterraAiBi} and \ref{prop:volterraAinox}, 
completely analogous to Lemma~\ref{lem:FSlgexp} (or Lemma~\ref{lem:bessel-}).
$\Phi_+$ is just the complex conjugate of $\Phi_-$.
\end{proof}

\subsubsection{Matching of the fundamental systems}
Next, we glue together the two fundamental systems $\{\Phi_j(\cdot,E;\hbar): j=1,2\}$ and 
$\{\Phi_\pm(\cdot,E;\hbar)\}$.
To this end we make use of the representation
$$ \phi_\pm(\zeta,\nu)=[1+\nu^{-1}\alpha_{\pm,1}(\nu)]\phi_1(\zeta,\nu)
\pm i[1+\nu^{-1}\alpha_{\pm,2}(\nu)]\phi_2(\zeta,\nu) $$
for $\zeta\geq 0$, $\nu \gg 1$ where $|\partial_\nu^\ell \alpha_{\pm,j}(\nu)|\leq C_\ell \nu^{-\ell}$ for
all $\ell \in \mathbb{N}_0$ and $j=1,2$.
This representation follows easily from the fact that the Airy functions are defined globally, 
see Lemma~\ref{lem:bessel+bessel-} for an explicit proof.

\begin{lemma}
\label{lem:Wlg}
For $\zeta\geq 0$ the functions $\Phi_\pm(\cdot,E;\hbar)$ from Lemma~\ref{lem:FSlgosc} have the representation
\begin{align*} 
\Phi_\pm(\zeta,E;\hbar)=&[1+\nu^{-1}\beta_{\pm,1}(E;\hbar)]\Phi_1(\zeta,E;\hbar) 
\pm 
i[1+\nu^{-1}\beta_{\pm,2}(E;\hbar)]\Phi_2(\zeta,E;\hbar) 
\end{align*}
where $\Phi_j(\cdot,E;\hbar)$, $j=1,2$, are from Lemma~\ref{lem:FSlgexp} and 
$\nu=2\frac{E}{\hbar}$.
The functions $\beta_{\pm,j}$ satisfy the estimates
$$ |\partial_E^\ell \beta_{\pm,j}(E;\hbar)|\leq C_\ell \hbar^{-\ell} $$
for all $0<\hbar \ll E\ll 1$ and $\ell \in \mathbb{N}_0$.
\end{lemma}

\begin{proof}
By construction (see Lemma~\ref{lem:FSlgexp}) we have 
$W(\Phi_1(\cdot,E;\hbar),\Phi_2(\cdot,E;\hbar))=\nu^\frac23 W(\Ai, \Bi)$.
Furthermore, by evaluating the Wronskians at $\zeta=0$ and noting that $\hbar \lesssim \nu^{-1}$, 
we obtain from Lemmas \ref{lem:FSlgexp}, 
\ref{lem:FSlgosc} that
$$ W(\Phi_\pm(\cdot,E;\hbar), \Phi_j(\cdot,E;\hbar))=\nu^\frac23 W(A_\pm, A_j)
[1+\nu^{-1}\tilde{\beta}_{\pm,j}(E;\hbar)],\quad j=1,2 $$
for functions $\tilde{\beta}_{\pm,j}(E;\hbar)$ satisfying $|\partial_E^\ell \tilde{\beta}_{\pm,j}(E;\hbar)|
\leq C_\ell \hbar^{-\ell}$, $\ell \in \mathbb{N}_0$, where
$A_1:=\Ai$, $A_2:=\Bi$ and $A_\pm:=\Ai\pm i\Bi$.
The result now follows by a straightforward computation, cf.~Lemma~\ref{lem:bessel+bessel-}.
\end{proof}

Now we are ready to obtain the desired representation of the outgoing Jost function in the case of large
$\nu$.

\begin{lemma}
\label{lem:Jostlg}
For $0<\hbar \ll E\ll 1$ the semiclassical outgoing Jost function $f_+(\cdot,E;\hbar)$ of the operator $H(\hbar)$
has the representation
\begin{align*}
f_+(x,E;\hbar)=&\pi^\frac12 2^{-\frac12}e^{i(\nu+\frac{\pi}{4})}\hbar^{\frac12-i\nu}\nu^{\frac23-i\nu}
e^{\frac{x}{4}}\varphi'(\tfrac{2}{\alpha}e^{-\frac{x}{2}},\alpha^2;\hbar)^{-\frac12}\zeta'(\tfrac{\alpha}{\hbar \nu}
\varphi(\tfrac{2}{\alpha}e^{-\frac{x}{2}},\alpha^2;\hbar))^{-\frac12} \\
&\times \Phi_-(\zeta(\tfrac{\alpha}{\hbar \nu}
\varphi(\tfrac{2}{\alpha}e^{-\frac{x}{2}},\alpha^2;\hbar)),E;\hbar)
\end{align*}
where $\nu=2\frac{E}{\hbar}$, $\alpha^2=\frac{\hbar^2}{4}+4E^2$ 
and $\Phi_-(\cdot,E;\hbar)$ is given in Lemma~\ref{lem:FSlgosc}.
\end{lemma}

\begin{proof}
By construction, the functions $\hat{g}_\pm(\cdot,E;\hbar)$ defined by
$$ \hat{g}_\pm(w,E;\hbar):=\zeta'(\tfrac{w}{\hbar_1 \nu})^{-\frac12}
\Phi_\pm(\zeta(\tfrac{w}{\hbar_1 \nu}),E;\hbar) $$
with $\Phi_\pm(\cdot,E;\hbar)$ from Lemma~\ref{lem:FSlgosc}
form a fundamental system for the normal form equation \eqref{eq:normalform} on $(0,1]$ provided that 
$\nu=2\frac{E}{\hbar} \gg 1$.
According to Lemmas \ref{lem:FSlgosc}, \ref{lem:bessel-} and the standard Airy asymptotics (see 
Corollary \ref{cor:Airy-}) we have 
$$ \Phi_\pm(\zeta,E;\hbar)\sim \phi_\pm(\zeta,\nu)\sim \Ai(\nu^{\frac23}\zeta)\pm i\Bi(\nu^{\frac23}\zeta)
\sim \pi^{-\frac12}e^{\pm i\frac{\pi}{4}}(-\nu^{\frac23}\zeta)^{-\frac14}
e^{\mp i \frac23 \nu (-\zeta)^{3/2}} $$
as $\zeta \to -\infty$ and thus,
\begin{align*} \hat{g}_\pm(w,E;\hbar)&\sim \pi^{-\frac12}e^{\pm i\frac{\pi}{4}}\nu^{-\frac16}
\zeta'(\tfrac{w}{\hbar_1 \nu})^
{-\frac12}[-\zeta(\tfrac{w}{\hbar_1 \nu})]^{-\frac14}e^{\mp i  \nu \frac23[-\zeta(\frac{w}{\hbar_1 \nu})]^{3/2}} 
\end{align*}
as $w \to 0+$.
Now recall from Lemma~\ref{lem:zeta} that
\begin{align*} \tfrac23 [-\zeta(x)]^{\frac32}&= -\log x+\log 2-1+O(x^2) \\
\zeta'(x)[-\zeta(x)]^{\frac12}&=\tfrac{1}{x}+O(x)
\end{align*}
as $x \to 0+$ which imply
\begin{align*} \hat{g}_\pm(w,E;\hbar)&\sim \pi^{-\frac12}e^{\pm i\frac{\pi}{4}}\nu^{-\frac16}
\big (\tfrac{w}{\hbar_1 \nu} \big )^{\frac12}e^{\mp i \nu [-\log (\frac{w}{\hbar_1 \nu})+\log 2-1]} \\
&\sim \pi^{-\frac12}2^{\mp i\nu}e^{\pm i(\nu+\frac{\pi}{4})}
\hbar^{-\frac12 \mp i\nu}\nu^{-\frac23\mp i\nu}(\alpha w)^{\frac12\pm i\nu}
\end{align*}
as $w \to 0+$.
According to Proposition~\ref{prop:normalform}, the solution $g_+(\cdot,E;\hbar)$ of Eq.~\eqref{eq:normalform} which corresponds to the outgoing Jost function $f_+(\cdot,E;\hbar)$ of the operator $H(\hbar)$ has the asymptotics
$$ g_+(w,E;\hbar)\sim 2^{-\frac12+i\nu}(\alpha w)^{\frac12-i\nu}\quad (w \to 0+) $$
and thus, we obtain
$$ g_+(w,E;\hbar)=\pi^\frac12 2^{-\frac12}e^{i(\nu+\frac{\pi}{4})}\hbar^{\frac12-i\nu}\nu^{\frac23-i\nu}
\hat{g}_-(w,E;\hbar) $$
which yields the claim by undoing the transformations that led to Eq.~\eqref{eq:normalform}.
\end{proof}

\section{The scattering matrix and the spectral measure}

In this section we compute the scattering matrix as well as the semiclassical spectral measure
associated to $H(\hbar)$. Thereby, we complete the proofs of
Theorems \ref{thm:main1} and \ref{thm:main2}.
The most important step consists of computing the Wronskian
$W(f_-(\cdot,E;\hbar),f_+(\cdot,E;\hbar))$.
To this end we need a similar asymptotic description of $f_-(\cdot,E;\hbar)$ as we have obtained in the previous section for
$f_+(\cdot,E;\hbar)$. 
This, however, is trivial due to the assumed symmetry of the problem.
In fact, the corresponding result for $f_-(\cdot,E;\hbar)$ can be obtained from $f_+(\cdot,E;\hbar)$ by switching 
from $x$ to $-x$.
From now on we write $\varphi_+$ instead of $\varphi$ for the diffeomorphism in Proposition~\ref{prop:phi} 
and $\varphi_-$ will denote the corresponding diffeomorphism associated with the construction of $f_-(\cdot,E;\hbar)$.
Furthermore, it is useful to introduce the following new notation.
\begin{definition}
\label{def:epsilon}
By $\epsilon(x,E;\hbar)$ we denote a \emph{generic real--valued} function that satisfies the bounds
$$ |\partial_E^\ell \partial_x^k \epsilon(x,E;\hbar)|\leq C_{k,\ell}\hbar^{-k-\ell} $$
for all $k,\ell \in \mathbb{N}_0$ and in a domain of $x,E,\hbar$ that follows from the context.
Moreover, we write $\epsilon_c(x,E;\hbar)$ if the function attains complex values as well.
\end{definition}

\subsection{The Wronskian}
Based on the results of Sec.~\ref{sec:FS} we have a good description of $f_+(\cdot,E;\hbar)$ on $(x_0,\infty)$ and similarly for $f_-(\cdot,E;\hbar)$ on $(-\infty,-x_0)$.
Without loss of generality we assume that $x_0<0$ and compute the Wronskian at $x=0$.

\begin{lemma}
\label{lem:Jostat0}
For the outgoing semiclassical Jost function $f_+(0,E;\hbar)$ of the operator $H(\hbar)$ 
we have the representation
\begin{align*} 
f_+(0,E;\hbar)&=\alpha^\frac12 \gamma_+(E;\hbar)
e^{\frac{1}{\hbar}(S_+(E;\hbar)+iT_+(E;\hbar))}
\Big [ 1+\hbar \epsilon(E;\hbar)+
\epsilon_c(E;\hbar)e^{-\frac{2}{\hbar}S_+(E;\hbar)}\Big ] \\
f_+'(0,E;\hbar)&=\alpha^\frac12 \hbar^{-1}c_+(E;\hbar)\gamma_+(E;\hbar)
e^{\frac{1}{\hbar}(S_+(E;\hbar)+iT_+(E;\hbar))} 
\Big [ 1+\hbar \epsilon(E;\hbar)+
\epsilon_c(E;\hbar)e^{-\frac{2}{\hbar}S_+(E;\hbar)}\Big ]
\end{align*} 
where $c_+(E;\hbar), S_+(E;\hbar)\gtrsim 1$, $T_+(E;\hbar)\in \mathbb{R}$, $|\gamma_+(E;\hbar)|\simeq 1$,
and we have the bounds
$$ |\partial_E^\ell A(E;\hbar)|\leq C_\ell \hbar^{-\ell},\quad A\in \{c_+, \gamma_+, S_+, T_+\} $$
for all $\ell \in \mathbb{N}_0$, $0<E\ll 1$ and $0<\hbar\ll 1$.
\end{lemma}

\begin{proof}
We start with the case $0 \leq \nu \lesssim 1$, i.e., $\alpha \simeq \hbar$.
According to Corollary \ref{cor:Jostsm} and Proposition~\ref{prop:phi} we have
\begin{align*}
 f_+(x,E;\hbar)=&\tfrac12 \pi^{-\frac12}\hbar^{\frac12-i\nu}\Gamma(1-i\nu)
 [1+\hbar \check{\gamma}_2(E;\hbar)]e^{\frac{x}{4}}
 \varphi_+'(\tfrac{2}{\alpha}e^{-\frac{x}{2}},\alpha^2;\hbar)^{-\frac12} 
e^{\frac{\alpha}{\hbar}\varphi_+(\frac{2}{\alpha}e^{-\frac{x}{2}}, \alpha^2;\hbar)} \\
&\times \Big [ 1+\hbar \epsilon(x,E;\hbar)+
\epsilon_c(x,E;\hbar)e^{-2\frac{\alpha}{\hbar}\varphi_+(\frac{2}{\alpha}e^{-\frac{x}{2}}, \alpha^2;\hbar)}\Big ],
 \end{align*} 
 cf.~Definition \ref{def:epsilon},
 where, as always, $\alpha^2=\frac{\hbar^2}{4}+4E^2$ and $\nu=2\frac{E}{\hbar}$.
 In order to compute $f_+'(0,E;\hbar)$ we note that
\begin{align*} \partial_x \varphi_+'(\tfrac{2}{\alpha}e^{-\frac{x}{2}},\alpha^2;\hbar)^{-\frac12}|_{x=0}&=
\tfrac12 \alpha^{-1}\varphi_+'(\tfrac{2}{\alpha},\alpha^2;\hbar)^{-\frac32}
\underbrace{\varphi_+''(\tfrac{2}{\alpha},\alpha^2;\hbar)}_{O(\alpha)} \\
&=\epsilon(E;\hbar) \\
\partial_x e^{\frac{\alpha}{\hbar}\varphi_+(\frac{2}{\alpha}e^{-\frac{x}{2}}, \alpha^2;\hbar)}|_{x=0}&=
-\hbar^{-1}\varphi_+'(\tfrac{2}{\alpha},\alpha^2;\hbar)
e^{\frac{\alpha}{\hbar}\varphi_+(\frac{2}{\alpha}, \alpha^2;\hbar)} \\
&=\hbar^{-1}\epsilon(E;\hbar)e^{\frac{\alpha}{\hbar}\varphi_+(\frac{2}{\alpha}, \alpha^2;\hbar)}
\end{align*}	
by Proposition~\ref{prop:phi}.
Hence, we infer
\begin{align*} 
f_+'(0,E;\hbar)&=\tfrac12 \pi^{-\frac12}\hbar^{-\frac12-i\nu}\Gamma(1-i\nu)
 [1+\check{\gamma}_2(E;\hbar)]
 \varphi_+'(\tfrac{2}{\alpha},\alpha^2;\hbar)^{\frac12} 
e^{\frac{\alpha}{\hbar}\varphi_+(\frac{2}{\alpha}, \alpha^2;\hbar)} \\
&\quad \times \Big [ 1+\hbar \epsilon(E;\hbar)+
\epsilon_c(E;\hbar)e^{-2\frac{\alpha}{\hbar}\varphi_+(\frac{2}{\alpha}, \alpha^2;\hbar)}\Big ].
\end{align*}
 By setting $S_+(E;\hbar):=\alpha \varphi_+(\frac{2}{\alpha},\alpha^2;\hbar)$
 we obtain from Proposition~\ref{prop:phi} the bounds 
 $|\partial_E^\ell S_+(E;\hbar)|\leq C_\ell \hbar^{-\ell}$, $\ell \in \mathbb{N}_0$ 
 as well as $S_+(E;\hbar)\gtrsim 1$.
 Furthermore, we set $T_+(E;\hbar):=-2 E \log \hbar$ and the result follows by recalling that
 $\varphi_+'(\frac{2}{\alpha},\alpha^2;\hbar)\simeq 1$ and $\hbar \simeq \alpha$ in the parameter regime
 which we are considering.
 
Next, we turn to the case $\nu \gg 1$, i.e., $E \gg \hbar$. 
From Lemmas \ref{lem:Jostlg} and \ref{lem:Wlg} we have the representation
\begin{align*}
f_+(x,E;\hbar)=&\pi^\frac12 2^{-\frac12}e^{i(\nu+\frac{\pi}{4})}\hbar^{\frac12-i\nu}\nu^{\frac23-i\nu}
e^{\frac{x}{4}}\varphi_+'(\tfrac{2}{\alpha}e^{-\frac{x}{2}},\alpha^2;\hbar)^{-\frac12}\zeta'(\tfrac{\alpha}{\hbar \nu}
\varphi_+(\tfrac{2}{\alpha}e^{-\frac{x}{2}},\alpha^2;\hbar))^{-\frac12} \\
&\times \Big [
[1+\nu^{-1}\beta_{-,1}(E;\hbar)]\Phi_1(\zeta(\tfrac{\alpha}{\hbar \nu}
\varphi_+(\tfrac{2}{\alpha}e^{-\frac{x}{2}},\alpha^2;\hbar)),E;\hbar) \\
&\quad -
i[1+\nu^{-1}\beta_{-,2}(E;\hbar)]\Phi_2(\zeta(\tfrac{\alpha}{\hbar \nu}
\varphi_+(\tfrac{2}{\alpha}e^{-\frac{x}{2}},\alpha^2;\hbar)),E;\hbar) \Big ].
\end{align*}
Now recall from Lemma~\ref{lem:FSlgexp} (see also Lemma~\ref{lem:bessel+} and Corollary \ref{cor:Airy+}) that
\begin{align*}
\Phi_1(\zeta,E;\hbar)&=\phi_1(\zeta,\nu)[1+\hbar\sigma_1(\zeta,E;\hbar)] \\
&=\Ai(\nu^{\frac23}\zeta)[1+\nu^{-1}a_1(\zeta,\nu)][1+\hbar \sigma_1(\zeta,E;\hbar)] \\
&=(4\pi)^{-\frac12}\langle \nu^{\frac23}\zeta\rangle^{-\frac14}e^{-\nu \frac23 \zeta^{3/2}}
[1+\nu^{-1}a_1(\zeta,\nu)][1+\hbar \sigma_1(\zeta,E;\hbar)]
\end{align*}
for $\zeta \in [0,\zeta_0(E;\hbar)]$ and analogously,
$$ \Phi_2(\zeta,E;\hbar)=\pi^{-\frac12}\langle \nu^\frac23 \zeta\rangle^{-\frac14}e^{\nu \frac23 \zeta^{3/2}}
[1+\nu^{-1}a_2(\zeta,\nu)][1+\hbar \sigma_2(\zeta,E;\hbar)] $$
for real--valued $a_j(\cdot,E;\hbar), \sigma_j(\cdot,E;\hbar)$ that satisfy
$$ |\partial_E^\ell \partial_\zeta^k a_j(\zeta,E;\hbar)|\leq 
C_{k,\ell}\langle \zeta \rangle^{-\frac32-k}\nu^{-\ell},\quad |\partial_E^\ell \partial_\zeta^k 
\sigma_j(\zeta,E;\hbar)|\leq 
C_{k,\ell}\langle \zeta \rangle^{-k}\hbar^{-\ell} $$
for all $\zeta \in [1,\zeta_0(E;\hbar)]$, $k,\ell \in \mathbb{N}_0$, $j=1,2$ and in the relevant domain of $E$, $\hbar$.
Now note that for $|x|\lesssim 1$ we have 
$$ \zeta(\tfrac{\alpha}{\hbar \nu}
\underbrace{\varphi_+(\tfrac{2}{\alpha}e^{-\frac{x}{2}},\alpha^2;\hbar)}_{\simeq \alpha^{-1}})
\simeq (\hbar \nu)^{-\frac23} $$
and thus,
$$ |\nu^{-1}a_j(\zeta(\tfrac{\alpha}{\hbar \nu}\varphi_+(\tfrac{2}{\alpha}e^{-\frac{x}{2}},\alpha^2;\hbar)),
\nu)|\lesssim  \hbar.  $$
Consequently, $f_+(x,E;\hbar)$ can be written as
\begin{align*}
f_+(x,E;\hbar)=&-2^{-\frac12}ie^{i(\nu+\frac{\pi}{4})}\hbar^{\frac12-i\nu}\nu^{\frac23-i\nu}
e^{\frac{x}{4}}\varphi_+'(\tfrac{2}{\alpha}e^{-\frac{x}{2}},\alpha^2;\hbar)^{-\frac12}\zeta'(\tfrac{\alpha}{\hbar \nu}
\varphi_+(\tfrac{2}{\alpha}e^{-\frac{x}{2}},\alpha^2;\hbar))^{-\frac12} \\
&\times [1+\nu^{-1}\beta_{-,2}(E;\hbar)] \langle \nu^\frac23 \zeta(\tfrac{\alpha}{\hbar \nu}
\varphi_+(\tfrac{2}{\alpha}e^{-\frac{x}{2}},\alpha^2;\hbar))\rangle^{-\frac14}
e^{\nu\frac23 \zeta(\tfrac{\alpha}{\hbar \nu}
\varphi_+(\tfrac{2}{\alpha}e^{-\frac{x}{2}},\alpha^2;\hbar))^{3/2}} \\
&\times \Big [1+\hbar \epsilon(x, E;\hbar)+\epsilon_c(x, E;\hbar)e^{-2
\nu\frac23 \zeta(\tfrac{\alpha}{\hbar \nu}
\varphi_+(\tfrac{2}{\alpha}e^{-\frac{x}{2}},\alpha^2;\hbar))^{3/2}} \Big ].
\end{align*}
for $|x|\lesssim 1$.
By definition of $\zeta$ (see Lemma~\ref{lem:zeta}) we have
$$ \tfrac23 \zeta(x)^{\frac32}=x[1+O(x^{-1})], \quad \zeta'(x)^{-\frac12}\zeta(x)^{-\frac14}=1+O(x^{-2}) $$
for, say, $x \geq 2$ and the $O$--terms behave like symbols.
We therefore obtain
\begin{align*}
&\zeta'(\tfrac{\alpha}{\hbar \nu}
\varphi_+(\tfrac{2}{\alpha},\alpha^2;\hbar))^{-\frac12}\langle \nu^\frac23 \zeta(\tfrac{\alpha}{\hbar \nu}
\varphi_+(\tfrac{2}{\alpha},\alpha^2;\hbar))\rangle^{-\frac14}
e^{\nu\frac23 \zeta(\tfrac{\alpha}{\hbar \nu}
\varphi_+(\tfrac{2}{\alpha},\alpha^2;\hbar))^{3/2}} \\
&\quad=\nu^{-\frac16}[1+\hbar \epsilon(E;\hbar)]e^{\frac{1}{\hbar}S_+(E;\hbar)} 
\end{align*}
where 
$$ S_+(E;\hbar)=\hbar \nu\tfrac23 \zeta(\tfrac{\alpha}{\hbar \nu}
\varphi_+(\tfrac{2}{\alpha},\alpha^2;\hbar))^\frac32=\alpha \varphi_+(\tfrac{2}{\alpha},
\alpha^2;\hbar)[1+\alpha \epsilon(E;\hbar)] $$
and by noting that $\hbar \nu \simeq \alpha$ this yields
the claimed expression for $f_+(0,E;\hbar)$ with $T_+(E;\hbar)=-2E \log \alpha$.
The expression for $f_+'(0,E;\hbar)$ follows immediately by observing that differentiation
with respect to $x$ only affects real--valued terms and by the symbol behavior of the involved
quantities one picks up a factor $\hbar^{-1}$.
\end{proof}

Due to the symmetry of the problem we find the analogous representation of the
Jost function $f_-(0,E;\hbar)$ by simply replacing ``$+$'' in Lemma~\ref{lem:Jostat0} by ``$-$''
and attaching a minus sign in front of the expression for the derivative. 
Consequently, we can easily calculate the Wronskian.

\begin{corollary}
\label{cor:W}
The Wronskians $W(f_-(\cdot,E;\hbar),f_+(\cdot,E;\hbar))$ and $W(\overline{f_-(\cdot,E;\hbar)},
f_+(\cdot,E;\hbar))$ have the form
\begin{align*} W(f_-(\cdot,E;\hbar),f_+(\cdot,E;\hbar))&=c(E;\hbar)\gamma_-(E;\hbar)\gamma_+(E;\hbar)
\alpha \hbar^{-1} e^{\frac{1}{\hbar}
(S(E;\hbar)+iT(E;\hbar))}
[1+\hbar \epsilon_c(E;\hbar)] \\
W(\overline{f_-(\cdot,E;\hbar)},f_+(\cdot,E;\hbar))&=c(E;\hbar)\overline{\gamma_-(E;\hbar)}\gamma_+(E;\hbar)
\alpha \hbar^{-1} e^{\frac{1}{\hbar}
(S(E;\hbar)+iT(E;\hbar)-2iT_-(E;\hbar))} \\
&\quad \times [1+\hbar \epsilon_c(E;\hbar)]
\end{align*}
where $S=S_-+S_+$, $T=T_-+T_+$, $c=c_-+c_+$ and $c_\pm$, $\gamma_\pm$, $S_\pm$ as well as $T_\pm$ are from 
Lemma~\ref{lem:Jostat0}.
\end{corollary}

\begin{proof}
From Lemma~\ref{lem:Jostat0} and the above comments we obtain
\begin{align*} W(f_-(\cdot,E;\hbar),f_+(\cdot,E;\hbar))=&[c_-(E;\hbar)+c_+(E;\hbar)]\gamma_-(E;\hbar)\gamma_+(E;\hbar) \\
& \times\alpha \hbar^{-1}e^{\frac{1}{\hbar}(S(E;\hbar)+iT(E;\hbar))}[1+\hbar \epsilon_c(E;\hbar)] 
\end{align*}
for $S=S_-+S_+$, $T=T_-+T_+$ and similarly for $W(\overline{f_-(\cdot,E;\hbar)}, f_+(\cdot,E;\hbar))$.
\end{proof}

\subsection{Proof of Theorem \ref{thm:main1}}
Recall that the reflection and transmission amplitudes $r(E;\hbar)$ and $t(E;\hbar)$ are defined
by the relation
$$ \overline{f_-(x,E;\hbar)}=t(E;\hbar)f_+(x,E;\hbar)+r(E;\hbar)f_-(x,E;\hbar). $$
Consequently, we obtain
$$ 2i\tfrac{E}{\hbar}=W(f_-(\cdot,E;\hbar), \overline{f_-(\cdot,E;\hbar)})=t(E;\hbar)W(f_-(\cdot,E;\hbar),
f_+(\cdot,E;\hbar)) $$
and the claimed expression for $t(E;\hbar)$ follows from Corollary \ref{cor:W}.
Similarly, 
$$ W(\overline{f_-(\cdot,E;\hbar)}, f_+(\cdot,E;\hbar))=
r(E;\hbar)W(f_-(\cdot,E;\hbar),f_+(\cdot,E;\hbar)) $$
and Corollary \ref{cor:W} yields the claimed form of $r(E;\hbar)$.

\subsection{Proof of Theorem \ref{thm:main2}}

In order to prove Theorem \ref{thm:main2} note that the semiclassical spectral measure $e(0,0,E;\hbar)$
is given by
$$ e(0,0,E;\hbar)=\mathrm{Im}\left [\frac{f_-(0,E;\hbar)f_+(0,E;\hbar)}{W(f_-(\cdot,E;\hbar),
f_+(\cdot,E;\hbar))} \right ]=\mathrm{Im}\left [\frac{f_+'(0,E;\hbar)}{f_+(0,E;\hbar)}
-\frac{f_-'(0,E;\hbar)}{f_-(0,E;\hbar)} \right ]^{-1}. $$
According to Lemma~\ref{lem:Jostat0} we have
\begin{align*}
&\frac{f_+'(0,E;\hbar)}{f_+(0,E;\hbar)}-\frac{f_-'(0,E;\hbar)}{f_-(0,E;\hbar)} \\
&\quad =
[c_+(E;\hbar)+c_-(E;\hbar)]\hbar^{-1}\left [1+\hbar \epsilon(E;\hbar)+\epsilon_c(E;\hbar)e^{-\frac{2}{\hbar}S_+(E;\hbar)}
+\epsilon_c(E;\hbar)e^{-\frac{2}{\hbar}S_-(E;\hbar)} \right]
\end{align*}
and hence,
$$ \mathrm{Im}\left [\frac{f_-(0,E;\hbar)f_+(0,E;\hbar)}{W(f_-(\cdot,E;\hbar),
f_+(\cdot,E;\hbar))} \right ]=\frac{\hbar}{c_+(E;\hbar)+c_-(E;\hbar)}\left [\epsilon(E;\hbar)
e^{-\frac{2}{\hbar}S_+(E;\hbar)}+\epsilon(E;\hbar)e^{-\frac{2}{\hbar}S_-(E;\hbar)} \right ] $$
which is the claim.

\begin{appendix}

\section{Symbol calculus}
\label{sec:symbol}

Roughly speaking, a function is said to have symbol character if it behaves like a polynomial under differentiation, i.e., 
each derivative loses one power.
More precisely, we give the following definition.
\begin{definition}
A function $f: I \to \mathbb{R}$ is said to belong to the set $S_\alpha(I)$ for $\alpha \in \mathbb{R}$ and 
$I \subset \mathbb{R}$ open iff $f$ is smooth on $I$ and satisfies
$$ |f^{(k)}(x)|\leq C_k |x|^{\alpha-k} $$
for all $k \in \mathbb{N}_0$ and all $x \in I$ where $C_k>0$ is a constant that only depends on $k$.
Elements of $S_\alpha(I)$ are said to be of \emph{symbol type} (or \emph{have symbol character}, \emph{behave like symbols}).
\end{definition}

The point is that symbol behavior is preserved under the usual algebraic and differential operations.
\begin{lemma}
\label{lem:symb1}
If $f \in S_\alpha(I)$ and $g \in S_\alpha(I)$ then $f+g \in S_\alpha(I)$, $\lambda f \in S_\alpha(I)$ 
  for any $\lambda \in \mathbb{R}$, and $f' \in S_{\alpha-1}(I)$.
Furthermore, if $f \in S_\alpha(I)$ and $g \in S_\beta(I)$ then $fg \in S_{\alpha+\beta}(I)$.
\end{lemma}

\begin{proof}
The statement concerning the product $fg$ is a consequence of the Leibniz rule and the rest
 follows directly from the definition.
\end{proof}

As a very convenient fact, the symbol behavior is even preserved under composition of functions.

\begin{lemma}
 \label{lem:symbc}
 Let $f \in S_\alpha(I)$ and $g \in S_\beta(J)$ with $g(x) \in I$ and $|g(x)| \gtrsim |x|^\beta$ 
for all $x \in J$.
 Then $f \circ g \in S_{\alpha \beta}(J)$.
\end{lemma}

\begin{proof}
 Since $g$ has range in $I$, the composition $f \circ g$ is well--defined and smooth on $J$.
 Furthermore, we have $|f(g(x))|\lesssim |g(x)|^\alpha \lesssim |x|^{\alpha \beta}$ for all $x \in J$ (use either $|g(x)| \lesssim |x|^\beta$ or $|g(x)| \gtrsim |x|^{\beta}$ depending on the sign of $\alpha$).
 In order to estimate $(f \circ g)^{(k)}$ for $k \in \mathbb{N}$, we claim that
 $$ (f \circ g)^{(k)}=\sum_{j=1}^k (f^{(j)} \circ g)\tilde{g}_{j,k} $$
 for suitable functions $\tilde{g}_{j,k} \in S_{j\beta-k}(J)$, $j=1,2,\dots,k$.
 To prove this, we proceed by induction. For $k=1$ we have $(f \circ g)'=(f' \circ g)g'$, i.e., $\tilde{g}_{1,1}=g'$ and Lemma~\ref{lem:symb1} implies $\tilde{g}_{1,1} \in S_{\beta-1}(J)$.
Assuming that the claim is true for $k$, we obtain
\begin{align*}
 (f \circ g)^{(k+1)}&=\sum_{j=1}^k (f^{(j+1)} \circ g)\tilde{g}_{j,k}g'+\sum_{j=1}^k (f^{(j)} \circ g)\tilde{g}_{j,k}' \\
 &=\sum_{j=1}^{k+1}(f^{(j)} \circ g)\tilde{g}_{j,k+1} 
\end{align*}
where
\begin{align*} 
\tilde{g}_{1,k+1}&=\tilde{g}_{1,k}' \in S_{\beta-(k+1)}(J) \\
\tilde{g}_{j,k+1}&=\tilde{g}_{j-1,k}g'+\tilde{g}_{j,k}' \in S_{j\beta-(k+1)}(J),\quad j=2,3,\dots,k \\
\tilde{g}_{k+1,k+1}&=\tilde{g}_{k,k}g' \in S_{(k+1)\beta-(k+1)}(J)
\end{align*}
by Lemma~\ref{lem:symb1} and the claim follows.
Consequently, we have
$$ |f^{(j)}(g(x))\tilde{g}_{j,k}(x)|\leq \tilde{C}_k |g(x)|^{\alpha-j}|x|^{j\beta-k} 
\leq C_k |x|^{\alpha \beta-j\beta+j\beta-k}=C_k |x|^{\alpha \beta-k} $$
for all $x \in J$, $k \in \mathbb{N}$ and $j=1,2,\dots,k$ where $\tilde{C}_k, C_k>0$ are constants.
\end{proof}

We also state a simple corollary.

\begin{corollary}
Let $f \in S_0(I)$ and $|f(x)| \simeq 1$ for all $x \in I$. 
Then $\frac{1}{f} \in S_0(I)$.
\end{corollary}

Finally, the inverse of a function that behaves like a symbol inherits this property.

\begin{lemma}
\label{lem:symbinv}
 Let $\alpha \not=0$, $f \in S_{\alpha}(I)$ and suppose that 
$|f(x)|\gtrsim |x|^\alpha$, $|f'(x)|\gtrsim |x|^{\alpha-1}$ for all $x \in I$.
Assume further that $f'(x)\not=0$ for all $x \in I$.
 Then $f^{-1} \in S_{\frac{1}{\alpha}}(J)$ where $J:=f(I)$.
\end{lemma}

\begin{proof}
 Since $f'$ is nonzero on $I$, $f^{-1}: J\to I$ with $J:=f(I)$ exists as a smooth function.
 By definition of the inverse and the assumptions on $f$, we have
 $$ |y|=|f(f^{-1}(y))|\simeq |f^{-1}(y)|^\alpha $$
 which implies $|f^{-1}(y)|\lesssim |y|^{\frac{1}{\alpha}}$ for all $y \in J$.
 Furthermore, note that
 $$ |(f^{-1})'(y)|=\left |\frac{1}{f'(f^{-1}(y))} \right |\lesssim |f^{-1}(y)|^{1-\alpha} \lesssim |y|^{\frac{1}{\alpha}-1} $$
 for all $y \in J$ by the assumption on $f'$.
 The claim now follows inductively based on this formula.
\end{proof}

\section{Symbol behavior of solutions to Volterra equations}
\label{sec:volterra}

In this section we discuss how the symbol behavior carries over to solutions of certain Volterra equations. 
We also allow the kernel to depend on an additional parameter $\lambda$ which is a situation we frequently encounter in this paper. It is then important to understand derivatives with respect to this parameter as well. 
We remark that the following result covers both, oscillatory and exponential behavior of the kernel. 

\begin{proposition}
\label{prop:volterra}
Let $x_0 \in \mathbb{R}$, $\lambda_0>0$ and $\omega \in \mathbb{C}\backslash \{0\}$, $\Re(\omega) \geq 0$. 
Furthermore, assume that $a(\cdot,\lambda)$ and $b(\cdot,\lambda)$ are (possibly complex--valued) functions that satisfy
$$ |\partial_\lambda^\ell \partial_x^k a(x,\lambda)|\leq C_{k,\ell}\langle x \rangle^{-k}\lambda^{-\ell},\quad  
|\partial_\lambda^\ell \partial_x^k b(x,\lambda)|\leq C_{k,\ell}\langle x \rangle^{-\alpha-k} \lambda^{\beta-\ell}$$
for all $x \geq x_0$, $\lambda \in (0,\lambda_0)$ 
and $k,\ell \in \mathbb{N}_0$ where $\alpha >1$ and $\beta \geq 0$.
Set
$$ K(x,y,\lambda):=\int_x^y e^{2\omega u}a(u,\lambda)du\; e^{-2\omega y}b(y, \lambda). $$
 Then the equation
 \begin{equation} 
\label{eq:volterra}
\varphi(x,\lambda)=\int_x^\infty K(x,y,\lambda)[1+\varphi(y,\lambda)]dy   
 \end{equation}
 has a unique solution $\varphi(\cdot,\lambda)$ that satisfies 
$$ |\partial_\lambda^\ell \partial_x^k\varphi(x,\lambda)|
\leq C_{k,\ell} \langle x \rangle^{-\alpha+1-k}\lambda^{\beta-\ell} $$ 
for all $x\geq x_0$, $\lambda \in (0,\lambda_0)$ and $k,\ell \in \mathbb{N}_0$.
\end{proposition}

\begin{proof}
 We seek a solution $\varphi(\cdot,\lambda)$ of the Volterra equation
 $$ \varphi(x,\lambda)=f(x,\lambda)+\int_x^\infty K(x,y,\lambda)\varphi(y,\lambda)dy $$
 where
 $$ f(x,\lambda):=\int_x^\infty K(x,y,\lambda)dy. $$
 By the assumption on $a$ and two integrations by parts, we obtain
 \begin{align*}\left |\int_x^y e^{2\omega u}a(u,\lambda)du \right |
&=\tfrac{1}{2|\omega|} \left |e^{2\omega y}a(y,\lambda)-e^{2\omega x}a(x,\lambda)-\int_x^y e^{2\omega u}\partial_u a(u,\lambda)du \right | \\
&\lesssim e^{2\mathrm{Re}(\omega) y}\left |1+\int_x^y |\partial_u^2 a(u,\lambda)|du \right | \lesssim e^{2\mathrm{Re}(\omega)y} \\
 \end{align*}
 for all $x_0 \leq x \leq y$ and $\lambda \in (0,\lambda_0)$ since $\mathrm{Re}(\omega) \geq 0$. 
This implies
 $$ |K(x,y,\lambda)|\leq \left |\int_x^y e^{2\omega u}a(u,\lambda)du\right | e^{-2 \mathrm{Re}(\omega) y}|b(y,\lambda)|\lesssim |b(y,\lambda)|\lesssim 
 \langle y \rangle^{-\alpha}\lambda^\beta $$
 for all $x_0 \leq x \leq y$ and $\lambda \in (0,\lambda_0)$ by the assumption on $b$.
 Consequently, we obtain
 $$ \int_{x_0}^\infty \sup_{\{x: x_0<x<y\}}|K(x,y,\lambda)|dy \lesssim \int_{x_0}^\infty \langle y \rangle^{-\alpha}dy \lesssim 1 $$
 since $\alpha>1$, $\beta\geq 0$ and a standard Volterra iteration yields the existence of a unique $\varphi(\cdot,\lambda)$ 
 satisfying Eq.~\eqref{eq:volterra} and $|\varphi(x,\lambda)| \lesssim 1$ for all $x \geq x_0$, $\lambda \in (0,\lambda_0)$.
 However, we have $|f(x,\lambda)|\lesssim \langle x \rangle^{-\alpha+1}\lambda^\beta$ and thus, Eq.~\eqref{eq:volterra} implies that 
 in fact $|\varphi(x,\lambda)|\lesssim \langle x \rangle^{-\alpha+1}\lambda^\beta$ for all $x \geq x_0$, $\lambda \in (0,\lambda_0)$.
This settles the case $k=\ell=0$.
 
In order to prove the derivative bounds note first that
\begin{align*} K(x,\eta+x,\lambda)&=\int_x^{\eta+x}e^{2\omega u}a(u,\lambda)du\;e^{-2\omega (\eta+x)}b(\eta+x,\lambda) \\ 
&=\int_0^\eta e^{2\omega (u+x)}a(u+x,\lambda)du\;e^{-2\omega (\eta+x)}b(\eta+x,\lambda) \\
&=\int_0^\eta e^{2\omega u}a(u+x,\lambda)du\;e^{-2\omega \eta}b(\eta+x,\lambda) 
\end{align*}
and thus,
\begin{align*} 
\left |\partial_\lambda^\ell \partial_x^k K(x,\eta+x,\lambda) \right |&\leq 
C_{k,\ell}\sum_{\ell'=0}^\ell 
\sum_{k'=0}^k \left |\int_0^\eta e^{2\omega u}\partial_\lambda^{\ell'} \partial_x^{k'} a(u+x,\lambda)du\; e^{-2\omega \eta}\partial_\lambda^{\ell-\ell'}\partial_x^{k-k'}b(\eta+x,\lambda) \right |\\
&\leq C_{k,\ell}\lambda^{\beta-\ell}\sum_{k'=0}^k \langle x\rangle^{-k'}\langle \eta+x \rangle^{-\alpha-k+k'}.
\end{align*}
Now we proceed by induction in $k+\ell$.
Let $n \in \mathbb{N}$ and assume that 
$$|\partial_\lambda^{\ell'} \partial_x^{k'} \varphi(x,\lambda)|\leq C_{k',\ell'}\langle x \rangle^{-\alpha+1-k'}\lambda^{\beta-\ell'}$$
for all $k',n' \in \mathbb{N}_0$ with $k'+\ell'<n$.
We claim that this implies $| \partial_\lambda^\ell \partial_x^k\varphi(x,\lambda)|\leq C_{k,\ell}\langle x \rangle^{-\alpha+1-k}\lambda^{\beta-\ell}$ for all $k,\ell \in \mathbb{N}_0$ with $k+\ell=n$.
Indeed, we have
\begin{align*}
 \partial_\lambda^\ell \partial_x^k \varphi(x,\lambda)&=\sum_{\ell'=0}^\ell \sum_{k'=0}^k \binom{\ell}{\ell'}\binom{k}{k'}
\int_0^\infty \partial_\lambda^{\ell-\ell'}\partial_x^{k-k'}K(x,\eta+x,\lambda) \partial_\lambda^{\ell'}\partial_x^{k'}
[1+\varphi(\eta+x,\lambda)]d\eta \\
&=\sum_{k'+\ell'<n}\int_0^\infty \partial_\lambda^{\ell-\ell'}\partial_x^{k-k'}K(x,\eta+x,\lambda) \partial_\lambda^{\ell'}\partial_x^{k'}
[1+\varphi(\eta+x,\lambda)]d\eta \\
&\quad +\int_0^\infty K(x,\eta+x,\lambda)\partial_\lambda^\ell \partial_x^k \varphi(\eta+x,\lambda)d\eta \\
&=O_\mathbb{C}(\langle x \rangle^{-\alpha+1-k}\lambda^{\beta-\ell})+\int_x^\infty K(x,y,\lambda)\partial_\lambda^\ell \partial_y^k \varphi(y,\lambda)dy
\end{align*}
and a Volterra iteration yields the claim.
\end{proof}

Furthermore, we discuss a variant which is useful when there is no decay in $x$ at all. In this case, of course, $\infty$ cannot be chosen as a base point and one has to start from a finite $x_0$. 
However, in certain situations it is possible to construct the solution on an interval which becomes
infinite as an additional parameter tends to zero.

\begin{proposition}
 \label{prop:volterranox}
Fix $x_0 \in \mathbb{R}$, $\lambda_0>0$, $\alpha >-1$, $\beta\geq \gamma \geq 0$ and assume that $\beta-(\alpha+1)\gamma\geq 0$. \footnote{The last condition follows from $\beta\geq \gamma$ if $\alpha \in (-1,0]$.}
Let $c$ be a real--valued function that satisfies $c(\lambda)\geq x_0$ and $|c^{(\ell)}(\lambda)|\leq C_\ell \lambda^{-\gamma-\ell}$ for all $\lambda \in (0,\lambda_0)$, $\ell \in \mathbb{N}_0$.
Furthermore, assume that the (possibly complex--valued) functions $a(\cdot,\lambda)$, $b(\cdot,\lambda)$ satisfy the bounds
$$ |\partial_\lambda^\ell \partial_x^k a(x,\lambda)|\leq C_{k,\ell}\langle x \rangle^{-k}\lambda^{-\ell},\quad
|\partial_\lambda^\ell \partial_x^k b(x,\lambda)|\leq C_{k,\ell}\langle x \rangle^{\alpha-k}\lambda^{\beta-\ell} $$
for all $x_0\leq x \leq c(\lambda)$, $\lambda \in (0,\lambda_0)$ and $k,\ell \in \mathbb{N}_0$.
Set
$$ K(x,y,\lambda):=\int_y^x e^{-2\omega u}a(u,\lambda)du\;e^{2\omega y}b(y,\lambda) $$
for $x_0\leq y \leq x \leq c(\lambda)$ where $\omega \in \mathbb{C}\backslash \{0\}$ and $\mathrm{Re}(\omega)\geq 0$.
Then the equation
$$ \varphi(x,\lambda)=\int_{x_0}^x K(x,y,\lambda)[1+\varphi(y,\lambda)]dy $$
has a unique solution $\varphi(\cdot,\lambda)$ that satisfies
$$ |\partial_\lambda^\ell \partial_x^k \varphi(x,\lambda)|\leq C_{k,\ell}\langle x \rangle^{\alpha+1-k}\lambda^{\beta-\ell} $$
for all $x_0 \leq x \leq c(\lambda)$, $\lambda \in (0,\lambda_0)$ and $k,\ell \in \mathbb{N}_0$.
\end{proposition}

\begin{proof}
 As in the proof of Proposition~\ref{prop:volterra} we obtain the estimate
 $$ |K(x,y,\lambda)|\lesssim |b(y,\lambda)|\lesssim \langle y \rangle^{\alpha}\lambda^\beta $$
 for all $x_0 \leq y \leq x \leq c(\lambda)$, $\lambda \in (0,\lambda_0)$ and therefore,
 $$ \int_{x_0}^{c(\lambda)}\sup_{\{x: x_0<y<x<c(\lambda)\}}|K(x,y,\lambda)|dy \lesssim \lambda^{\beta-(\alpha+1)\gamma}\lesssim 1. $$
 Consequently, a standard Volterra iteration yields the existence of $\varphi(\cdot,\lambda)$
 with $|\varphi(x,\lambda)|\lesssim 1$ for all $x_0 \leq x \leq c(\lambda)$, $\lambda \in (0,\lambda_0)$ and we obtain
 $$ |\varphi(x,\lambda)|\lesssim \int_{x_0}^x |K(x,y,\lambda)|dy\lesssim \langle x \rangle^{\alpha+1}\lambda^{\beta} $$
 in the relevant domain of $x$ and $\lambda$ which settles the case $k=\ell=0$.
 
Differentiating with respect to $\lambda$ we obtain
\begin{align*} \partial_\lambda^\ell \varphi(x,\lambda)=&\sum_{\ell'=1}^\ell \binom{\ell}{\ell'}
\int_{x_0}^x \partial_\lambda^{\ell'}K(x,y,\lambda)\partial_\lambda^{\ell-\ell'}\varphi(y,\lambda)dy \\
&+\int_{x_0}^x K(x,y,\lambda)\partial_\lambda^\ell \varphi(y,\lambda)dy
\end{align*}
for any $\ell \in \mathbb{N}$ and the claim for $k=0$ follows inductively by Volterra iteration.

If $k>0$ the situation is a bit more subtle since it is not completely obvious from the onset how 
to obtain decay in $x$ for large $x$ since the Volterra iteration starts at $x=x_0$.
The effect that provides decay can be illustrated by the following simple example: while 
$\int_0^x \langle y \rangle^{-N}dy$ for $N \in \mathbb{N}$ does not decay in $x$ as $x\to \infty$, 
we do have the estimate
$$ e^{-x}\int_0^x e^y \langle y \rangle^{-N}dy \leq C_N \langle x \rangle^{-N} $$
which follows immediately by means of one integration by parts and noting that 
$y \mapsto e^y \langle y \rangle^{-N}$ is monotonically 
increasing for $y\geq N$.
In order to exploit this, one has to use some nice properties of the differentiated kernel 
$$ \tilde{K}(x,y,\lambda):=
\partial_x K(x,y,\lambda)=e^{-2\omega(x-y)}a(x,\lambda)b(y,\lambda) $$
which the original kernel $K(x,y,\lambda)$ does not enjoy.
Consequently, the appropriate starting point is the equation 
\begin{align}
\label{eq:volterradiff}
 \partial_x \varphi(x,\lambda)&=\underbrace{K(x,x,\lambda)}_{=0}[1+\varphi(x,\lambda)]+\int_{x_0}^x \tilde{K}(x,y,\lambda)
[1+\varphi(y,\lambda)]dy \nonumber \\
&=\int_0^{x-x_0}\tilde{K}(x, x-\eta,\lambda)[1+\varphi(x-\eta,\lambda)]d\eta. 
\end{align}
Now we make the following two observations.
\begin{enumerate}
\item $|\partial_\lambda^\ell \partial_x^k \tilde{K}(x,x_0,\lambda)|\leq C_{k,\ell,N}\langle x \rangle^{-N}
\lambda^{\beta-\ell}$ for all $x_0\leq x\leq c(\lambda)$, $\lambda \in (0,\lambda_0)$, $k,\ell \in \mathbb{N}_0$ and any $N \in \mathbb{N}$,
\item we have
$$ |\partial_\lambda^\ell \partial_x^k \tilde{K}(x,x-\eta,\lambda)|\leq C_{k,\ell}\lambda^{\beta-\ell}
e^{-2\Re(\omega)\eta}\sum_{k'=0}^k \langle x \rangle^{-k'}\langle x-\eta \rangle^{\alpha-k+k'} $$
and thus,
for any $m \in \mathbb{N}_0$, 
\begin{align*} \int_0^{x-x_0}|\partial_\lambda^\ell \partial_x^k \tilde{K}(x,x-\eta,\lambda)|\langle x-\eta\rangle^{-m}d\eta
\leq C_{k,\ell,m} \langle x \rangle^{\alpha-k-m}\lambda^{\beta-\ell}
 \end{align*}
 in the relevant domain of $x$, $\lambda$ and for all $k,\ell \in \mathbb{N}_0$.
\end{enumerate}
Having this in mind we immediately infer from Eq.~\eqref{eq:volterradiff} the estimate
$$|\partial_x \varphi(x,\lambda)|\lesssim \langle x \rangle^\alpha \lambda^\beta
+\langle x \rangle^{\alpha+1}\lambda^{2\beta}\lesssim \langle x \rangle^\alpha \lambda^\beta$$ 
as claimed.
Note that we have traded $\langle x \rangle$ for $\lambda^{-\beta}$ which is possible since
$\beta \geq \gamma$ by assumption.
It is now a straightforward matter to set up an induction based on Eq.~\eqref{eq:volterradiff}
that proves the remaining bounds for the higher derivatives.
\end{proof}

\section{Airy functions}
\label{sec:airy}
Apart from the results on Volterra equations with Airy kernels, 
the material in this appendix is mostly standard, see e.g~\cite{Olver}, \cite{Miller}. 

Airy's equation
\begin{equation}
\label{eq:Airy} 
y''(x)=xy(x)  
\end{equation}
is the simplest example of an equation with a turning point.
It is well-known that the Airy function
\begin{equation} 
\Ai(x):=\tfrac{1}{\pi}\int_0^\infty \cos(\tfrac13 t^3+xt)dt, \quad x \in \mathbb{R}  
\end{equation}
is a solution of Eq.~\eqref{eq:Airy}.
The second standard solution of Airy's equation \eqref{eq:Airy} is traditionally denoted by $\Bi$ and defined by
$$ \Bi(x):=\tfrac{1}{\pi}\int_0^\infty \left [\exp(-\tfrac13 t^3+xt)+\sin(\tfrac13 t^3+xt) \right ]dt $$
for $x \in \mathbb{R}$.
We have the Wronskian relation $W(\Ai,\Bi)=\frac{1}{\pi}$. 
By transforming the real integrals into contour integrals, one obtains the analytic continuations of $\Ai$ and $\Bi$ to the complex plane.
The corresponding integral representations can then be used to study the asymptotics by the method of steepest descent.
This is by now textbook material (see e.g.,~\cite{Miller} for a very nice account) and a classical result is
\begin{align*} 
\label{eq:Airyasym}
\Ai(x)&=(4 \pi)^{-\frac12} x^{-\frac14}e^{-\frac23 x^{3/2}}[1+O(x^{-\frac32})] \\
\Bi(x)&=\pi^{-\frac12}x^{-\frac14}e^{\frac23 x^{3/2}}[1+O(x^{-\frac32})] \nonumber
\end{align*}
as $x\to \infty$.
Note, however, that this approach does not automatically yield the symbol behavior of the $O$--terms which is most important for our purposes.
Consequently, we choose a different method based on the Liouville--Green transform.

\subsection{Asymptotics in the exponential regime}
For $x\geq 0$, solutions to Eq.~\eqref{eq:Airy} are expected to increase or decrease exponentially and thus,
the idea is to transform Eq.~\eqref{eq:Airy} into the form 
$$ \psi''-\psi=V \psi $$
with a small right--hand side that can be treated as a perturbation.
This is done by means of a Liouville--Green transform.
Based on Section \ref{sec:LG}, we define the desired change of variables $\xi(x)$ by $\xi'(x)^2=x$ which immediately yields $\xi(x)=\frac23 x^\frac32$.
Clearly, $\xi: (0,\infty) \to (0,\infty)$ is a diffeomorphism.
Setting $\psi(\xi(x))=\xi'(x)^\frac12 y(x)=x^\frac14 y(x)$, Eq.~\eqref{eq:Airy} transforms into
\begin{equation} 
\label{eq:Airytransf+}
\psi''(\xi)-\psi(\xi)=-\tfrac{5}{36} \xi^{-2}\psi(\xi)  
\end{equation}
and the right--hand side is indeed small for $\xi$ large.
Since we are interested in the asymptotics as $x \to \infty$, we may restrict ourselves to 
$\xi \geq 1$.
Note that the functions $\xi \mapsto e^{\pm \xi}$ are solutions to Eq.~\eqref{eq:Airytransf+} if the right--hand side is set to zero.
Now we construct a fundamental system to Eq.~\eqref{eq:Airytransf+} by perturbing $\xi \mapsto e^{\pm \xi}$.

\begin{lemma}
\label{lem:Airy+}
 Eq.~\eqref{eq:Airytransf+} has a fundamental system $\{\psi_\pm\}$ of the form
 $$ \psi_\pm(\xi)=e^{\pm \xi}[1+a_\pm(\xi)] $$
 where the functions $a_\pm$ are real--valued and satisfy the bounds 
$$ |a_\pm^{(k)}(\xi)|\leq C_k \xi^{-1-k} $$
for all $\xi \geq 1$ and $k \in \mathbb{N}_0$.
\end{lemma}

\begin{proof}
We start with $\psi_-$.
 According to Lemma~\ref{lem:perturb}, it suffices to consider the Volterra equation
 \begin{equation}
\label{eq:Airyvolt+} 
a_-(\xi)=\int_\xi^\infty \int_\xi^\eta e^{2 \eta'}d\eta'\; 
 e^{-2\eta}(-\tfrac{5}{36})\eta^{-2}[1+a_-(\eta)]d\eta.
 \end{equation}
 Consequently, Proposition~\ref{prop:volterra} yields the existence of $a_-$ satisfying the stated bounds.
 
 In order to construct the increasing solution $\psi_+$, we use the standard reduction ansatz, i.e., we set
 $$ \psi_+(\xi)=2\psi_-(\xi)\int_c^\xi \psi_-(\eta)^{-2}d\eta $$
 which is certainly well--defined for large enough $c$ and $\xi \geq c$ and it is a solution to
 Eq.~\eqref{eq:Airytransf+}.
 If we define $a_+$ by $\psi_+(\xi)=e^\xi[1+a_+(\xi)]$ it follows that
 $$
 a_+(\xi)=2e^{-2\xi}[1+a_-(\xi)]\int_c^\xi e^{2\eta}[1+\tilde{a}_-(\eta)]d\eta-1
 $$
 where $\tilde{a}_-$ satisfies $|\tilde{a}_-^{(k)}(\xi)|\leq C_k \xi^{-1-k}$ for all $\xi\geq 1$ and
 $k\in\mathbb{N}_0$ (cf.~Appendix~\ref{sec:symbol}).
 Consequently, we obtain
 \begin{align*}
 a_+(\xi)&=e^{-2\xi}[1+a_-(\xi)](e^{2\xi}-e^{2c})+2e^{-2\xi}[1+a_-(\xi)]
 \int_c^\xi e^{2\eta}\tilde{a}_-(\eta)d\eta-1 \\
 &=a_-(\xi)-e^{2c}e^{-2\xi}[1+a_-(\xi)]+2[1+a_-(\xi)]e^{-2\xi}\int_c^\xi e^{2\eta}\tilde{a}_-(\eta)d\eta.
 \end{align*}
 Now we have
 $$ e^{-2\xi}\int_c^x e^{2\eta}\tilde{a}_-(\eta)d\eta=\tfrac12 \tilde{a}_-(\xi)-\tfrac12 e^{2c}e^{-2\xi}
 \tilde{a}_-(\xi)-\tfrac12 e^{-2\xi}\int_c^\xi e^{2\eta}\tilde{a}_-(\eta)d\eta=O(\xi^{-1}) $$
 by noting that $e^{2\eta}\tilde{a}_-(\eta)$ is monotonically increasing for large $\eta$
 and the $O$--term behaves like a symbol by Appendix~\ref{sec:symbol}.
 This proves the claimed bounds for $a_+$ and by solving an initial value problem, the solution
 $\psi_+$ can be smoothly extended to $\xi\geq 1$.
 \end{proof}
 
 \begin{corollary}
 \label{cor:Airy+}
  For $x\geq 0$, the Airy functions $\Ai$ and $\Bi$ can be written as
  \begin{align*}
   \Ai(x)&=(4\pi)^{-\frac12}\langle x\rangle^{-\frac14}e^{-\frac23 x^{3/2}}[1+a(x)] \\
   \Bi(x)&=\pi^{-\frac12}\langle x \rangle^{-\frac14}e^{\frac23 x^{3/2}}[1+b(x)]
  \end{align*}
  where $a$ and $b$ are real--valued and satisfy
  $$ |a^{(k)}(x)|\leq C_k \langle x \rangle^{-\frac32-k}, \quad
  |b^{(k)}(x)|\leq C_k \langle x \rangle^{-\frac32-k}$$
for all $x \geq 0$ and $k \in \mathbb{N}_0$.
Furthermore, $\Ai(x)$ and $\Bi(x)$ do not vanish for $x\geq 0$.
 \end{corollary}
 
 \begin{proof}
  Let $\psi_\pm$ be as in Lemma~\ref{lem:Airy+} and set
$$y_\pm(x):=x^{-\frac14}\psi_\pm(\xi(x))=x^{-\frac14}e^{\pm \frac23 x^{3/2}}[1+a_\pm(\xi(x))]. $$
  By construction, $\{y_\pm\}$ is a fundamental system for Airy's equation \eqref{eq:Airy} and thus, there exist constants $c_\pm$ such that $\Ai(x)=c_- y_-(x)+c_+ y_+(x)$.
  However, since $\Ai(x)$ approaches zero as $x \to \infty$ whereas $y_+(x)$ grows unbounded, it is immediately clear that $c_+=0$.
  Similarly, $\Bi$ must be of the form $\Bi(x)=d_- y_-(x)+d_+ y_+(x)$ for constants $d_\pm$.
  This yields the stated representations for $x \gtrsim 1$. 
  However, the fundamental system $\{y_\pm\}$ can be smoothly extended to $[0,\infty)$ since Airy's equation is regular there.
  The fact that $\Ai(x)$ and $\Bi(x)$ do not vanish for $x\geq 0$ is well--known and follows from the respective integral representations. 
 \end{proof}

\subsection{Asymptotics in the oscillatory regime}
Due to the sign change of the right--hand side, solutions to Airy's equation \eqref{eq:Airy} oscillate on the negative real axis and therefore, their asymptotic behavior as $x\to -\infty$ differs radically from Corollary \ref{cor:Airy+}.
Indeed, the classical result \cite{Olver} is
\begin{align*}
 \Ai(-x)\pm i\Bi(-x)\sim \pi^{-\frac12}e^{\pm i\frac{\pi}{4}}x^{-\frac14}e^{\mp i\frac23 x^{3/2}}
\end{align*}
as $x\to \infty$.
In order to estimate the error terms, it is convenient to switch from $x$ to $-x$, i.e., we consider the Airy equation
\begin{equation}
 \label{eq:Airy-}
 y''(x)=-xy(x)
\end{equation}
on $x \geq 0$.
Note that $\Ai(-x)$ and $\Bi(-x)$ are solutions to Eq.~\eqref{eq:Airy-}.
As before, we apply the Liouville--Green transform $\psi(\xi(x))=x^\frac14 y(\xi(x))$ with $\xi(x)=\frac23 x^\frac32$ which yields
\begin{equation}
 \label{eq:Airytransf-}
 \psi''(\xi)+\psi(\xi)=-\tfrac{5}{36}\xi^{-2}\psi(\xi)
\end{equation}
and this time, the appropriate approximate solutions are $\xi \mapsto e^{\pm i\xi}$.

\begin{lemma}
 There exists a fundamental system $\{\psi_\pm\}$ for Eq.~\eqref{eq:Airytransf-} of the form
 $$ \psi_\pm(\xi)=e^{\pm i\xi}[1+a_\pm(\xi)] $$
 where the functions $a_\pm$ satisfy
 $$ |a_\pm^{(k)}(\xi)|\leq C_k \xi^{-1-k} $$
 for all $\xi \geq 1$ and $k\in \mathbb{N}_0$.
\end{lemma}

\begin{proof}
The existence of $\psi_-$ with the stated properties follows from Proposition~\ref{prop:volterra} and 
$\psi_+$ is just the complex conjugate of $\psi_-$.
\end{proof}

As a consequence, we immediately obtain the following representation of $\Ai\pm i\Bi$ in the oscillatory regime.

\begin{corollary}
 \label{cor:Airy-}
 For $x \geq 0$, the Airy functions $\Ai \pm i\Bi$ have the representation
 $$ \Ai(-x)\pm i\Bi(-x)=\pi^{-\frac12}e^{\pm i\frac{\pi}{4}}\langle x \rangle^{-\frac14}e^{\mp i\frac23 x^{3/2}}[1+a_\pm(x)] $$
 where the $a_\pm$ satisfy
 $$ |a_\pm^{(k)}(x)|\leq C_k\langle x \rangle^{-\frac32-k} $$
 for all $x\geq 0$ and $k \in \mathbb{N}_0$.
\end{corollary}

\subsection{Volterra equations with kernels composed of Airy functions}

For equations with turning points, the Airy functions are of fundamental importance. 
Consequently, they frequently appear in the kernels of Volterra equations in perturbation theory. 
The following result is the analog to Proposition~\ref{prop:volterra} in the exponential case $\omega=1$.

\begin{proposition}
\label{prop:volterraAi}
Fix $\lambda_0>0$ and suppose $a(\cdot,\lambda)$, $b(\cdot,\lambda)$ are (possibly complex--valued) functions that satisfy
$$ |\partial_\lambda^\ell \partial_x^k a(x,\lambda)|\leq C_{k,\ell}\langle x \rangle^{-k}\lambda^{-\ell},\quad  
|\partial_\lambda^\ell \partial_x^k b(x,\lambda)|\leq C_{k,\ell}\langle x \rangle^{-\alpha-k} \lambda^{\beta-\ell}$$
for all $x \geq 0$, $\lambda \in (0,\lambda_0)$ 
and $k,\ell \in \mathbb{N}_0$ where $\alpha>\frac12$, $\beta \geq 0$.
Set
$$ K(x,y,\lambda):=\int_x^y \Ai(u)^{-2}a(u,\lambda)du\; 
\Ai(y)^2 b(y, \lambda). $$
 Then the equation
 \begin{equation*} 
\varphi(x,\lambda)=\int_x^\infty K(x,y,\lambda)[1+\varphi(y,\lambda)]dy   
 \end{equation*}
 has a unique solution $\varphi(\cdot,\lambda)$ that satisfies 
$$ |\partial_\lambda^\ell \partial_x^k\varphi(x,\lambda)|
\leq C_{k,\ell} \langle x \rangle^{-\alpha+\frac12-k}\lambda^{\beta-\ell} $$ 
for all $x\geq 0$, $\lambda \in (0,\lambda_0)$ and $k,\ell \in \mathbb{N}_0$.
\end{proposition} 

\begin{remark}
We emphasize two important differences between the exponential and the Airy case. 
First, the required decay in $x$ of the function $b(x,\lambda)$ is weaker. In the Airy case one only
needs $\alpha>\frac12$ whereas in the exponential case one has to require $\alpha>1$ (see Proposition
\ref{prop:volterra}).
Second, the loss of decay in $x$ of the solution $\varphi(x,\lambda)$ compared to $b(x,\lambda)$
is weaker. In the Airy case one only loses $\langle x \rangle^\frac12$ as opposed to $\langle x \rangle$
in the exponential case.
\end{remark}

\begin{proof}[Proof of Proposition~\ref{prop:volterraAi}]
 The strategy is to reduce the problem to Proposition~\ref{prop:volterra}.
 Define $\eta: \mathbb{R} \to \mathbb{R}$ by $\eta(\xi):=(\frac32)^\frac23 \xi \langle \xi \rangle^{-\frac13}$.
 Then $\eta$ is a diffeomorphism and we have 
$\eta(\xi)=(\frac32)^\frac23 \xi^\frac23$ for $\xi \geq 1$.
Furthermore, we note that $\eta'(\xi) \simeq \langle \xi \rangle^{-\frac13}$ and $|\eta^{(k)}(\xi)|\leq C_k \langle \xi \rangle^{\frac23-k}$ for all $\xi \in \mathbb{R}$ and $k \in \mathbb{N}_0$.
We intend to study the equation
\begin{align*} 
\varphi(x,\lambda)&=\int_x^\infty K(x,y,\lambda)[1+\varphi(y,\lambda)]dy=\int_{\eta^{-1}(x)}^\infty K(x,\eta(\xi'),\lambda)[1+\varphi(\eta(\xi'),\lambda)]\eta'(\xi')d\xi'.
\end{align*}
Hence, writing $\tilde{\varphi}(\xi,\lambda)=\varphi(\eta(\xi),\lambda)$, we obtain
\begin{equation} 
\label{eq:volterraAitilde}
\tilde{\varphi}(\xi,\lambda)=\int_\xi^\infty \tilde{K}(\xi,\xi',\lambda)[1+\tilde{\varphi}(\xi',\lambda)]d\xi' 
\end{equation}
where $\tilde{K}(\xi,\xi',\lambda):=K(\eta(\xi),\eta(\xi'),\lambda)\eta'(\xi')$.
By assumption, the kernel $\tilde{K}$ is of the form
\begin{align*} \tilde{K}(\xi,\xi',\lambda)&=\int_{\eta(\xi)}^{\eta(\xi')}\Ai(u)^{-2}
a(u,\lambda)du\;
 \Ai(\eta(\xi'))^2 b(\eta(\xi'),\lambda) \eta'(\xi') \\
 &=\int_{\xi}^{\xi'}\Ai(\eta(\xi''))^{-2}a(\eta(\xi''),\lambda)\eta'(\xi'')d\xi'' 
 \;
 \Ai(\eta(\xi'))^2 b(\eta(\xi'),\lambda) \eta'(\xi'). 
 \end{align*}
 According to Corollary \ref{cor:Airy+}, we have the representation
 $$ \Ai(\eta(\xi))^{\pm 2}
=(4\pi)^{\mp 1}\langle \eta(\xi)\rangle^{\mp \frac12}e^{\mp 2\xi}[1+a_\pm(\xi,\lambda)] $$
 where $|\partial_\lambda^\ell \partial_x^k a_\pm(\xi,\lambda)|\leq C_{k,\ell} \langle \xi \rangle^{-1-k}
\lambda^{-\ell}$ for all $\xi\geq 0$, $\lambda \in (0,\lambda_0)$ and $k,\ell\in \mathbb{N}_0$.
 Writing 
\begin{align*}\tilde{a}(\xi,\lambda)&=a(\eta(\xi),\lambda)\eta'(\xi)\langle \eta(\xi)\rangle^\frac12
[1+a_-(\xi)] \\ 
\tilde{b}(\xi,\lambda)&=b(\eta(\xi),\lambda)\eta'(\xi)\langle \eta(\xi)\rangle^{-\frac12}[1+a_+(\xi)],
 \end{align*}
we obtain the bounds
 $$ |\partial_\lambda^\ell \partial_\xi^k \tilde{a}(\xi,\lambda)|\leq C_{k,\ell}\langle \xi \rangle^{-k}\lambda^{-\ell}, \quad |\partial_\lambda^\ell \partial_\xi^k \tilde{b}(\xi,\lambda)|\leq C_{k,\ell}\langle \xi \rangle^{-\frac23\alpha-\frac23-k}\lambda^{\beta-\ell} $$
 for all $\xi \geq 0$, $\lambda \in (0,\lambda_0)$ and $k,\ell \in \mathbb{N}_0$, see Lemma~\ref{lem:symbc}, and $\tilde{K}$ can be written as
 $$ \tilde{K}(\xi,\xi',\lambda)=\int_\xi^{\xi'}e^{2\xi''}\tilde{a}(\xi'',\lambda)d\xi''\; e^{-2\xi'}\tilde{b}(\xi',\lambda). $$
Consequently, Proposition~\ref{prop:volterra} applies to Eq.~\eqref{eq:volterraAitilde} and we obtain the bounds 
$$|\partial_\lambda^\ell \partial_\xi^k \tilde{\varphi}(\xi,\lambda)|\leq C_{k,\ell}
\langle \xi \rangle^{-\frac23 \alpha+\frac13-k}\lambda^{\beta-\ell}$$
which, via Lemma~\ref{lem:symbc}, translate into the stated bounds on $\varphi$. 
\end{proof}

Next, we consider the oscillatory case. 

\begin{proposition}
 \label{prop:volterraAiBi}
 Let $\lambda_0>0$ and suppose $a(\cdot,\lambda)$, $b(\cdot,\lambda)$ are (possibly complex--valued) functions that satisfy
$$ |\partial_\lambda^\ell \partial_x^k a(x,\lambda)|\leq C_{k,\ell}\langle x \rangle^{-k}\lambda^{-\ell},\quad  
|\partial_\lambda^\ell \partial_x^k b(x,\lambda)|\leq C_{k,\ell}\langle x \rangle^{-\alpha-k} \lambda^{\beta-\ell}$$
for all $x \geq 0$, $\lambda \in (0,\lambda_0)$ 
and $k,\ell \in \mathbb{N}_0$ where $\alpha>\frac12$, $\beta \geq 0$.
Set
$$ K(x,y,\lambda):=\int_x^y [\Ai(-u)\pm i\Bi(-u)]^{-2}a(u,\lambda)du\; 
[\Ai(-y)\pm i\Bi(-y)]^2 b(y, \lambda). $$
 Then the equation
 \begin{equation*} 
\varphi(x,\lambda)=\int_x^\infty K(x,y,\lambda)[1+\varphi(y,\lambda)]dy   
 \end{equation*}
 has a unique solution $\varphi(\cdot,\lambda)$ that satisfies 
$$ |\partial_\lambda^\ell \partial_x^k\varphi(x,\lambda)|
\leq C_{k,\ell} \langle x \rangle^{-\alpha+\frac12-k}\lambda^{\beta-\ell} $$ 
for all $x\geq 0$, $\lambda \in (0,\lambda_0)$ and $k,\ell \in \mathbb{N}_0$.
\end{proposition}

\begin{proof}
 Based on the representation of $\Ai\pm i\Bi$ from Corollary \ref{cor:Airy-}, the problem can be reduced to
 $$ \tilde{\varphi}(\xi,\lambda)=\int_\xi^\infty \tilde{K}(\xi,\xi',\lambda)[1+\tilde{\varphi}(\xi',\lambda)],\quad \xi \geq 0 $$
 where 
 $$ \tilde{K}(\xi,\xi',\lambda)=\int_\xi^{\xi'}e^{\pm 2i \xi''}\tilde{a}(\xi'',\lambda)d\xi''\;
 e^{\mp 2i \xi'}\tilde{b}(\xi',\lambda) $$
 and $\tilde{\varphi}$, $\tilde{a}$, $\tilde{b}$ are like in the proof of Proposition~\ref{prop:volterraAi}.
 Thus, as before, the claim follows from Proposition~\ref{prop:volterra}.
\end{proof}

Finally, we state the ``Airy version'' of Proposition~\ref{prop:volterranox}.

\begin{proposition}
 \label{prop:volterraAinox}
Fix $x_0 \in \mathbb{R}$, $\lambda_0>0$, $\alpha >-\frac12$, $\beta\geq \frac32 \gamma \geq 0$ and assume that $\beta-(\alpha+\frac12)\gamma\geq 0$. 
Let $c$ be a real--valued function that satisfies $c(\lambda)\geq x_0$ and $|c^{(\ell)}(\lambda)|\leq C_\ell \lambda^{-\gamma-\ell}$ for all $\lambda \in (0,\lambda_0)$, $\ell \in \mathbb{N}_0$.
Furthermore, assume that the (possibly complex--valued) functions $a(\cdot,\lambda)$, $b(\cdot,\lambda)$ satisfy the bounds
$$ |\partial_\lambda^\ell \partial_x^k a(x,\lambda)|\leq C_{k,\ell}\langle x \rangle^{-k}\lambda^{-\ell},\quad
|\partial_\lambda^\ell \partial_x^k b(x,\lambda)|\leq C_{k,\ell}\langle x \rangle^{\alpha-k}\lambda^{\beta-\ell} $$
for all $x_0\leq x \leq c(\lambda)$, $\lambda \in (0,\lambda_0)$ and $k,\ell \in \mathbb{N}_0$.
Set
$$ K(x,y,\lambda):=\int_x^y \Bi(u)^{-2}a(u,\lambda)du\;\Bi(y)^2 b(y,\lambda) $$
for $x_0\leq y \leq x \leq c(\lambda)$.
Then the equation
$$ \varphi(x,\lambda)=\int_{x_0}^x K(x,y,\lambda)[1+\varphi(y,\lambda)]dy $$
has a unique solution $\varphi(\cdot,\lambda)$ that satisfies
$$ |\partial_\lambda^\ell \partial_x^k \varphi(x,\lambda)|\leq C_{k,\ell}\langle x \rangle^{\alpha+\frac12-k}\lambda^{\beta-\ell} $$
for all $x_0 \leq x \leq c(\lambda)$, $\lambda \in (0,\lambda_0)$ and $k,\ell \in \mathbb{N}_0$.
\end{proposition}

\begin{proof}
Analogous to the proof of Proposition~\ref{prop:volterraAi}, the Volterra equation for $\varphi$ we want to solve
can be rewritten as
$$\tilde{\varphi}(\xi,\lambda)=\int_{\xi_0}^\xi \tilde{K}(\xi,\xi',\lambda)
[1+\tilde{\varphi}(\xi',\lambda)]d\xi',\quad \xi \in [\xi_0,\eta^{-1}(c(\lambda))] $$
where $\eta(\xi):=(\frac32)^\frac23\xi\langle \xi \rangle^{-\frac13}$, $\tilde{\varphi}(\xi,\lambda):=\varphi(\eta(\xi),\lambda)$, $\xi_0:=\eta^{-1}(x_0)$ 
 and
$$ \tilde{K}(\xi,\xi',\lambda)=\int_{\xi'}^{\xi}e^{-2\xi''}\tilde{a}(\xi'',\lambda)d\xi''\;
e^{2\xi'}\tilde{b}(\xi',\lambda). $$
The functions $\tilde{a}$ and $\tilde{b}$ are given explicitly in the proof of Proposition~\ref{prop:volterraAi} in terms of $a$ and $b$ and they satisfy the bounds
$$ |\partial_\lambda^\ell \partial_\xi^k \tilde{a}(\xi,\lambda)|\leq C_{k,\ell}\langle \xi \rangle^{-k}
\lambda^{-\ell},\quad
|\partial_\lambda^\ell \partial_\xi^k \tilde{b}(\xi,\lambda)|\leq C_{k,\ell}\langle \xi \rangle^{\frac23
\alpha-\frac23-k}
\lambda^{\beta-\ell} $$
in the relevant domain.
Furthermore, we have the bounds $|\partial_\lambda^\ell \eta^{-1}(c(\lambda))|\leq C_\ell \lambda^{-\frac32 \gamma-\ell}$ for all $\lambda \in (0,\lambda_0)$.
Consequently, Proposition~\ref{prop:volterranox} yields the existence of $\tilde{\varphi}(\cdot,\lambda)$ satisfying the bounds 
$$|\partial_\lambda^\ell \partial_\xi^k \tilde{\varphi}(\xi,\lambda)|\leq C_{k,\ell}
\langle \xi \rangle^{\frac23 \alpha+\frac13-k}\lambda^{\beta-\ell}$$
and the claim follows.
\end{proof}

\section{Modified Bessel functions}
\label{sec:bessel}

For the convenience of the reader we provide the essential details of the asymptotic theory for the modified Bessel functions.
The following results are mostly standard, see Chapter~7, \S8, Chapter 10, \S~7, and Chapter~11, \S~10 of \cite{Olver}, 
as well as \cite{Balogh}, \cite{Dunster}, \cite{Temme}. 
However, as before in Appendix~\ref{sec:airy}, our focus is a bit different and we derive
the necessary results based on the theory developed in Appendices \ref{sec:symbol}, \ref{sec:volterra}
and \ref{sec:airy}.

Consider the modified Bessel equation
\begin{equation}
\label{eq:modbessel}
 w^2 y''(w) + wy'(w) + (\nu^2 - w^2) y(w) =0
\end{equation}
where in general $w, \nu\in\C$, however, for our purposes it suffices to consider $w>0$, $\nu \in \mathbb{R}$. 
In particular, we are interested in the behavior for $\nu$ large.
Solutions of Eq.~\eqref{eq:modbessel} can be obtained by a standard Frobenius ansatz around the regular singular point $w=0$ which yields the \emph{modified Bessel function} $I_{i\nu}$, given by
\begin{equation}
 \label{eq:besselIasym0}
I_{i\nu}(w)=(\tfrac12 w)^{i\nu}\sum_{k=0}^\infty \frac{(\tfrac14 w^2)^k}{k!\Gamma(i\nu+k+1)}=\frac{2^{-i\nu}}{\Gamma(1+i\nu)}w^{i\nu}[1+b(w, \nu)], 
\end{equation}
where the function $b$ satisfies the bounds $|\partial_\nu^\ell \partial_w^k b(w, \nu)|\leq C_{k,\ell}w^{\max\{2-k,0\}}$ for, say, all $0<w \leq 2$, $|\nu|\lesssim 1$ and $k,\ell \in \mathbb{N}_0$.
In order to find the asymptotics of $I_{i\nu}(w)$ for $w \to \infty$, one needs some kind of global representation of $I_{i\nu}$ which is provided by the well--known integral formula
\begin{equation} 
\label{eq:besselIint}
I_{i\nu}(w)=\frac{\left (\frac12 w \right)^{i\nu}}{\pi^\frac12 \Gamma(\frac12+i\nu)}\int_{-1}^1 (1-t^2)^{i\nu-\frac12}e^{wt}dt. 
\end{equation}
The classical result for the asymptotics of $I_{i\nu}$ is
\begin{equation} 
\label{eq:besselIasyminf}
I_{i\nu}(w)=(2\pi w)^{-\frac12}\left [e^w(1+O(w^{-1}))-ie^{\nu \pi}e^{-w}(1+O(w^{-1})) \right ] 
\end{equation}
as $w \to \infty$.

\subsection{Asymptotics for small $\nu$}
For our purposes it is important to be able to construct solutions to the modified Bessel equation via Volterra iterations. 
This procedure has the advantage that it automatically yields derivative bounds for the error terms.
Furthermore, it may be considered as a warm--up exercise for the large $\nu$ asymptotics studied 
below and also, for the more complicated problems   we have to deal with in this paper. 

As always,
it is convenient to remove the first derivative in Eq.~\eqref{eq:modbessel} by defining a new dependent variable $B(w):=w^{\frac12}y(w)$ which consequently satisfies
\begin{equation} 
\label{eq:besselno1st}
B''(w)+\left (\frac{\nu^2+\tfrac14}{w^2}-1 \right )B(w)=0.  
\end{equation}
At this point it should be noted that $\sqrt{\cdot}I_{\pm i\nu}$ are solutions to Eq.~\eqref{eq:besselno1st} and from Eq.~\eqref{eq:besselIasym0} we immediately obtain 
$$ W(\sqrt{\cdot}I_{-i\nu}, \sqrt{\cdot}I_{i\nu})=\tfrac{2i}{\pi}\sinh(\pi \nu) $$
which implies that $\{\sqrt{\cdot}I_{-i\nu}, \sqrt{\cdot}I_{i\nu}\}$ is a fundamental system for Eq.~\eqref{eq:besselno1st} provided that $\nu \in \mathbb{R}\backslash \{0\}$.
In particular, this shows that there is no $w_0>0$ such that $I_{-i\nu}(w_0)=I_{i\nu}(w_0)=0$ and, since $\overline{I_{i\nu}}=I_{-i\nu}$, we conclude that $I_{i\nu}(w)$ does not vanish for any $w>0$ and $\nu \in \mathbb{R}$ (the case $\nu=0$ follows directly from Eq.~\eqref{eq:besselIint}).

\subsubsection{The exponential regime}
We proceed by constructing a fundamental system for Eq.~\eqref{eq:besselno1st} with known 
behavior as $w \to \infty$.
This is the exponential region of the equation and we expect to obtain an exponentially decaying and an exponentially increasing solution.
Matching this system to $\sqrt{\cdot}I_{i\nu}$ yields bounds for the derivatives of the $O$--terms in Eq.~\eqref{eq:besselIasyminf}.
In order to construct the desired fundamental system, we rewrite Eq.~\eqref{eq:besselno1st} as 
$$ B''(w)-B(w)=-\frac{\nu^2+\frac14}{w^2}B(w) $$
and treat the right--hand side perturbatively.

\begin{lemma}
\label{lem:FSbessel+sm}
Fix $w_0>0$ sufficiently large and $\nu_0>0$ arbitrary.
Then there exists a fundamental system $\{B_1(\cdot,\nu), B_2(\cdot,\nu)\}$ for Eq.~\eqref{eq:besselno1st} of the form
\begin{align*}
 B_1(w,\nu)&=e^{-w}[1+b_1(w,\nu)] \\
 B_2(w,\nu)&=e^{w}[1+b_2(w,\nu)]
\end{align*}
where the functions $b_j(\cdot,\nu)$, $j=1,2$, are real--valued and satisfy
$$ |\partial_\nu^\ell \partial_w^k b_j(w,\nu)|\leq C_{k,\ell}w^{-1-k} $$
for all $w \geq w_0$, $\nu \in [-\nu_0,\nu_0]$ and $k,\ell \in \mathbb{N}_0$.
\end{lemma}

\begin{proof}
 We start with the decaying solution $B_1$. According to Lemma~\ref{lem:perturb}, we have to construct a solution $b_1(\cdot,\nu)$ to the equation
 $$ b_1(w,\nu)=\int_w^\infty \int_w^v e^{2u}du\;e^{-2v}\left (-\frac{\nu^2+\frac14}{v^2} \right )[1+b_1(v,\nu)]dv $$
 for $w \geq w_0$ and thus, Proposition~\ref{prop:volterra} yields the desired result.
 If $w_0$ is sufficiently large, we obviously have $B_1(w,\nu)>0$ for all $w \geq w_0$ and we obtain $B_2(\cdot,\nu)$ 
 by the standard reduction ansatz, i.e.,
 $$ B_2(w,\nu):=2B_1(w,\nu)\left [\int_{w_0}^w B_1(v,\nu)^{-2}dv+1 \right ], $$
 cf.~the proof of Lemma~\ref{lem:Airy+}.
\end{proof}

\subsubsection{The oscillatory regime}
Next, we deal with the oscillatory regime.
The behavior of solutions in this regime is completely described by Eq.~\eqref{eq:besselIasym0} and
there is no need to employ Volterra iterations since everything follows by Frobenius' method.
As a consequence, we obtain the following.

\begin{corollary}
\label{cor:bessel0sm}
Fix $w_0,\nu_0>0$.
Then there exists a fundamental system $\{B_\pm(\cdot,\nu)\}$ for Eq.~\eqref{eq:besselno1st} of the form
 $$ B_\pm(w,\nu)=w^{\frac12\pm i\nu}[1+b_\pm(w,\nu)] $$
 where the functions $b_\pm(\cdot,\nu)$ satisfy the bounds
 $$ |\partial_\nu^\ell \partial_w^k b_\pm(w,\nu)|\leq C_{k,\ell}w^{\max\{2-k,0\}} $$
 for all $0 \leq w \leq w_0$, $0< \nu \leq \nu_0$ and $k,\ell \in \mathbb{N}_0$.
\end{corollary}

\subsubsection{Matching of the fundamental systems and the classical Bessel asymptotics}

In order to be able to glue together the fundamental systems $\{B_\pm(\cdot,\nu)\}$ and $\{B_j(\cdot,\nu): j=1,2\}$ obtained in Lemma~\ref{lem:FSbessel+sm} and Corollary \ref{cor:bessel0sm}, we have to calculate the Wronskians $W(B_\pm(\cdot,\nu),B_j(\cdot,\nu))$, or, equivalently, the precise asymptotics of $B_\pm(w,\nu)$ as $w \to \infty$. 
However, without any global information, one can only obtain a trivial result here, namely that the Wronskian is a smooth function of $\nu$.
Thus, at this point we have to make use of the global information provided by the classical asymptotics Eq.~\eqref{eq:besselIasyminf} that follow from the integral representation \eqref{eq:besselIint}.
Note that the following result yields the desired symbol character of the $O$--terms in Eq.~\eqref{eq:besselIasyminf}.

\begin{lemma}
 \label{lem:WBB}
 Let $w_0>0$ be sufficiently large and $\nu_0>0$ be arbitrary.
 The Bessel functions $B_\pm(\cdot,\nu)$ from Corollary \ref{cor:bessel0sm} have the representation
\begin{align*} B_\pm(w,\nu)&=2^{\pm i\nu}\Gamma(1\pm i\nu)\sqrt{w}I_{\pm i\nu}(w)\\
 &=\pi^{-\frac12}2^{-\frac12\pm i\nu}
 \Gamma(1\pm i\nu)\left [e^w[1+b_2(w,\pm \nu)]-ie^{\pm \nu \pi}e^{-w}[1+b_1(w,\pm \nu)]\right ] 
\end{align*}
where the error terms $b_j(\cdot,\pm \nu)$, $j=1,2$, are real--valued and satisfy the bounds
$$ |\partial_\nu^\ell \partial_w^k b_j(w,\pm \nu)|\leq C_{k,\ell}w^{-1-k} $$
 for all $w\geq w_0$, $0 < \nu \leq \nu_0$ and $k,\ell \in \mathbb{N}_0$.
 In particular, we have 
$$W(B_\pm(\cdot,\nu),B_1(\cdot,\nu))=-\pi^{-\frac12}2^{\frac12\pm i\nu}
 \Gamma(1\pm i\nu) $$
 and $|B_\pm(w,\nu)|>0$ for all $w>0$ and $0 < \nu \leq \nu_0$.
\end{lemma}

\begin{proof}
 According to Eq.~\eqref{eq:besselIasym0}, we have the asymptotics 
$$ \sqrt{w}I_{\pm i\nu}(w)\sim \frac{2^{\mp i\nu}}{\Gamma(1\pm i\nu)}w^{\frac12\pm i\nu}\quad (w \to 0+) $$
and thus, by comparison with the asymptotics of $B_\pm(\cdot,\nu)$, we obtain
$$B_\pm(w,\nu)=2^{\pm i\nu}\Gamma(1\pm i\nu)\sqrt{w}I_{\pm i\nu}(w). $$
This implies $|B_\pm(w,\nu)|>0$ for all $w>0$.
For notational convenience we set $\tilde{B}_\pm(w,\nu):=(2\pi)^\frac12 \sqrt{w}I_{\pm i\nu}(w)$.
From Eq.~\eqref{eq:besselIasyminf} we have the asymptotics
$$ \tilde{B}_\pm(w,\nu)=e^w[1+O(w^{-1})]-ie^{\pm\nu\pi}e^{-w}[1+O(w^{-1})] $$
as $w \to \infty$.
Since $\tilde{B}_\pm(\cdot,\nu)$ is a solution to Eq.~\eqref{eq:besselno1st} and $\{B_j(\cdot,\nu):j=1,2\}$ 
from Lemma~\ref{lem:FSbessel+sm} 
is a fundamental system for that equation, there must exist connection coefficients $\tilde{\beta}_{\pm,j}(\nu)$ such that
$$ \tilde{B}_\pm(w,\nu)=\tilde{\beta}_{\pm,1}(\nu)B_1(w,\nu)+\tilde{\beta}_{\pm,2}(\nu)B_2(w,\nu). $$
Dividing this equation by $e^{w}$ and letting $w \to \infty$, we obtain $\tilde{\beta}_{\pm,2}(\nu)=1$.
This implies
$$ \mathrm{Im}\tilde{B}_\pm(w,\nu)=\mathrm{Im}\tilde{\beta}_{\pm,1}(\nu)B_1(w,\nu) $$
and multiplication by $e^w$ yields $\mathrm{Im}\tilde{\beta}_{\pm,1}(\nu)=-e^{\pm \nu \pi}$ in the limit $w \to \infty$.
Since $B_\pm(\cdot,\nu)=(2\pi)^{-\frac12}2^{\pm i\nu}\Gamma(1\pm i\nu)\tilde{B}_\pm(\cdot,\nu)$, we infer
$$ \frac{1}{(2\pi)^{-\frac12}2^{\pm i\nu}\Gamma(1\pm i\nu)}\frac{W(B_\pm(\cdot,\nu),B_1(\cdot,\nu))}{W(B_1(\cdot,\nu),B_2(\cdot,\nu))}=
-\tilde{\beta}_{\pm,2}(\nu)=-1 $$
and the claim follows from $W(B_1(\cdot,\nu),B_2(\cdot,\nu))=2$.
\end{proof}

\subsection{Asymptotics for large $\nu$}
For the construction of fundamental systems of Eq.~\eqref{eq:besselno1st} in the case of large $\nu$, a more careful treatment of the turning point is necessary. 
The character of the solutions is dictated by the sign of $(\nu^2+\frac14-w^2)$, i.e., if $\nu^2+\frac14-w^2>0$  solutions tend to oscillate whereas for $\nu^2+\frac14-w^2<0$ we have exponentially increasing or decreasing behavior.
The turning point $w_t$ which separates these two regimes is $w_t(\nu)=\sqrt{\nu^2+\frac14}$ and it moves with $\nu$ and behaves like $w_t(\nu)\sim \nu$ for $\nu \to \infty$.
Consequently, it is reasonable to base the analysis of Bessel's equation on Airy functions.
However, it turns out to be crucial for the asymptotic analysis to fix the turning point first by introducing the rescaled variable $x:=\frac{w}{\nu}$, i.e., we set $\psi(x):=B(\nu x)$ and Eq.~\eqref{eq:besselno1st} transforms into
\begin{equation} 
\label{eq:besselrescaled}
\psi''(x)-\nu^2 (1-\tfrac{1}{x^2})\psi(x)=-\tfrac{1}{4x^2}\psi(x). 
\end{equation}
on $x>0$.
Heuristically speaking, we expect the right--hand side to be negligible provided that $\nu$ is sufficiently large. 
Consequently, we ignore the right--hand side altogether for the moment and observe that 
the turning point for the ``remaining equation'' is now fixed at $x=1$.

\subsubsection{Reduction to a perturbed Airy equation}
The idea is now to define new variables such that Eq.~\eqref{eq:besselrescaled} turns into a perturbed Airy equation.
This is accomplished by means of a Liouville--Green transform.
According to Eq.~\eqref{eq:LG} we require 
\begin{equation} 
\label{eq:LGtrafo}
\zeta'(x)^2 \zeta(x)=1-\tfrac{1}{x^2}  
\end{equation}
which yields
\begin{equation} 
\label{eq:zeta}
\zeta(x)=\sign(x-1)\left |\tfrac32 \int_1^x \sqrt{\left |1-\tfrac{1}{y^2} \right |}dy \right |^\frac23 
\end{equation}
by separation of variables, cf.~Lemma~\ref{lem:zeta}.
Note that $\zeta$ is chosen in such a way that the turning point $x=1$ is mapped to $0$.
By a Taylor expansion around $x=1$ it is readily seen that $\zeta: (0,\infty) \to \mathbb{R}$ is smooth (in fact, real--analytic) and obviously surjective. Furthermore, we have $\zeta'>0$ and this shows that $\zeta$ maps $(0,\infty)$ to $\mathbb{R}$ diffeomorphically and thus, $\zeta$ defines an admissible change of variables as desired.
Consequently, Eq.~\eqref{eq:besselrescaled} transforms into
\begin{equation} 
\label{eq:Airypert}
\underbrace{\phi''(\zeta)=\nu^2 \zeta \phi(\zeta)}_{\mbox{Airy}}+\underbrace{V_2(\zeta)\phi(\zeta)}_{\mbox{pert.}} 
\end{equation}
for $\zeta \in \mathbb{R}$ with $\phi(\zeta(x)):=\zeta'(x)^\frac12 \psi(x)$ and, cf.~Eq.~\eqref{eq:LG},
$$ V_2(\zeta(x))=-\frac{3}{4}\frac{\zeta''(x)^2}{\zeta'(x)^4}+\frac{1}{2}\frac{\zeta'''(x)}{\zeta'(x)^3}-\frac{1}{4x^2 \zeta'(x)^2}=\frac{5}{16 \zeta(x)^2}+\frac{\zeta(x)x^2(4+x^2)}{4(1-x^2)^3} $$
where the last expression follows straightforwardly from Eq.~\eqref{eq:LGtrafo}.
Note that the apparent singularity at $x=1$ cancels by a Taylor expansion and the fact that $\zeta(1)=0$. Consequently, we have $V_2 \in C^\infty(\mathbb{R})$ and
it is evident from the above expression and Appendix~\ref{sec:symbol} that $V_2$ satisfies the bounds $|V_2^{(k)}(\zeta)|\leq C_k \langle \zeta \rangle^{-2-k}$ for all $k \in \mathbb{N}_0$ and all $\zeta \in \mathbb{R}$.

\subsubsection{The exponential regime}
The Airy equation 
$$ \phi''(\zeta)=\nu^2 \zeta \phi(\zeta) $$
has the fundamental system $\{\Ai(\nu^\frac23 \cdot), \Bi(\nu^\frac23 \cdot)\}$ where $\Ai$, $\Bi$ are the standard Airy functions, see Appendix~\ref{sec:airy}. 
This immediately allows us to construct a fundamental system for Eq.~\eqref{eq:Airypert} in the exponential regime. 
\begin{lemma}
\label{lem:bessel+}
 For $\zeta \geq 0$, $\nu \gg 1$ there exists a fundamental system $\{\phi_1(\cdot,\nu),\phi_2(\cdot,\nu)\}$ of 
Eq.~\eqref{eq:Airypert} of the form
 \begin{align*}
  \phi_1(\zeta, \nu)&=\Ai(\nu^\frac23 \zeta)[1+\nu^{-1}a_1(\zeta, \nu)] \\
  \phi_2(\zeta, \nu)&=\Bi(\nu^\frac23 \zeta)[1+\nu^{-1}a_2(\zeta, \nu)] 
 \end{align*}
where the functions $a_j$, $j=1,2$, are smooth, real--valued and $|a_j(\zeta,\nu)|\lesssim 1$ in 
the above range of $\zeta$ and $\nu$. Furthermore, $a_j$ satisfy the bounds
$$ |\partial_\nu^\ell \partial_\zeta^k a_j(\zeta, \nu)|\leq C_{k,\ell} \langle \zeta \rangle^{-\frac32-k}\nu^{-\ell},\quad \zeta \geq 1 $$
as well as
$$ |\partial_\nu^\ell a_j(0,\nu)|\leq C_\ell \nu^{-\ell},\quad |\partial_\nu^\ell \partial_\zeta a_j(0,\nu)|\leq C_\ell \nu^{\frac23-\ell} $$
for all $k,\ell \in \mathbb{N}_0$ and $\nu\gg 1$.
\end{lemma}

\begin{proof}
We start with $\phi_1$ and according to Lemma~\ref{lem:perturb} we have to construct a solution $a_1$ of
 \begin{align*}
\nu^{-1} a_1(\zeta, \nu)&=\int_\zeta^\infty \int_\zeta^\eta \Ai(\nu^\frac23 \eta')^{-2}d\eta' \Ai (\nu^\frac23 \eta)^2V_2(\eta)[1+\nu^{-1}a_1(\eta, \nu)]d\eta \\
&=\nu^{-\frac43}\int_{\nu^\frac23 \zeta}^\infty \int_{\nu^\frac23 \zeta}^y \Ai(u)^{-2} du\;\Ai(y)^2 V_2(\nu^{-\frac23}y)[1+\nu^{-1}a_1(\nu^{-\frac23} y, \nu)]dy
 \end{align*}
 for $\zeta \geq 0$.
 Thus, the function $\tilde{a}_1(x, \nu):=\nu^{-1}a_1(\nu^{-\frac23}x)$ satisfies the Volterra equation
 $$ \tilde{a}_1(x, \nu)=\int_x^\infty K(x,y, \nu)[1+\tilde{a}_1(y, \nu)]dy $$
 where
 $$ K(x,y, \nu):=\int_x^y \Ai(u)^{-2}du\;\Ai(y)^2 \nu^{-\frac43}V_2(\nu^{-\frac23}y). $$
 In view of Proposition~\ref{prop:volterraAi} we set $\lambda:=\nu^{-1}$ and, since we are interested in large $\nu$, we assume $\lambda \in (0,1)$.
 If $x \geq \lambda^{-\frac23} \geq 1$, we have $x=\langle x \rangle$ and hence, $\langle \lambda^{\frac23}x\rangle^{-2}=\lambda^{-\frac43}x^{-2}$ which yields the bounds
$$ |\partial_\lambda^\ell \partial_x^k \lambda^{\frac43}V_2(\lambda^{\frac23} x)|
\leq C_{k,\ell}\langle x \rangle^{-2-k}\lambda^{-\ell},\quad k,\ell \in \mathbb{N}_0 $$
for all $x\geq \lambda^{-\frac23}$ and $\lambda \in (0,1)$.
Consequently, Proposition~\ref{prop:volterraAi} yields the existence of $\tilde{a}_1$ with the bounds
 $$ |\partial_\lambda^\ell \partial_x^k \tilde{a}_1(x, \lambda^{-1})|\leq C_{k,\ell}\langle x \rangle^{-\frac32-k}\lambda^{-\ell} $$
 and, via Lemma~\ref{lem:symbc}, these translate into
 $$ |\partial_\nu^\ell \partial_x^k \tilde{a}_1(x, \nu)|\leq C_{k,\ell}\langle x \rangle^{-\frac32-k}\nu^{-\ell} $$
 for all $x \geq \nu^\frac23$, $\nu > 1$ and $k,\ell \in \mathbb{N}_0$.
We obtain $|a_1(\zeta, \nu)|=|\nu \tilde{a}_1(\nu^{\frac23}\zeta, \nu)|\lesssim \nu \langle \nu^\frac23 \zeta \rangle^{-\frac32}$ and
 $$ |\partial_\nu^\ell \partial_\zeta^k a_1(\zeta, \nu)|\leq C_{k,\ell}\langle \zeta \rangle^{-\frac32-k}\nu^{-\ell} $$
 for all $\zeta \geq 1$, $\nu > 1$ and $k,\ell \in \mathbb{N}_0$ as claimed.
In the case $0 \leq x \leq \lambda^{-\frac23}$ (which implies $\langle x \rangle \leq \lambda^{-\frac23}$) we have
\begin{align*}
\tilde{a}_1(x,\lambda^{-1})=\tilde{a}_1(\lambda^{-\frac23},\lambda^{-1})+\int_x^{\lambda^{-\frac23}}K(x,y,\lambda^{-1})[1+\tilde{a}_1(y,\lambda^{-1})]dy
\end{align*}
and from the above we obtain the bounds 
$|\partial_\lambda^\ell \tilde{a}_1(\lambda^{-\frac23},\lambda^{-1})|\leq C_\ell \lambda^{1-\ell}$ for $\ell \in \mathbb{N}_0$.
Furthermore, we have
$$ |\partial_\lambda^\ell \partial_x^k \lambda^{\frac43}V_2(\lambda^\frac23 x)|\leq C_{k,\ell}\langle x \rangle^{-k}\lambda^{\frac43-\ell} $$
for all $0\leq x \leq \lambda^{-\frac23}$, $\lambda \in (0,1)$ and $k,\ell \in \mathbb{N}_0$ and, similar to 
Proposition~\ref{prop:volterraAi} (cf.~also Proposition~\ref{prop:volterraAinox}), we infer inductively by Volterra iterations that 
$$|\partial_\lambda^\ell \partial_x^k \tilde{a}_1(x,\lambda^{-1})\leq C_{k,\ell}\langle x \rangle^{-k}\lambda^{1-\ell}$$
which implies the claimed bounds for $a_1(0,\nu)$ and $\partial_\zeta a_1(0,\nu)$.

For the function $\phi_2$ we apply the standard reduction ansatz, i.e., we use the fact that 
$\phi_1(\zeta, \nu)>0$ for all $\zeta \geq 0$ provided that $\nu$ is large enough (hence the condition $\nu \gg 1$) and set
$$ \phi_2(\zeta, \nu)=\tfrac{1}{\pi}\nu^{\frac23} \phi_1(\zeta, \nu) \int_0^\zeta \phi_1(\eta, \nu)^{-2}d\eta
+\phi_1(\zeta,\nu). $$
By recalling that $W(\Ai,\Bi)=\frac{1}{\pi}$ 
it is readily verified that $\phi_2$ is indeed of the stated form. 
\end{proof}

\subsubsection{The oscillatory regime}
For $\zeta \leq 0$ we cannot apply the same construction as in Lemma~\ref{lem:bessel+} due to the zeros of $\Ai$ and $\Bi$ on the negative real axis. Consequently, we construct a complex 
fundamental system based on $\Ai\pm i\Bi$.

\begin{lemma}
 \label{lem:bessel-}
 For $\zeta \leq 0$, $\nu \geq 1$ there exists a fundamental system $\{\phi_\pm(\cdot, \nu)\}$ of Eq.~\eqref{eq:Airypert} of the form
 $$ \phi_\pm(\zeta, \nu)=[\Ai(\nu^\frac23 \zeta)\pm i\Bi(\nu^\frac23 \zeta)][1+\nu^{-1}a_\pm(\zeta, \nu)] $$
 where the functions $a_\pm(\cdot, \nu)$ are smooth and $|a_\pm(\zeta,\nu)|\lesssim 1$ in the above range 
 of $\zeta$ and $\nu$.
 Furthermore, $a_\pm$ satisfy the bounds
 $$ |\partial_\nu^\ell \partial_\zeta^k a_\pm(\zeta, \nu)|\leq C_{k,\ell}\langle \zeta\rangle^{-\frac32-k}\nu^{-\ell},\quad \zeta \leq -1 $$
 as well as
 $$ |\partial_\nu^\ell a_\pm(0,\nu)|\leq C_\ell \nu^{-\ell},\quad 
 |\partial_\nu^\ell \partial_\zeta a_\pm(0,\nu)|\leq C_{k,\ell}\nu^{\frac23-\ell} $$
 for all $\nu \geq 1$ and $k,\ell \in \mathbb{N}_0$.
\end{lemma}

\begin{proof}
After switching from $\zeta$ to $-\zeta$, the proof becomes completely analogous to the proof of Lemma~\ref{lem:bessel+}. The only difference being that one has to apply Proposition~\ref{prop:volterraAiBi} instead of Proposition~\ref{prop:volterraAi}. Furthermore, the reduction ansatz is not needed: once $\phi_-$ is constructed, one sets $\phi_+:=\overline{\phi_-}$.
\end{proof}

\subsubsection{Matching of the fundamental systems}

Finally, we match the fundamental systems $\{\phi_j(\cdot, \nu): j=1,2\}$ and $\{\phi_\pm(\cdot,\nu)\}$ from 
Lemmas \ref{lem:bessel+} and \ref{lem:bessel-}.
This is done by evaluating their Wronskians at $\zeta=0$.
Unlike in the small $\nu$ case, we do not use any global information on the Bessel functions. 
This is not necessary here since we have a small parameter $\nu^{-1}$ and the Airy functions are defined globally.

\begin{lemma}
\label{lem:bessel+bessel-}
 For $\zeta \geq 0$, the functions $\phi_\pm(\cdot, \nu)$ from Lemma~\ref{lem:bessel-} have the representation
 $$ \phi_\pm(\zeta, \nu)=[1+\nu^{-1}\alpha_{\pm,1}(\nu)]\phi_1(\zeta, \nu)\pm i[1+\nu^{-1}\alpha_{\pm,2}(\nu)]\phi_2(\zeta, \nu) $$
 where $\phi_j$, $j=1,2$, are from Lemma~\ref{lem:bessel+} and the coefficients $\alpha_{\pm,j}$ satisfy the bounds
 $$ |\partial_\nu^\ell \alpha_{\pm,j}(\nu)|\leq C_\ell \nu^{-\ell} $$
 for all $\nu \gg 1$ and $\ell \in \mathbb{N}_0$, $j=1,2$.
\end{lemma}

\begin{proof}
The ansatz
$$ \phi_\pm(\zeta, \nu)=\tilde{\alpha}_{\pm,1}(\nu)\phi_1(\zeta, \nu)+\tilde{\alpha}_{\pm,2}(\nu)\phi_2(\zeta, \nu) $$
implies
$$ \tilde{\alpha}_{\pm,1}(\nu)=\frac{W(\phi_\pm(\cdot, \nu),\phi_2(\cdot, \nu))}{W(\phi_1(\cdot, \nu), \phi_2(\cdot, \nu))},\quad
\tilde{\alpha}_{\pm,2}(\nu)=\frac{W(\phi_\pm(\cdot, \nu),\phi_1(\cdot, \nu))}{W(\phi_2(\cdot, \nu), \phi_1(\cdot, \nu))}. $$
Note that the Wronskians are functions of $\nu$ only since Eq.~\eqref{eq:besselrescaled} does not have a first derivative.
 Setting $A_1:=\Ai$, $A_2:=\Bi$ and $A_\pm:=\Ai\pm i\Bi$, we obtain, for $j=1,2$,
 \begin{align*} 
\phi_j(0, \nu)&=A_j(0)[1+O(\nu^{-1})]& \phi'_j(0, \nu)&=\nu^\frac23 A_j'(0)[1+O(\nu^{-1})] \\
\phi_\pm(0, \nu)&=A_\pm(0)[1+O_\mathbb{C}(\nu^{-1})] &\phi'_\pm(0, \nu)&=\nu^\frac23 A_\pm'(0)[1+O_\mathbb{C}(\nu^{-1})]
 \end{align*}
 where all $O$--terms behave like symbols.
 This yields
 $$
  \tilde{\alpha}_{\pm,1}(\nu)=\frac{\nu^\frac23 W(A_1\pm iA_2,A_2)[1+O_\mathbb{C}(\nu^{-1})]}{\nu^\frac23 W(A_1,A_2)[1+O(\nu^{-1})]}=1+O_\mathbb{C}(\nu^{-1})
 $$
 provided that $\nu$ is sufficiently large, and, analogously,
 $$
 \tilde{\alpha}_{\pm,2}(\nu)=\frac{\nu^\frac23 W(A_1\pm iA_2,A_1)[1+O_\mathbb{C}(\nu^{-1})]}{\nu^\frac23 W(A_2,A_1)[1+O(\nu^{-1})]}=\pm i+O_\mathbb{C}(\nu^{-1}) $$
 where the $O_\mathbb{C}$--terms behave like symbols.
\end{proof}

\end{appendix}

\bibliography{semexp}
\bibliographystyle{plain}

\end{document}